\documentclass{dis}
\usepackage{amsmath}
\usepackage{amssymb}

\newtheorem{theorem}{Theorem}[chapter]
\newtheorem{proposition}[theorem]{Proposition}
\newtheorem{lemma}[theorem]{Lemma}
\newtheorem{corollary}[theorem]{Corollary}

\theoremstyle{definition}
\newtheorem{definition}[theorem]{Definition}
\newtheorem{example}[theorem]{Example}
\newtheorem{remark}[theorem]{Remark}

\chaptersnewpage

\refnewpage

\newcommand{\eq}[1]{\mbox{\rm(\ref{#1})}}
\newcommand{\Rb}{\mathbb{R}}  \newcommand{\Nb}{\mathbb{N}}
\newcommand{\Qb}{\mathbb{Q}} \newcommand{\Dc}{\mathcal{D}}
\newcommand{\Cb}{\mathbb{C}}  \newcommand{\Kb}{\mathbb{K}}
\newcommand{\es}{\varnothing}   \newcommand{\vep}{\varepsilon}
\newcommand{\al}{\alpha}           \newcommand{\vfi}{\varphi}
\newcommand{\BV}{\mbox{\rm BV}}  \newcommand{\Reg}{\mbox{\rm Reg}}
\newcommand{\Bd}{\mbox{\rm B}}    \newcommand{\Mon}{\mbox{\rm Mon}}
\newcommand{\lan}{\langle}         \newcommand{\ran}{\rangle}
\newcommand{\rra}{\rightrightarrows}
\newcommand{\pw}{pointwise}      \newcommand{\rc}{relatively compact}
\newcommand{\D}{\displaystyle}
\newcommand{\wto}{\stackrel{\mbox{\scriptsize\rm \!w\,}}{\to}} % weak convergence
\newcommand{\ov}[1]{\overline{#1}}   \newcommand{\wt}[1]{\widetilde{#1}}

\begin{document}


\mathclass{Primary 26A45, 40A30;  Secondary 54E35, 26A48.}

\thanks{The article was prepared within the framework of the Academic Fund Program
at the National Research University Higher School of Economics (HSE) in 2017--2018
(grant~no.\,17-01-0050) and by the Russian Academic Excellence Project ``5--100''.}

\abbrevauthors{V.\,V.~Chistyakov}
\abbrevtitle{Approximate variation}

\title{The approximate variation to pointwise selection principles}

\author{Vyacheslav V.~Chistyakov}
\address{Department of Informatics, Mathematics and Computer Science\\
National Research University Higher School of Economics\\
Bol'shaya Pech{\"e}rskaya Street 25/12\\
Nizhny Novgorod 603155, Russian Federation\\
E-mail: vchistyakov@hse.ru, czeslaw@mail.ru}

\maketitledis

\tableofcontents

\begin{abstract}
Let $T\subset\Rb$, $M$ be a metric space with metric $d$, and $M^T$ be the set of all
functions mapping $T$ into $M$. Given $f\in M^T$, we study the properties of the
\emph{approximate variation\/} $\{V_\vep(f)\}_{\vep>0}$, where $V_\vep(f)$ is the
greatest lower bound of Jordan variations $V(g)$ of functions $g\in M^T$ such that
$d(f(t),g(t))\le\vep$ for all $t\in T$. The notion of $\vep$-variation $V_\vep(f)$ was
introduced by Fra{\v n}kov{\'a} [Math.\ Bohem.\ 116 (1991), 20--59] for intervals
$T=[a,b]$ in $\Rb$ and $M=\Rb^N$ and extended to the general case by Chistyakov
and Chistyakova [Studia Math.\ 238 (2017), 37--57]. We prove directly the following basic
pointwise selection principle: {\it If a sequence of functions $\{f_j\}_{j=1}^\infty$ from
$M^T$ is such that the closure in $M$ of the set $\{f_j(t):j\in\Nb\}$ is compact for all
$t\in T$ and\/ $\limsup_{j\to\infty}V_\vep(f_j)$ is finite for all\/ $\vep>0$, then it contains
a subsequence, which converges pointwise on $T$ to a bounded regulated function
$f\in M^T$.} We establish several variants of this result for sequences of regulated and
nonregulated functions, for functions with values in reflexive separable Banach spaces,
for the almost everywhere convergence and weak pointwise convergence of extracted
subsequences, and comment on the necessity of assumptions in the selection principles.
The sharpness of all assertions is illustrated by examples.
\end{abstract}
\makeabstract

\chapter{Introduction} \label{s:intro}

A pointwise selection principle is a statement which asserts that under certain specified
assumptions on a given sequence of functions $f_j:T\to M$ ($j\in\Nb$), their domain $T$
and range $M$, the sequence admits a subsequence converging in (the topology of)
$M$ pointwise (=everywhere) on the set $T$; in other words, this is a compactness
theorem in the topology of pointwise convergence. Our intention here is twofold: first,
to draw attention to a conjunction of pointwise selection principles and characterizations of
regulated functions (cf.\ also \cite{waterman80}) and, second, to exhibit the main goal
of this paper.

To be specific, we let $T=I=[a,b]$ be a closed interval in $\Rb$ and $M=\Rb$ and
denote by: $\Rb^I$ the set of \emph{all\/} functions mapping $I$ into $\Rb$, $\Mon(I)$
the set of \emph{monotone\/} functions, $\BV(I)$ the set of functions of
\emph{bounded\/} (Jordan) \emph{variation}, and $\Reg(I)$ the set of
\emph{regulated\/} functions from $\Rb^I$. Recall that $f\in\Rb^I$ is regulated
provided the left limit $f(t-0)\in\Rb$ exists at each point $a<t\le b$ and the right limit
$f(t+0)\in\Rb$ exists at each point $a\le t<b$. Clearly,
$\Mon(I)\subset\BV(I)\subset\Reg(I)$, and it is well known that each function from
$\Reg(I)$ is bounded, has a finite or countable set of discontinuity points, and is the
uniform limit of a sequence of step functions on~$I$. Scalar- (and vector-) valued
regulated functions are of importance in various branches of analysis, e.g., the theory
of convergence of Fourier series, stochastic processes, Riemann- and Lebesgue-Stieltjes
integrals, generalized ordinary differential equations, impulse controls, modular analysis
(\cite{Aumann}, \cite{MMS}, \cite{Dieu}, \cite{GNW}, \cite{Hild}, \cite{Loja}, \cite{Nat},
\cite{Rao}, \cite{Rudin}, \cite{Saks}, \cite{Schwabik}, \cite{Schwartz}).

In order for a sequence of functions $\{f_j\}\subset\Rb^I$ to have a pointwise convergent
subsequence, it is quite natural, by virtue of Bolzano-Weierstrass' theorem (viz., a bounded
sequence in $\Rb$ admits a convergent subsequence), that $\{f_j\}$ should be
\emph{pointwise bounded\/} (i.e., $\sup_{j\in\Nb}|f_j(t)|<\infty$ for all $t\in I$).
However, a pointwise (or even uniformly) bounded sequence $\{f_j\}\subset\Rb^I$
need not have a pointwise convergent subsequence: a traditional example is the sequence
$f_j(t)=\sin(jt)$ for $j\in\Nb$ and $t\in I=[0,2\pi]$ (see Remark~\ref{r:sinjt} below).
So, additional assumptions on $\{f_j\}$ are to be imposed.

The historically first pointwise selection principles are due to Helly \cite{Helly}:
\emph{a uniformly bounded sequence $\{f_j\}\subset\Mon(I)$ contains a pointwise
convergent subsequence\/} (whose pointwise limit belongs to $\Mon(I)$). This theorem,
a selection principle for monotone functions, is based on and extends Bolzano-%
Weierstrass' theorem and implies one more Helly's selection principle for functions of
bounded variation (in \eq{e:bVfj} below, $V(f)$ denotes the Jordan variation of
$f\in\Rb^I$): \emph{a pointwise bounded sequence $\{f_j\}\subset\Rb^I$ satisfying
  \begin{equation} \label{e:bVfj}
\sup_{j\in\Nb}V(f_j)<\infty
  \end{equation}
contains a pointwise convergent subsequence\/} (with the pointwise limit from
$\BV(I)$). Note that condition \eq{e:bVfj} of uniform boundedness of variations may be
replaced by a (seemingly) more general condition $\limsup_{j\to\infty}V(f_j)<\infty$.

It is well known that Helly's selection principles play a significant role in analysis (e.g., 
\cite{Hild}, \cite{Nat}, \cite{Saks}). 
A vast literature already exists concerning generalizations of
Helly's principles for various classes of functions (\cite{Barbu}--\cite{JMS},
\cite{JDCS}--\cite{MatTr}, \cite{MMS}, \cite{Studia17}--\cite{waterman80},
\cite{Dudley}--\cite{Gnilka}, \cite{Megre}, \cite{MuOr}, \cite{Schramm}, 
\cite{IzVUZ}--\cite{Wat76}, and references therein) and their applications
(\cite{Barbu}, \cite{Sovae}, \cite{MMS}, \cite{JMAA2019}, \cite{JFA05}--\cite{Studia02},
\cite{Dudley}, \cite{GNW}, \cite{Hermes}, \cite{Schwabik}).
We recall some of these generalizations, which are relevant for our purposes.

Let $\varphi:[0,\infty)\to[0,\infty)$ be a nondecreasing continuous function such that
$\varphi(0)=0$, $\varphi(u)>0$ for $u>0$, and $\varphi(u)\to\infty$ as $u\to\infty$.
We say that $f\in\Rb^I$ is \emph{of bounded $\varphi$-variation\/} on~$I$ (in the sense
of Wiener and Young) and write $f\in\BV_{\!\varphi}(I)$ if the following quantity, called the
\emph{$\varphi$-variation\/} of $f$, is finite:
  \begin{equation*}
V_\varphi(f)=\sup\,\biggl\{\sum_{i=1}^n\varphi\bigl(|f(I_i)|\bigr):
\mbox{$n\in\Nb$ and $\{I_i\}_1^n\prec I$}\biggr\},
  \end{equation*}
where the notation $\{I_i\}_1^n\prec I$ stands for a non-ordered collection of $n$
non-overlapping intervals $I_i=[a_i,b_i]\subset I$ and $|f(I_i)|=|f(b_i)-f(a_i)|$,
$i=1,\dots,n$. \label{p:nco}
(In particular, if $\varphi(u)\!=\!u$, we have $V_\varphi(f)=V(f)$.)
It was shown by Musielak and Orlicz \cite{MuOr} that $\BV_{\!\varphi}(I)\subset\Reg(I)$,
and if $\varphi$ is additionally convex and $\varphi'(0)\equiv\lim_{u\to+0}\varphi(u)/u=0$,
then $\BV(I)$ is a proper subset of $\BV_{\!\varphi}(I)$. Goff\-man, Moran and
Waterman \cite{GMW} characterized the set $\Reg(I)$ as follows: if $f\!\in\!\Reg(I)$ and
\mbox{$\min\{f(t-0),f(t+0)\}\!\le\! f(t)\!\le\!\max\{f(t-0),f(t+0)\}$} at each point
$t\in I$ of discontinuity of $f$, then there is a convex function $\varphi$ (as above) with
$\varphi'(0)=0$ such that $f\in\BV_{\!\varphi}(I)$. A generalization of Helly's theorem
for BV functions, the so called \emph{Helly-type selection principle}, was established in
\cite{MuOr}, where condition \eq{e:bVfj} was replaced by
$\sup_{j\in\Nb}V_\varphi(f_j)<\infty$.

One more Helly-type selection principle is due to Waterman \cite{Wat76}, who replaced
condition \eq{e:bVfj} by $\sup_{j\in\Nb}V_\Lambda(f_j)<\infty$, where
$V_\Lambda(f)$ is the Waterman \emph{$\Lambda$-variation\/} of $f\in\Rb^I$
defined by (\cite{Wat72})
  \begin{equation*}
V_\Lambda(f)=\sup\,\biggl\{\sum_{i=1}^n\frac{|f(I_i)|}{\lambda_i}:
\mbox{$n\in\Nb$ and $\{I_i\}_1^n\prec I$}\biggr\};
  \end{equation*}
here $\Lambda=\{\lambda_i\}_{i=1}^\infty$ is a \emph{Waterman sequence}, i.e.,
$\Lambda\subset(0,\infty)$ is nondecreasing, unbounded and
$\sum_{i=1}^\infty1/\lambda_i=\infty$. (Formally, $V_\Lambda(f)=V(f)$ for
$\lambda_i=1$,~$i\in\Nb$.) For the set $\Lambda\BV(I)=\{f\in\Rb^I:
V_\Lambda(f)<\infty\}$ of functions \emph{of $\Lambda$-bounded variation},
Waterman \cite{Wat72} showed that $\Lambda\BV(I)\subset\Reg(I)$ and $\BV(I)$ is a
proper subset of $\Lambda\BV(I)$. Perlman \cite{Perl} proved that
$\BV(I)=\bigcap_\Lambda\Lambda\BV(I)$ and obtained the following characterization of
regulated functions: $\Reg(I)=\bigcup_\Lambda\Lambda\BV(I)$, where the intersection
above and the union are taken over all Waterman sequences~$\Lambda$
(but not over any countable collection). 

Taking into account that the sets $\Mon(I)$, $\BV(I)$, $\BV_{\!\varphi}(I)$, and
$\Lambda\BV(I)$ are contained in $\Reg(I)$, Helly's selection principles and
their generalizations alluded to above are compactness theorems in the class of
regulated functions.

In the literature, there are characterizations of the set $\Reg(I)$, which do not rely on
notions of bounded (or generalized bounded) variations of any kind. One of them was
given by Chanturiya (\cite{Chan74}, \cite{Chan75}) in the form
$\Reg(I)=\{f\in\Rb^I:\nu_n(f)=o(n)\}$, where E.~Landau's small `$o$' means, as usual,
that $o(n)/n\to0$ as $n\to\infty$, and the sequence $\{\nu_n(f)\}_{n=1}^\infty%
\subset[0,\infty]$, called the \emph{modulus of variation\/} of $f$, is defined by
(\cite{Chan74}, cf.\ also \cite[Section 11.3.7]{GNW})
  \begin{equation*}
\nu_n(f)=\sup\,\biggl\{\sum_{i=1}^n|f(I_i)|:\{I_i\}_1^n\prec I\biggr\},\quad n\in\Nb.
  \end{equation*}
Note that $\nu_n(f)\le V(f)$ for all $n\in\Nb$ and $\nu_n(f)\to V(f)$ as $n\to\infty$.
The author (\cite{JMAA05}, \cite{Ischia}) replaced condition \eq{e:bVfj} by
(a very weak one)
  \begin{equation} \label{e:nunfj}
\limsup_{j\to\infty}\nu_n(f_j)=o(n)
  \end{equation}
and obtained a Helly-type pointwise selection principle (in which the pointwise limit of the
extracted subsequence of $\{f_j\}$ belongs to $\Reg(I)$ and) which contains, as
particular cases, all the above Helly-type selection principles and many others
(\cite{DAN06}, \cite{MatTr}, \cite{Studia17}). Assumption \eq{e:nunfj} is applicable
to sequences of nonregulated functions, so the corresponding Helly-type pointwise
selection principle under \eq{e:nunfj} is already outside the scope of regulated functions.
To see this, let $\Dc\in\Rb^I$ be the Dirichlet function on $I=[0,1]$ (i.e., $\Dc(t)=1$ if
$t\in I$ is rational, and $\Dc(t)=0$ otherwise) and $f_j(t)=\Dc(t)/j$ for $j\in\Nb$ and
$t\in I$. We have $f_j\notin\Reg(I)$ and $\nu_n(f_j)=n/j$ for all $j,n\in\Nb$, and so,
\eq{e:nunfj} is satisfied while \eq{e:bVfj} is not (for any kinds of generalized variations
including $V_\vfi$ and $V_\Lambda$).
A special feature of condition \eq{e:nunfj} is that, for $f\in\Reg(I)$, it is
\emph{necessary\/} for the uniform convergence of $\{f_j\}$ to $f$, and
`almost necessary' for the pointwise convergence of $\{f_j\}$ to~$f$---note that this is
not at all the case for (uniform) conditions of the form~\eq{e:bVfj}.

Dudley and Norvai{\v s}a \cite[Part~III, Section~2]{Dudley} presented the following
characterization of regulated functions:
\mbox{$\Reg(I)=\{f\in\Rb^I:N_\vep(f)\!<\!\infty\,\,\forall\,\vep\!>\!0\}$},
where the (untitled) quantity $N_\vep(f)\in\{0\}\cup\Nb\cup\{\infty\}$ for $f\in\Rb^I$
is given by
  \begin{equation*}
N_\vep(f)=\sup\,\Bigl\{n\in\Nb:\mbox{$\exists\,\{I_i\}_1^n\prec I$ such that
$\displaystyle\min_{1\le i\le n}|f(I_i)|>\vep$}\Bigr\},\quad\vep>0
  \end{equation*}
(with $\sup\varnothing=0$). They established a Helly-type pointwise selection principle
in the class $\Reg(I)$ by replacing \eq{e:bVfj} with $\sup_{j\in\Nb}N_\vep(f_j)<\infty$
for all \mbox{$\vep>0$}. In a series of papers by the author, Maniscalco and Tretyachenko
(\cite[Chap\-ter~5]{MMS}, \cite{waterman80}, \cite{IzVUZ}, \cite{MZ08}), it was shown
that we get a more powerful selection principle (outside the scope of regulated functions)
if \eq{e:bVfj} is replaced by
  \begin{equation} \label{e:Nefj}
\limsup_{j\to\infty}N_\vep(f_j)<\infty\quad\,\,\mbox{for \,all}\quad\,\,\vep>0.
  \end{equation}
If we let the sequence of nonregulated functions $f_j(t)=\Dc(t)/j$ be as above, we find
$N_\vep(f_j)=\infty$ if $j<1/\vep$ and $N_\vep(f_j)=0$ if $j\ge1/\vep$, and so, condition
\eq{e:Nefj} is satisfied. Moreover, \eq{e:Nefj} is \emph{necessary\/} for the uniform
convergence and `almost necessary' for the pointwise convergence of $\{f_j\}$ to
$f\in\Reg(I)$. A comparison of different Helly-type pointwise selection principles is presented
in \cite{JMAA05}--\cite{MatTr}, \cite{MMS}, \cite{Studia17}, \cite{Manisc}.

Essential for the present paper, one more characterization of regulated functions is due to
Fra{\v n}kov{\'a} \cite{Fr}:
\mbox{$\Reg(I)=\{f\in\Rb^I:V_\vep(f)<\infty\,\,\forall\,\vep>0\}$}, where the
\emph{$\vep$-variation\/} $V_\vep(f)$ of $f\in\Rb^I$ is defined by
(\cite[Definition~3.2]{Fr})
  \begin{equation*}
V_\vep(f)=\inf\,\bigl\{V(g):\mbox{$g\in\BV(I)$ and $|f(t)-g(t)|\le\vep$
$\forall\,t\in I$}\bigl\},\quad\vep>0
  \end{equation*}
(with $\inf\varnothing=\infty$). She established a Helly-type selection principle in the class
$\Reg(I)$ under the assumption of uniform boundedness of $\vep$-variations
$\sup_{j\in\Nb}V_\vep(f_j)\!<\!\infty$ for all $\vep\!>\!0$
in place of \eq{e:bVfj}. However, following the `philosophy' of \eq{e:nunfj} and
\eq{e:Nefj}, a weaker condition, replacing \eq{e:bVfj}, is of the~form
  \begin{equation} \label{e:Vee}
\limsup_{j\to\infty}V_\vep(f_j)<\infty\quad\,\,\mbox{for \,all}\quad\,\,\vep>0.
  \end{equation}
Making use of \eq{e:Vee}, the author and Chistyakova \cite{Studia17} proved a
Helly-type pointwise selection principle outside the scope of regulated functions by
showing that \eq{e:Vee} implies \eq{e:nunfj}. If the sequence $f_j(t)=\Dc(t)/j$ is as
above, we get $V_\vep(f_j)=\infty$ if $j<1/(2\vep)$ and $V_\vep(f_j)=0$ if
$j\ge1/(2\vep)$, and so, \eq{e:Vee} is fulfilled while the uniform $\vep$-variations
are unbounded for $0<\vep<1/2$.

In this paper, we present a direct proof of a Helly-type pointwise selection principle
under \eq{e:Vee}, not relying on \eq{e:nunfj}, and show that condition \eq{e:Vee} is
necessary for the uniform convergence and `almost necessary' for the pointwise
convergence of $\{f_j\}$ to $f\in\Reg(I)$ (cf. Remark~\ref{r:neces} below).

All the above pointwise selection principles are based on the Helly selection theorem for
monotone functions. A different kind of a pointwise selection principle, basing on
Ramsey's theorem from formal logic \cite{Ramsey}, was given by Schrader~%
\cite{Schrader}. In order to recall it, we introduce a notation: given a sign-changing
function $f\in\Rb^I$, we denote by $\mathcal{P}(f)$ the set of all finite collections
of points $\{t_i\}_{i=1}^n\subset I$ with $n\in\Nb$ such that $t_1<t_2<\dots<t_n$
and either $(-1)^if(t_i)>0$ for all $i=1,\dots,n$, or $(-1)^if(t_i)<0$ for all $i=1,\dots,n$,
or $(-1)^if(t_i)=0$ for all $i=1,\dots,n$. The quantity
  \begin{equation*}
\mathcal{T}(f)=\sup\,\biggl\{\sum_{i=1}^n|f(t_i)|:\mbox{$n\in\Nb$ and
$\{t_i\}_{i=1}^n\in\mathcal{P}(f)$}\biggr\}
  \end{equation*}
is said to be Schrader's \emph{oscillation\/} of $f$ on $I$; if $f$ is nonnegative on $I$
or $f$ is nonpositive on $I$, we set $\mathcal{T}(f)=\sup_{t\in I}|f(t)|$. Schrader
proved that \emph{if $\{f_j\}\subset\Rb^I$ is such that
$\sup_{j,k\in\Nb}\mathcal{T}(f_j-f_k)<\infty$, then $\{f_j\}$ contains a subsequence,
which converges everywhere on~$I$.} This is an \emph{irregular\/} pointwise
selection principle in the sense that, although the sequence $\{f_j\}$ satisfying
Schrader's condition is pointwise bounded on $I$, we cannot infer any `regularity'
properties of the (pointwise) limit function (e.g., it may be applied to the sequence
$f_j(t)=(-1)^j\Dc(t)$ for $j\in\Nb$ and $t\in[0,1]$). Maniscalco \cite{Manisc} proved
that Schrader's assumption and condition \eq{e:nunfj} are independent (in the sense that
they produce different pointwise selection principles). Extensions of Schrader's result
are presented in \cite{JMAA08}, \cite{waterman80}, \cite{JMAA13}, \cite{Piazza}.

One of the goals of this paper is to obtain irregular pointwise selection principles in
terms of Fra{\v n}kov{\'a}'s $\vep$-variations $V_\vep(f)$ (Section~\ref{ss:irreg}).

This paper is a thorough self-contained study of the \emph{approximate variation},
i.e., the family $\{V_\vep(f)\}_{\vep>0}$ for functions $f:T\to M$ mapping a nonempty
subset $T$ of $\Rb$ into a metric space $(M,d)$. We develop a number of pointwise
(and almost everywhere) selection principles, including irregular ones, for sequences of
functions with values in metric spaces, normed spaces and reflexive separable Banach
spaces. All assertions and their sharpness are illustrated by concrete examples.
The plan of the exposition can be clearly seen from the Contents. Finally, it is to be noted
that, besides powerful selection principles, based on $\vep$-variations, the notion of
approximate variation gives a nice and highly nontrivial example of a \emph{metric
modular\/} in the sense of the author (\cite{DAN06-2}, \cite{NA05}, \cite{MMS}), or
a classical \emph{modular} in the sense of Musielak-Orlicz (\cite{Mus}, \cite{MuOr-2})
if $(M,\|\cdot\|)$ is a normed linear space. Results corresponding to the modular aspects
of the approximate variation will be published elsewhere.

\chapter{The approximate variation and its properties} \label{s:appv}

\section{Notation and terminology}

We begin by introducing notations and the terminology which will be used throughout
this paper.

Let $T$ be a nonempty set (in the sequel, $T\subset\Rb$),  $(M,d)$ be a metric space
with metric $d$, and $M^T$ be the set of all functions $f:T\to M$ mapping $T$
into~$M$. The set $M^T$ is equipped with the (extended-valued)
\emph{uniform metric}
  \begin{equation*}
d_{\infty,T}(f,g)=\sup_{t\in T}\,d(f(t),g(t)),\quad\,\,f,g\in M^T.
  \end{equation*}
The letter $c$ stands, as a rule, for a \emph{constant\/} function $c\in M^T$
(sometimes identified with $c\in M$).

The \emph{oscillation\/} of a function $f\in M^T$ on the set $T$ is the quantity%
\footnote{The notation for the oscillation $|f(T)|$ should not be confused with the
notation for the increment $|f(I_i)|=|f(b_i)-f(a_i)|$ from p.~\pageref{p:nco}, the latter
being used only in the Introduction.}
  \begin{equation*}
|f(T)|\equiv|f(T)|_d=\sup_{s,t\in T}d(f(s),f(t))\in[0,\infty],
  \end{equation*}
also known as the \emph{diameter of the image\/} $f(T)=\{f(t):t\in T\}\subset M$.
We denote by $\Bd(T;M)=\{f\in M^T:|f(T)|<\infty\}$ the set of all
\emph{bounded functions\/} from $T$ into~$M$.

Given $f,g\in M^T$ and $s,t\in T$, by the triangle inequality for $d$, we find
  \begin{equation} \label{e:1_1}
d_{\infty,T}(f,g)\le|f(T)|+d(f(t),g(t))+|g(T)|
  \end{equation}
and
  \begin{equation} \label{e:10}
d(f(s),f(t))\le d(g(s),g(t))+2d_{\infty,T}(f,g);
  \end{equation}
the definition of the oscillation and inequality \eq{e:10} imply
  \begin{equation} \label{e:1s2}
|f(T)|\le|g(T)|+2d_{\infty,T}(f,g).
  \end{equation}
Clearly (by \eq{e:1_1} and \eq{e:1s2}), $d_{\infty,T}(f,g)<\infty$ for all
$f,g\in\Bd(T;M)$ and, for any \emph{constant\/} function $c\in M^T$,
$\Bd(T;M)=\{f\in M^T:d_{\infty,T}(f,c)<\infty\}$.

For a sequence of functions $\{f_j\}\equiv\{f_j\}_{j=1}^\infty\subset M^T$ and
$f\in M^T$, we write:

(a) $f_j\to f$ on $T$ to denote the \emph{pointwise\/} (=\,\emph{everywhere\/})
\emph{convergence\/} of $\{f_j\}$ to $f$ (that is, $\D\lim_{j\to\infty}d(f_j(t),f(t))=0$
for all $t\in T$);

(b) $f_j\rra f$ on $T$ to denote the \emph{uniform convergence\/} of $\{f_j\}$ to $f$:
$\D\lim_{j\to\infty}d_{\infty,T}(f_j,f)=0$.
(Clearly, (b) implies (a), but not vice versa.)

Recall that a sequence of functions $\{f_j\}\subset M^T$ is said to be \emph{\pw\ \rc\/}
on $T$ provided the closure in $M$ of the set $\{f_j(t):j\in\Nb\}$
is compact for all $t\in T$.

From now on, we suppose that $T$ is a (nonempty) subset of the reals~$\Rb$.

The (Jordan) \emph{variation\/} of $f\in M^T$ is the quantity (e.g.,
\cite[Chapter~4, Section~9]{Schwartz})
  \begin{equation*}
V(f,T)=\sup_P\sum_{i=1}^md(f(t_i),f(t_{i-1}))\in[0,\infty],
  \end{equation*}
where the supremum is taken over all partitions $P$ of $T$, i.e., $m\in\Nb$ and
$P=\{t_i\}_{i=0}^m\subset T$ such that $t_{i-1}\le t_i$ for all $i=1,2,\dots,m$.
We denote by $\BV(T;M)=\{f\in M^T:V(f,T)<\infty\}$ the set of all
\emph{functions of bounded variation\/} from $T$ into~$M$.

The following four basic properties of the functional $V$ are well-known.
Given $f\in M^T$, we have:

\begin{itemize} \label{p:V}
\item[(V.1)] $|f(T)|\le V(f,T)$ (and so, $\BV(T;M)\subset\Bd(T;M)$);
\item[(V.2)] $V(f,T)=V(f,T\cap(-\infty,t])+V(f,T\cap[t,\infty))$ for all $t\in T$
  (\emph{additivity\/} of $V$ in the second variable, cf.~%
  \cite{Var}, \cite{JDCS}, \cite{Schwartz});
\item[(V.3)] if $\{f_j\}\subset M^T$ and $f_j\to f$ on $T$, then
  $V(f,T)\le\liminf_{j\to\infty}V(f_j,T)$ (sequential \emph{lower  semicontinuity\/} of
  $V$ in the first variable, cf.~\cite{JDCS}, \cite{MatSb});
\item[(V.4)] a \pw\ \rc\ sequence of functions $\{f_j\}\subset M^T$ satisfying condition
  $\sup_{j\in\Nb}V(f_j,T)<\infty$ contains a subsequence, which converges \pw\
  on $T$ to a function $f\in\BV(T;M)$ (Helly-type \emph{pointwise
  selection principle}, cf.~\cite{JMAA}, \cite{Sovae}).
\end{itemize}

In what follows, the letter $I$ denotes a closed interval $I=[a,b]$ with the endpoints
$a,b\in\Rb$, $a<b$.

Now, we recall the notion of a regulated function (introduced in \cite{Aumann} for real
valued functions). We say (\cite{JMAA05}) that a function $f\in M^I$ is \emph{regulated}
\label{p:reg} (or \emph{proper}, or \emph{simple\/}) and write $f\in\Reg(I;M)$ if it
satisfies the Cauchy condition at every point of $I=[a,b]$, i.e., $d(f(s),f(t))\to0$ as
$I\ni s,t\to\tau-0$ for each $a<\tau\le b$, and $d(f(s),f(t))\to0$ as
$I\ni s,t\to\tau'+0$ for each $a\le\tau'<b$. It is well-known (e.g.,
\cite{Var}, \cite{JMAA05},\cite{Schwartz}) that
  \begin{equation*}
\BV(I;M)\subset\Reg(I;M)\subset\Bd(I;M),
  \end{equation*}
the set $\Reg(I;M)$ of all regulated function is closed with respect to the uniform
convergence, and the pair $(\Reg(I;M),d_{\infty,I})$ is a complete metric space provided
$(M,d)$ is complete (see also \cite[Theorem~2]{Studia17} for some generalization).
Furthermore, if $(M,d)$ is complete, then, by Cauchy's criterion, we have:
$f\in\Reg(I;M)$ if and only if the left limit $f(\tau-0)\in M$ exists at each
point $a<\tau\le b$ (meaning that $d(f(t),f(\tau-0))\to0$ as $I\ni t\to\tau-0$),
and the right limit $f(\tau'+0)\in M$ exists at each point $a\le\tau'<b$
(i.e., $d(f(t),f(\tau'+0))\to0$ as $I\ni t\to\tau'+0$).

Regulated functions can be uniformly approximated by step functions (see \eq{e:stR}) as
follows. Recall that $f\in M^I$ is said to be a \emph{step function\/} (in symbols,
$f\in\mbox{\rm St}(I;M)$) provided, for some $m\in\Nb$, there exists a partition
$a=t_0<t_1<t_2<\dots<t_{m-1}<t_m=b$ of $I=[a,b]$ such that $f$ takes a
constant value on each (open) interval $(t_{i-1},t_i)$, $i=1,2,\dots,m$. Clearly,
  \begin{equation} \label{e:StBV}
\mbox{\rm St}(I;M)\subset\BV(I;M).
  \end{equation}
Furthermore (cf.\ \cite[(7.6.1)]{Dieu}), we have
  \begin{equation} \label{e:stR}
\Reg(I;M)=\{f\in M^I:\mbox{$\exists\,\{f_j\}\subset\mbox{\rm St}(I;M)$
such that $f_j\rra f$ on $I$}\}
  \end{equation}
(if, in addition, $f\in\BV(I;M)$, then $\{f_j\}\subset\mbox{\rm St}(I;M)$
can be chosen such that $f_j\rra f$ on $I$ and $V(f_j,I)\le V(f,I)$ for all $j\in\Nb$,
cf.~\cite[Section~1.27]{Var}).

\section{Definition of the approximate variation} \label{ss:dav}

\begin{definition} \label{def:av}
The \emph{approximate variation\/} of a function $f\!\in\! M^T$ is the one-parameter
family $\{V_\vep(f,T)\}_{\vep>0}$ of \emph{$\vep$-variations\/} defined,
for each $\vep>0$,~by
  \begin{equation} \label{e:av}
V_\vep(f,T)=\inf\,\{V(g,T):\mbox{$g\in\BV(T;M)$ and $d_{\infty,T}(f,g)\le\vep$}\}
  \end{equation}
(with the convention that $\inf\es=\infty$).
\end{definition}

The notion of $\vep$-variation, which plays a crucial role in this paper, is originally due
to Fra{\v n}kov{\'a} \cite[Definition~3.2]{Fr} for $T=I=[a,b]$ and $M=\Rb^N$. It was
also considered and extended in \cite[Sections~4, 6]{Studia17} to any $T\subset\Rb$ and
metric space $(M,d)$, and \cite{JMAA17} for metric space valued functions of two variables.

A few comments concerning Definition~\ref{def:av} are in order. Sometimes it is
convenient to rewrite \eq{e:av} as $V_\vep(f,T)=\inf\{V(g,T):g\in G_{\vep,T}(f)\}$,
where
  \begin{equation*}
G_{\vep,T}(f)=\{g\in\BV(T;M):d_{\infty,T}(f,g)\le\vep\}.
  \end{equation*}
So, we obtain the value $V_\vep(f,T)$ if we ``minimize'' the lower semicontinuous
functional $g\mapsto V(g,T)$ over the metric subspace $G_{\vep,T}(f)$ of $\BV(T;M)$.
Clearly, $V_\vep(f,T)\in[0,\infty]$, and the value $V_\vep(f,T)$ does not change if
we replace condition $g\in\mbox{\rm BV}(T;M)$ at the right-hand side of \eq{e:av}
by less restrictive conditions $g\in M^T$ or $g\in\Bd(T;M)$.

Condition $V_\vep(f,T)=\infty$ simply means that $G_{\vep,T}(f)=\es$, i.e.,
  \begin{equation} \label{e:besk}
\mbox{$V(g,T)=\infty$ \,for all \,$g\in M^T$ \,such that \,$d_{\infty,T}(f,g)\le\vep$.}
  \end{equation}

The finiteness of $V_\vep(f,T)$ is equivalent to the following: for any number
$\eta>V_\vep(f,T)$ there is a function $g\in\BV(T;M)$, depending on $\vep$ and $\eta$,
such that $d_{\infty,T}(f,g)\le\vep$ and $V_\vep(f,T)\le V(g,T)\le\eta$.
Given $k\in\Nb$, setting $\eta=V_\vep(f,T)+(1/k)$, we find that
there is $g_k^\vep\in\BV(T;M)$ such that
  \begin{equation} \label{e:fin}
d_{\infty,T}(f,g_k^\vep)\le\vep\quad\mbox{and}\quad
V_\vep(f,T)\le V(g_k^\vep,T)\le V_\vep(f,T)+(1/k);
  \end{equation}
in particular, \eq{e:fin} implies $V_\vep(f,T)=\lim_{k\to\infty}V(g_k^\vep,T)$.

Given $\vep>0$, condition $V_\vep(f,T)=0$ is characterized as follows (cf.~\eq{e:fin}):
  \begin{equation} \label{e:zer}
\mbox{$\exists\,\{g_k\}\!\subset\!\BV(T;M)$ such that
$\D\sup_{k\in\Nb}d_{\infty,T}(f,g_k)\!\le\!\vep$ and
$\D\lim_{k\to\infty}\!V(g_k,T)\!=\!0$.}
  \end{equation}
In particular, if $g_k=c$ is a constant function on $T$ for all $k\in\Nb$, we have:
  \begin{equation} \label{e:ze1}
\mbox{if \,$d_{\infty,T}(f,c)\le\vep$, \,then \,$V_\vep(f,T)=0$.}
  \end{equation}
This is the case when $|f(T)|\le\vep$; more explicitly, \eq{e:ze1} implies
  \begin{equation} \label{e:zero}
\mbox{if \,$\vep>0$ \,and \,$|f(T)|\le\vep$, \,then \,$V_\vep(f,T)=0$.}
  \end{equation}
In fact, fixing $t_0\in T$, we may define a constant function by $c(t)=f(t_0)$
for all $t\in T$, so that $d_{\infty,T}(f,c)\le|f(T)|\le\vep$.

The lower bound $|f(T)|$ for $\vep$ in \eq{e:zero} can be refined provided $f\in M^T$
satisfies certain additional assumptions. By \eq{e:1s2}, $|f(T)|\le2d_{\infty,T}(f,c)$ for
every constant function $c\in M^T$.
 Now, if $|f(T)|=2d_{\infty,T}(f,c)$ for some $c$, we have:
  \begin{equation} \label{e:zero2}
\mbox{if \,$\vep>0$ \,and \,$\vep\ge|f(T)|/2$, \,then \,$V_\vep(f,T)=0$.}
  \end{equation}
To see this, note that $d_{\infty,T}(f,c)=|f(T)|/2\le\vep$ and apply \eq{e:ze1}.

The number $|f(T)|/2$ in \eq{e:zero2} is the \emph{best possible\/} lower bound
for $\vep$, for which we may have $V_\vep(f,T)=0$; in fact, by Lemma~\ref{l:71}(b)
(see below), if $V_\vep(f,T)=0$, then $|f(T)|\le2\vep$, i.e., $\vep\ge|f(T)|/2$.
In other words, if $0<\vep<|f(T)|/2$, then $V_\vep(f,T)\ne0$.

To present an example of condition $|f(T)|=2d_{\infty,T}(f,c)$, suppose $f\in M^T$
has only two values, i.e., $f(T)=\{x,y\}$ for some $x,y\in M$, $x\ne y$. Then,
the mentioned condition is of the form
  \begin{equation} \label{e:2max}
d(x,y)=2\max\{d(x,c),d(y,c)\}\quad\mbox{for \,some}\quad c\in M.
  \end{equation}
Condition \eq{e:2max} is satisfied for such $f$ if, for instance, $(M,\|\cdot\|)$
is a \emph{normed linear space\/} over $\mathbb{K}=\Rb$ or $\mathbb{C}$
(always equipped) with the \emph{induced metric\/} $d(u,v)=\|u-v\|$, $u,v\in M$.
\label{p:nls}
In fact, we may set $c(t)=c=(x+y)/2$, $t\in T$. Note that \eq{e:2max} is concerned
with a certain form of `convexity' of metric space $(M,d)$ (cf.~%
\cite[Example~1]{Studia17}).

If $f(T)=\{x,y,z\}$, condition $|f(T)|=2d_{\infty,T}(f,c)$ is of the form
  \begin{equation*}
\max\{d(x,y),d(x,z),d(y,z)\}=2\max\{d(x,c),d(y,c),d(z,c)\}.
  \end{equation*}

Some elementary properties of $\vep$-variation(s) of $f\in M^T$ are gathered in

\begin{lemma} \label{l:ele}
{\rm(a)} The function $\vep\mapsto V_\vep(f,T):(0,\infty)\to[0,\infty]$ is
{\sl nonincreasing}, and so, the following inequalities hold\/
{\rm(}for one-sided limits{\rm):}
  \begin{equation} \label{e:osli}
\mbox{$V_{\vep+0}(f,T)\le V_\vep(f,T)\le V_{\vep-0}(f,T)$ \,in \,$[0,\infty]$
\,for \,all \,$\vep>0$.}
  \end{equation}
\par\hspace{-14pt}
{\rm(b)} If $\es\ne T_1\subset T_2\subset T$, then $V_\vep(f,T_1)\le V_\vep(f,T_2)$
  \,for \,all \,$\vep>0$.  
\end{lemma}

\proof
(a) Let $0<\vep_1<\vep_2$. Since $d_{\infty,T}(f,g)\le\vep_1$ implies
$d_{\infty,T}(f,g)\le\vep_2$ for $g\in M^T$, we get
$G_{\vep_1,T}(f)\subset G_{\vep_2,T}(f)$, and so, by \eq{e:av},
$V_{\vep_2}(f,T)\le V_{\vep_1}(f,T)$.

\smallbreak
(b) Given $g\in M^T$, $T_1\subset T_2$ implies
$d_{\infty,T_1}(f,g)\le d_{\infty,T_2}(f,g)$. So, for any $\vep>0$,
$G_{\vep,T_2}(f)\subset G_{\vep,T_1}(f)$, which, by \eq{e:av}, yields
$V_\vep(f,T_1)\le V_\vep(f,T_2)$.
\sq

\section{Variants of the approximate variation} \label{ss:vav}

Here we consider two modifications of the notion of approximate variation.

The first one is obtained if we replace the nonstrict inequality $\le\vep$ in \eq{e:av} by
the strict inequality $<\vep$; namely, given $f\in M^T$ and $\vep>0$, we set
  \begin{equation} \label{e:avm1}
V_\vep'(f,T)=\inf\,\{V(g,T):\mbox{$g\in\BV(T;M)$ such that $d_{\infty,T}(f,g)<\vep$}\}
  \end{equation}
($\inf\es=\infty$). Clearly, Lemma~\ref{l:ele} holds for $V_\vep'(f,T)$. More
specific properties of $V_\vep'(f,T)$ are exposed in the following

\begin{proposition} \label{pr:1}
Given $f\in M^T$, we have\/{\rm:}
  \begin{itemize}
\renewcommand{\itemsep}{0.0pt plus 0.5pt minus 0.25pt}
\item[{\rm(a)}] the function $\vep\mapsto V_\vep'(f,T)$, mapping $(0,\infty)$ into
  $[0,\infty]$, is continuous from the left on $(0,\infty);$
\item[{\rm(b)}] inequalities $V_{\vep_1}'(f,T)\le V_{\vep+0}'(f,T)\le V_\vep(f,T)%
  \le V_{\vep-0}(f,T)\le V_\vep'(f,T)$ hold for all\/ $0<\vep<\vep_1$.
  \end{itemize}
\end{proposition}

\proof
(a) In view of \eq{e:osli} for $V_\vep'(f,T)$, given $\vep>0$, it suffices to show that
$V_{\vep-0}'(f,T)\le V_\vep'(f,T)$ provided $V_\vep'(f,T)<\infty$. By \eq{e:avm1},
for any number $\eta>V_\vep'(f,T)$ there is $g=g_{\vep,\eta}\in\BV(T;M)$ such that
$d_{\infty,T}(f,g)<\vep$ and $V(g,T)\le\eta$. If a number $\vep'$ is such that
$d_{\infty,T}(f,g)<\vep'<\vep$, then \eq{e:avm1}
implies $V_{\vep'}'(f,T)\le V(g,T)\le\eta$. Passing to the limit as $\vep'\to\vep-0$,
we get $V_{\vep-0}'(f,T)\le\eta$ for all $\eta>V_\vep'(f,T)$, and so,
$V_{\vep-0}'(f,T)\le V_\vep'(f,T)$.

(b) To prove the first inequality, we note that $V_{\vep_1}'(f,T)\le V_{\vep'}'(f,T)$ for
all $\vep'$ with $\vep<\vep'<\vep_1$. It remains to pass to the limit as $\vep'\to\vep+0$.

For the second inequality, let $g\in G_{\vep,T}(f)$, i.e., $g\in\BV(T;M)$ and
$d_{\infty,T}(f,g)\le\vep$. Then, for any number $\vep'$ such that $\vep<\vep'$,
by virtue of \eq{e:avm1}, $V_{\vep'}'(f,T)\le V(g,T)$, and so, as $\vep'\to\vep+0$,
$V_{\vep+0}'(f,T)\le V(g,T)$. Taking the infimum over all $g\in G_{\vep,T}(f)$,
we obtain the second inequality.

The third inequality is a consequence of \eq{e:osli}.

Since $\{g\in\BV(T;M):d_{\infty,T}(f,g)<\vep\}\subset G_{\vep,T}(f)$, we have
$V_\vep(f,T)\le V_\vep'(f,T)$. Replacing $\vep$ by $\vep'$ with $0<\vep'<\vep$,
we get $V_{\vep'}(f,T)\le V_{\vep'}'(f,T)$, and so, passing to the limit as $\vep'\to\vep-0$
and taking into account item (a) above, we arrive at the fourth inequality.
\sq

In contrast to Proposition~\ref{pr:1}(a), it will be shown in Lemma~\ref{l:proper}(a) that
the function $\vep\mapsto V_\vep(f,T)$ is continuous from the right on $(0,\infty)$
\emph{only\/} under the additional assumption on the metric space $(M,d)$
(to be \emph{proper}).

In the case when $T=I=[a,b]$, the second variant of the approximate variation is
obtained if we replace the set of functions of bounded variation $\BV(I;M)$ in
\eq{e:av} by the set of step functions $\mbox{\rm St}(I;M)$: given $f\in M^I$,
  \begin{equation} \label{e:avm2}
V_\vep^s(f,I)=\inf\,\{V(g,I):\mbox{$g\in\mbox{\rm St}(I;M)$ and
$d_{\infty,I}(f,g)\le\vep$}\},\quad\vep>0
  \end{equation}
($\inf\es=\infty$). Clearly, $V_\vep^s(f,I)$ has the properties from Lemma~\ref{l:ele}.

\begin{proposition} \label{pr:2}
$V_{\vep+0}^s(f,I)\le V_\vep(f,I)\le V_\vep^s(f,I)$ for all $f\in M^I$ and $\vep>0$.
\end{proposition}

\proof
By \eq{e:StBV},
$\{g\in\mbox{\rm St}(I;M):d_{\infty,I}(f,g)\le\vep\}\subset G_{\vep,I}(f)$,
and so, \eq{e:av} and \eq{e:avm2} imply the right-hand side inequality.

In order to prove the left-hand side inequality, we may assume that $V_\vep(f,I)<\infty$.
By \eq{e:av}, for any $\eta>V_\vep(f,I)$ there is $g=g_{\vep,\eta}\in\BV(I;M)$
such that $d_{\infty,I}(f,g)\le\vep$ and $V(g,I)\le\eta$. Since
$g\in\BV(I;M)\subset\Reg(I;M)$, by virtue of \eq{e:stR}, there is a sequence
$\{g_j\}\subset\mbox{\rm St}(I;M)$ such that $g_j\rra g$ on~$I$ and
$V(g_j,I)\le V(g,I)$ for all natural~$j$. Hence $\limsup_{j\to\infty}V(g_j,I)\le V(g,I)$
and, by property (V.3) (p.~\pageref{p:V}), $V(g,I)\le\liminf_{j\to\infty}V(g_j,I)$,
and so, $\lim_{j\to\infty}V(g_j,I)=V(g,I)$. Now, let $\vep'>0$ be arbitrary. Then,
there is $j_1=j_1(\vep')\in\Nb$ such that $V(g_j,I)\le V(g,I)+\vep'$ for all $j\ge j_1$,
and, since $g_j\rra g$ on $I$, there is $j_2=j_2(\vep')\in\Nb$ such that
$d_{\infty,I}(g_j,g)\le\vep'$ for all $j\ge j_2$. Noting that, for all $j\ge\max\{j_1,j_2\}$,
$g_j\in\mbox{\rm St}(I;M)$ and
  \begin{equation*}
d_{\infty,I}(f,g_j)\le d_{\infty,I}(f,g)+d_{\infty,I}(g,g_j)\le\vep+\vep',
  \end{equation*}
by the definition \eq{e:avm2} of $V_\vep^s(f,I)$, we get
  \begin{equation*}
V_{\vep+\vep'}^s(f,I)\le V(g_j,I)\le V(g,I)+\vep'\le\eta+\vep'.
  \end{equation*}
Passing to the limit as $\vep'\to+0$, we find $V_{\vep+0}^s(f,I)\le\eta$ for all
$\eta>V_\vep(f,I)$, and so, $V_{\vep+0}^s(f,I)\le V_\vep(f,I)$.
\sq

Propositions \ref{pr:1}(b) and \ref{pr:2} show that the quantities $V_\vep'(f,T)$ and
$V_\vep^s(f,I)$ are somehow `equivalent' to $V_\vep(f,T)$, so their theories will no longer
be developed in the sequel, and the theory of $V_\vep(f,T)$ is sufficient for our purposes.

\section{Properties of the approximate variation} \label{ss:pro}

In order to effectively calculate the approximate variation of a function, we need more
of its properties. Item (a) in the next lemma justifies the term `approximate variation',
introduced in Definition~\ref{def:av}.

\begin{lemma} \label{l:71}
Given $f\in M^T$, we have\/{\rm:}\par\vspace{-6pt}
  \begin{itemize}
\item[{\rm(a)}] $\lim_{\vep\to+0}V_\vep(f,T)=\sup_{\vep>0}V_\vep(f,T)=V(f,T);$
\item[{\rm(b)}] $|f(T)|\le V_\vep(f,T)+2\vep$ \,for all \,$\vep>0;$
\item[{\rm(c)}] $|f(T)|=\infty$ {\rm(}i.e., $f\!\notin\!\Bd(T;M)${\rm)}
  if and only if \,$V_\vep(f,T)\!=\!\infty$ for all $\vep\!>\!0;$
\item[{\rm(d)}] $\inf_{\vep>0}(V_\vep(f,T)+\vep)\le|f(T)|\le
  \inf_{\vep>0}(V_\vep(f,T)+2\vep);$
\item[{\rm(e)}] $|f(T)|=0$ {\rm(}i.e.,\,$f$ is constant{\rm)}
   if and only if\, $V_\vep(f,T)\!=\!0$ for all $\vep\!>\!0;$ 
\item[{\rm(f)}] if \,$0<\vep<|f(T)|$, then
  \,$\max\{0,|f(T)|-2\vep\}\le V_\vep(f,T)\le V(f,T)$.
  \end{itemize}
\end{lemma}

\proof
(a) By Lemma~\ref{l:ele}(a), $C\equiv\lim_{\vep\to+0}V_\vep(f,T)%
=\sup_{\vep>0}V_\vep(f,T)$ is well-defined in $[0,\infty]$. First, we assume that
$f\in\BV(T;M)$. Since $f\in G_{\vep,T}(f)$ for all $\vep>0$, definition \eq{e:av} implies
$V_\vep(f,T)\le V(f,T)$ for all $\vep>0$, and so, $C\le V(f,T)<\infty$. Now, we prove
that $V(f,T)\le C$. By definition of $C$, 
for every $\eta>0$, there is $\delta=\delta(\eta)>0$
such that $V_\vep(f,T)<C+\eta$ for all $\vep\in(0,\delta)$. Let
$\{\vep_k\}_{k=1}^\infty\subset(0,\delta)$ be such that $\vep_k\to0$ as $k\to\infty$.
For every $k\in\Nb$, the definition of $V_{\vep_k}(f,T)<C+\eta$ implies the existence
of $g_k\in\BV(T;M)$ such that $d_{\infty,T}(f,g_k)\le\vep_k$ and $V(g_k,T)\le C+\eta$.
Since $\vep_k\to0$, $g_k\rra f$ on $T$, and so, property (V.3) on p.\,\pageref{p:V} yields
  \begin{equation*}
V(f,T)\le\liminf_{k\to\infty}V(g_k,T)\le C+\eta\quad\mbox{for \,all}\quad\eta>0,
  \end{equation*}
whence $V(f,T)\le C<\infty$. Thus, $C$ and $V(f,T)$ are finite or not simultaneously,
and $C=V(f,T)$, which establishes (a).

(b) The inequality is clear if $V_\vep(f,T)=\infty$, so we assume that $V_\vep(f,T)$ is
finite. By definition \eq{e:av}, for every $\eta>V_\vep(f,T)$ there is
$g=g_\eta\in\BV(T;M)$ such that $\mbox{$d_{\infty,T}(f,g)\le\vep$}$ and
$V(g,T)\le\eta$. Inequality \eq{e:1s2} and property (V.1) on p.\,\pageref{p:V} imply
  \begin{equation*}
|f(T)|\le|g(T)|+2d_{\infty,T}(f,g)\le V(g,T)+2\vep\le\eta+2\vep.
  \end{equation*}
It remains to take into account the arbitrariness of $\eta>V_\vep(f,T)$.

(c) The necessity is a consequence of item (b). To prove the sufficiency, assume, on the 
contrary, that $|f(T)|\!<\!\infty$. Then, by \eq{e:zero}, for any \mbox{$\vep\!>\!|f(T)|$},
we have $V_\vep(f,T)=0$, which contradicts the assumption $V_\vep(f,T)=\infty$.

(d) The right-hand side inequality is equivalent to item (b). To establish the left-hand
side inequality, we note that if $|f(T)|<\infty$ and $\vep>|f(T)|$, then, by \eq{e:zero},
$V_\vep(f,T)=0$, and so,
  \begin{equation*}
|f(T)|=\inf_{\vep>|f(T)|}\vep=\inf_{\vep>|f(T)|}(V_\vep(f,T)+\vep)
\ge\inf_{\vep>0}(V_\vep(f,T)+\vep).
  \end{equation*}
Now, if $|f(T)|=\infty$, then, by item (c), $V_\vep(f,T)+\vep=\infty$ for all $\vep>0$,
and so, $\inf_{\vep>0}(V_\vep(f,T)+\vep)=\infty$.

(e) ($\Rightarrow\!$) Since $f$ is constant on $T$, $f\in\BV(T;M)$ and
$d_{\infty,T}(f,f)=0<\vep$, and so, definition \eq{e:av} implies
$0\le V_\vep(f,T)\le V(f,T)=0$.

($\!\Leftarrow$) By virtue of item (d), if $V_\vep(f,T)=0$ for all $\vep>0$, then
$|f(T)|=0$.

(f) We may assume that $f\in\Bd(T;M)$. By item (a), $V_\vep(f,T)\le V(f,T)$, and
by item (b), $|f(T)|-2\vep\le V_\vep(f,T)$ for all $0<\vep<|f(T)|/2$. It is also clear
that $0\le V_\vep(f,T)$ for all $|f(T)|/2\le\vep<|f(T)|$.
\sq

\begin{remark} \label{r:8one} \rm
By \eq{e:zero} and Lemma~\ref{l:71}(c),\,(e), the $\vep$-variation $V_\vep(f,T)$,
initially defined for all $\vep>0$ and $f\in M^T$, is \emph{completely characterized\/}
whenever $\vep>0$ and $f\in\Bd(T;M)$ are such that $0<\vep<|f(T)|$.

The sharpness of assertions in Lemma~\ref{l:71}(b), (d) is presented in
Example~\ref{ex:gDf}(b), (c), (d) on pp.~\pageref{p:L71b}--\pageref{p:36d}
(for (a),\,(b),\,(f), see Example~\ref{exa:t}).
\end{remark}

In order to get the first feeling of the approximate variation, we present an example
(which later on will be generalized, cf.\ Example~\ref{ex:1}).

\begin{example} \label{exa:t} \rm
Let $f:T=[0,1]\to[0,1]$ be given by $f(t)=t$. We are going to evaluate $V_\vep(f,T)$,
$\vep>0$. Since $|f(T)|=1$, by \eq{e:zero}, $V_\vep(f,T)=0$ for all $\vep\ge1=|f(T)|$.
Moreover, if $c(t)\equiv1/2$ on $T$, then $|f(t)-c(t)|\le1/2$ for all $t\in T$, and so,
$V_\vep(f,T)=0$ for all $\vep\ge1/2$. Now, suppose $0<\vep<1/2$. By Lemma~%
\ref{l:71}(f), $V_\vep(f,T)\ge1-2\vep$. To establish the reverse inequality, define
$g:T\to\Rb$ by $g(t)=(1-2\vep)t+\vep$, $0\le t\le1$ (draw the graph on the plane).
Clearly, $g$ is increasing on $[0,1]$ and, for all $t\in[0,1]$,
  \begin{equation*}
f(t)-\vep=t-\vep=t-2\vep+\vep\le g(t)=t-2\vep t+\vep\le t+\vep=f(t)+\vep,
  \end{equation*}
i.e., $d_{\infty,T}(f,g)=\sup_{t\in T}|f(t)-g(t)|\le\vep$. It follows that
  \begin{equation*}
V(g,T)=g(1)-g(0)=(1-2\vep+\vep)-\vep=1-2\vep,
  \end{equation*}
and so, by definition \eq{e:av}, we get $V_\vep(f,T)\le V(g,T)=1-2\vep$. Thus,
  \begin{equation*}
\mbox{if \,$f(t)=t$, \,then}\,\,\, V_\vep(f,[0,1])=\left\{
  \begin{tabular}{ccr}
$\!\!1-2\vep$ & \mbox{if} & $0<\vep<1/2$,\\[2pt]
$\!\!0$ & \mbox{if} & $\vep\ge1/2$.
  \end{tabular}\right.
  \end{equation*}
\end{example}

\begin{lemma}[semi-additivity of the approximate variation] \label{l:mor}
Given $f\in M^T$, $\vep>0$, $t\in T$, if\/ $T_1=T\cap(-\infty,t]$ and\/
$T_2=T\cap[t,\infty)$,  then we have\/{\rm:}
  \begin{equation*}
V_\vep(f,T_1)+V_\vep(f,T_2)\le V_\vep(f,T)\le V_\vep(f,T_1)+V_\vep(f,T_2)+2\vep.
  \end{equation*}
\end{lemma}

\proof
1. First, we prove the left-hand side inequality. We may assume that
$V_\vep(f,T)<\infty$ (otherwise, the inequality is obvious). By definition \eq{e:av},
given $\eta>V_\vep(f,T)$, there is $g=g_\eta\in\BV(T;M)$ such that
$d_{\infty,T}(f,g)\le\vep$ and $V(g,T)\le\eta$. We set $g_1(s)=g(s)$ for all $s\in T_1$
and $g_2(s)=g(s)$ for all $s\in T_2$, and note that $g_1(t)=g(t)=g_2(t)$. Since,
for $i=1,2$, we have $d_{\infty,T_i}(f,g_i)\le d_{\infty,T}(f,g)\le\vep$ and
$g_i\in\BV(T_i;M)$, by \eq{e:av}, we find $V_\vep(f,T_i)\le V(g_i,T_i)$, and so,
the additivity property (V.2) of $V$ (p.~\pageref{p:V})~implies
  \begin{align*}
V_\vep(f,T_1)+V_\vep(f,T_2)&\le V(g_1,T_1)+V(g_2,T_2)=
  V(g,T_1)+V(g,T_2)\\[2pt]
&=V(g,T)\le\eta\quad\mbox{for \,all}\quad\eta>V_\vep(f,T).
  \end{align*}
This establishes the left-hand side inequality.

2. Now, we prove the right-hand side inequality. We may assume that $V_\vep(f,T_1)$
and $V_\vep(f,T_2)$ are finite (otherwise, our inequality becomes $\infty=\infty$).
We may also assume that $T\cap(-\infty,t)\ne\es$ and $T\cap(t,\infty)\ne\es$
(for, otherwise, we have $T_1=T\cap(-\infty,t]=\{t\}$ and $T_2=T$, or
\mbox{$T_2=T\cap[t,\infty)=\{t\}$} and $T_1=T$, respectively, and the inequality
is clear). By definition \eq{e:av}, for $i=1,2$, given $\eta_i>V_\vep(f,T_i)$, there
exists $g_i\in\BV(T_i;M)$ such that $d_{\infty,T_i}(f,g_i)\le\vep$ and $V(g_i,T_i)\le\eta_i$.
Given $u\in M$ (to be specified below), we define $g\in\BV(T;M)$ by
  \begin{equation*}
\mbox{$g(s)\!=\!g_1(s)$ if $s\in T\!\cap\!(-\infty,t)$, $g(t)\!=\!u$, and $g(s)\!=\!g_2(s)$
if $s\in T\!\cap\!(t,\infty)$.}
  \end{equation*}
Arguing with partitions of $T_i$ for $i=1,2$ (see step~3 below) and applying the
triangle inequality for $d$, we get
  \begin{equation} \label{e:nin1}
V(g,T_i)\le V(g_i,T_i)+d(g(t),g_i(t))\le\eta_i+d(u,g_i(t)).
  \end{equation}
By the additivity (V.2) of $V$, we find
  \begin{equation} \label{e:nin2}
V(g,T)=V(g,T_1)+V(g,T_2)\le\eta_1+d(u,g_1(t))+\eta_2+d(u,g_2(t)).
  \end{equation}
Now, we set $u=g_1(t)$ (by symmetry, we may set $u=g_2(t)$ as well). Since
$g=g_1$ on $T_1=T\cap(-\infty,t]$ and $g=g_2$ on $T\cap(t,\infty)\subset T_2$, we get
  \begin{equation} \label{e:nin3}
d_{\infty,T}(f,g)\le\max\{d_{\infty,T_1}(f,g_1),d_{\infty,T_2}(f,g_2)\}\le\vep.
  \end{equation}
Noting that (cf.~\eq{e:nin2})
  \begin{align*}
d(u,g_2(t))&=d(g_1(t),g_2(t))\le d(g_1(t),f(t))+d(f(t),g_2(t))\\[2pt]
&\le d_{\infty,T_1}(g_1,f)+d_{\infty,T_2}(f,g_2)\le\vep+\vep=2\vep,
  \end{align*}
we conclude from \eq{e:av}, \eq{e:nin3} and \eq{e:nin2} that
  \begin{equation*}
V_\vep(f,T)\le V(g,T)\le\eta_1+\eta_2+2\vep.
  \end{equation*}
The arbitrariness of numbers $\eta_1\!>\!V_\vep(f,T_1)$ and $\eta_2\!>\!V_\vep(f,T_2)$
proves the desired inequality.

3. \emph{Proof of\/ \eq{e:nin1}} for $i=1$ (the case $i=2$ is similar). Let
$\{t_k\}_{k=0}^m\subset T_1$ be a partition of $T_1$, i.e.,
$t_0<t_1<\dots<t_{m-1}<t_m=t$. Since $g(s)=g_1(s)$ for $s\in T$, $s<t$, we have:
  \begin{align*}
\sum_{k=1}^md(g(t_k),g(t_{k-1}))&=\sum_{k=1}^{m-1}d(g(t_k),g(t_{k-1}))
  +d(g(t_m),g(t_{m-1}))\\
&=\sum_{k=1}^{m-1}d(g_1(t_k),g_1(t_{k-1}))+d(g_1(t_m),g_1(t_{m-1}))\\
&\qquad+d(g(t_m),g(t_{m-1}))-d(g_1(t_m),g_1(t_{m-1}))\\[4pt]
&\le V(g_1,T_1)+|d(g(t),g_1(t_{m-1}))-d(g_1(t),g_1(t_{m-1}))|\\[4pt]
&\le V(g_1,T_1)+d(g(t),g_1(t)),
  \end{align*}
where the last inequality is due to the triangle inequality for $d$. Taking the supremum
over all partitions of $T_1$, we obtain the left-hand side inequality in \eq{e:nin1}
for $i=1$.
\sq

\begin{remark} \label{r:ifo}
The informative part of Lemma~\ref{l:mor} concerns the case when $f\in\Bd(T;M)$
and $0<\vep<|f(T)|$; if fact, if $\vep\ge|f(T)|$, then $\vep\ge|f(T_1)|$ and
$\vep\ge|f(T_2)|$, and so, by \eq{e:zero}, 
$V_\vep(f,T)=V_\vep(f,T_1)=V_\vep(f,T_2)=0$. The sharpness of the inequalities
in Lemma~\ref{l:mor} is shown in Example~\ref{ex:2}.
\end{remark}

Interestingly, the approximate variation characterizes regulated functions. The following
assertion is Fra{\v n}kov{\'a}'s result \cite[Proposition~3.4]{Fr} rewritten from $I=[a,b]$
and $M=\Rb^N$ to the case of an arbitrary metric space $(M,d)$ (which was
announced in \cite[equality~(4.2)]{Studia17}).

\begin{lemma} \label{l:Regc}
$\Reg(I;M)=\{f\in M^I:\mbox{$V_\vep(f,I)<\infty$ for all $\vep>0$.}\}$
\end{lemma}

\proof
($\subset$) If $f\!\in\!\Reg(I;M)$, then, by \eq{e:stR}, given $\vep\!>\!0$, there is
$g_\vep\!\in\!\mbox{\rm St}(I;M)$ such that $d_{\infty,I}(f,g_\vep)\le\vep$. Since
$g_\vep\in\BV(I;M)$, definition \eq{e:av} implies $V_\vep(f,I)\le V(g_\vep,I)<\infty$.

($\supset$) Suppose $f\in M^I$ and $V_\vep(f,I)<\infty$ for all $\vep>0$. Given
$a<\tau\le b$, let us show that $d(f(s),f(t))\to0$ as $s,t\to\tau-0$ (the arguments for
$a\le\tau'<b$ and the limit as $s,t\to\tau'+0$ are similar). Let $\vep>0$ be arbitrary.
We define the \emph{$\vep$-variation function\/} by $\vfi_\vep(t)=V_\vep(f,[a,t])$,
$t\in I$. By Lemma~\ref{l:ele}(b),
$0\le\vfi_\vep(s)\le\vfi_\vep(t)\le V_\vep(f,I)<\infty$ for all $s,t\in I$, $s\le t$, i.e.,
$\vfi_\vep:I\to[0,\infty)$ is bounded and nondecreasing, and so, the left limit
$\lim_{t\to\tau-0}\vfi_\vep(t)$ exists in $[0,\infty)$. Hence, there is
$\delta=\delta(\vep)\in(0,\tau-a]$ such that $|\vfi_\vep(t)-\vfi_\vep(s)|<\vep$
for all $s,t\in[\tau-\delta,\tau)$. Now, let $s,t\in[\tau-\delta,\tau)$, $s\le t$, be
arbitrary. Lemma~\ref{l:mor} (with $T_1=[a,s]$, $T_2=[s,t]$ and $T=[a,t]$) implies
$V_\vep(f,[s,t])\le\vfi_\vep(t)-\vfi_\vep(s)<\vep$. By the definition of $V_\vep(f,[s,t])$,
there is $g=g_\vep\in\BV([s,t];M)$ such that $d_{\infty,[s,t]}(f,g)\le\vep$ and
$V(g,[s,t])\le\vep$. Thus, by virtue of \eq{e:10},
  \begin{equation*}
d(f(s),f(t))\le d(g(s),g(t))+2d_{\infty,[s,t]}(f,g)\le V(g,[s,t])+2\vep\le3\vep.
  \end{equation*}
This completes the proof that $\lim_{I\ni s,t\to\tau-0}d(f(s),f(t))=0$.
\sq

\begin{remark} \label{r:indir}
We presented a direct proof of assertion $(\supset)$ in Lemma~\ref{l:Regc}. Indirectly,
we may argue as in \cite[Proposition~3.4]{Fr} as follows. Since, for each $k\in\Nb$,
$V_{1/k}(f,I)<\infty$, by definition \eq{e:av}, there is $g_k\in\BV(I;M)$ such that
$d_{\infty,I}(f,g_k)\le1/k$ (and $V(g_k,I)\le V_{1/k}(f,I)+(1/k)$). Noting that
$g_k\rra f$ on~$I$, each $g_k\in\Reg(I;M)$, and $\Reg(I;M)$ is closed with respect
to the uniform convergence, we get $f\in\Reg(I;M)$. An illustration of 
Lemma~\ref{l:Regc} is presented in Examples~\ref{ex:1} and \ref{ex:gDf}.
\end{remark}

Now we study the approximate variation in its interplay with the uniform convergence
of sequences of functions (see also Examples~\ref{ex:rieq}--\ref{ex:ucbw}).

\begin{lemma} \label{l:uc}
Suppose $f\in M^T$, $\{f_j\}\subset M^T$ and $f_j\rra f$ on $T$. We have\/{\rm:}
  \begin{itemize}
\renewcommand{\itemsep}{0.0pt plus 0.5pt minus 0.25pt}
\item[{\rm(a)}]
$\D\!\! V_{\vep+0}(f,T)\!\le\!\liminf_{j\to\infty}V_\vep(f_j,T)\!\le\!%
  \limsup_{j\to\infty}V_\vep(f_j,T)\!\le\! V_{\vep-0}(f,T)$ for all \mbox{$\vep\!>\!0;$}
\item[{\rm(b)}] $\!\!$if $V_\vep(f_j,T)\!<\!\infty$ for all $\vep\!>\!0$ and $j\!\in\!\Nb$,
  then $V_\vep(f,T)\!<\!\infty$ for all $\vep\!>\!0$.
  \end{itemize}
\end{lemma}

\proof
(a) Only the first and the last inequalities are to be verified.

1. In order to prove the first inequality, we may assume (passing to a suitable
subsequence of $\{f_j\}$ if necessary) that the right-hand side (i.e., the $\liminf$)
is equal to $C\equiv\lim_{j\to\infty}V_\vep(f_j,T)<\infty$. Suppose $\eta>0$ is given
arbitrarily. Then, there is $j_0=j_0(\eta)\in\Nb$ such that $V_\vep(f_j,T)\le C+\eta$
for all $j\ge j_0$. By the definition of $V_\vep(f_j,T)$, for every $j\ge j_0$ there is
$g_j=g_{j,\eta}\in\BV(T;M)$ such that $d_{\infty,T}(f_j,g_j)\le\vep$ and
$V(g_j,T)\le V_\vep(f_j,T)+\eta$. Since $f_j\rra f$ on $T$, we have
$d_{\infty,T}(f_j,f)\to0$ as $j\to\infty$, and so, there is $j_1=j_1(\eta)\in\Nb$ such that
$d_{\infty,T}(f_j,f)\le\eta$ for all $j\ge j_1$. Noting that
  \begin{equation*}
d_{\infty}(f,g_j)\le d_{\infty}(f,f_j)+d_{\infty}(f_j,g_j)\le\eta+\vep\quad
\mbox{for all $j\ge\max\{j_0,j_1\}$,}
  \end{equation*}
we find, by virtue of definition \eq{e:av},
  \begin{equation*}
V_{\eta+\vep}(f,T)\le V(g_j,T)\le V_\vep(f_j,T)+\eta\le(C+\eta)+\eta=C+2\eta.
  \end{equation*}
Passing to the limit as $\eta\to+0$, we arrive at $V_{\vep+0}(f,T)\le C$, which
was to be proved.

2. To establish the last inequality, with no loss of generality we may assume that
$V_{\vep-0}(f,T)<\infty$. Given $\eta>0$, there is $\delta=\delta(\eta,\vep)\in(0,\vep)$
such that if $\vep'\in[\vep-\delta,\vep)$, we have
$V_{\vep'}(f,T)\le V_{\vep-0}(f,T)+\eta$. Since $f_j\rra f$ on $T$, given
$\vep-\delta\le\vep'<\vep$, there is $j_0=j_0(\vep',\vep)\in\Nb$ such that
$d_{\infty,T}(f_j,f)\le\vep-\vep'$ for all $j\ge j_0$. By the definition of $V_{\vep'}(f,T)$,
for every $j\in\Nb$ we find $g_j=g_{j,\vep'}\in\BV(T;M)$ such that
$d_{\infty,T}(f,g_j)\le\vep'$ and
  \begin{equation*}
V_{\vep'}(f,T)\le V(g_j,T)\le V_{\vep'}(f,T)+(1/j),
  \end{equation*}
so that $\lim_{j\to\infty}V(g_j,T)=V_{\vep'}(f,T)$. Noting that, for all $j\ge j_0$,
  \begin{equation*}
d_{\infty,T}(f_j,g_j)\le d_{\infty,T}(f_j,f)+d_{\infty,T}(f,g_j)\le(\vep-\vep')+\vep'=\vep,
  \end{equation*}
we find from \eq{e:av} that $V_\vep(f_j,T)\le V(g_j,T)$ for all $j\ge j_0$. It follows that
  \begin{equation*}
\limsup_{j\to\infty}V_\vep(f_j,T)\le\lim_{j\to\infty}V(g_j,T)=V_{\vep'}(f,T)
\le V_{\vep-0}(f,T)+\eta.
  \end{equation*}
It remains to take into account the arbitrariness of $\eta>0$.

(b) Let $\vep>0$ and $0<\vep'<\vep$. Given $j\in\Nb$, since $V_{\vep'}(f_j,T)<\infty$,
by definition \eq{e:av}, there is $g_j\in\BV(T;M)$ such that $d_{\infty,T}(f_j,g_j)\le\vep'$
and $V(g_j,T)\le V_{\vep'}(f_j,T)+1$. Since $f_j\rra f$ on $T$, there is
$j_0=j_0(\vep-\vep')\in\Nb$ such that $d_{\infty,T}(f_{j_0},f)\le\vep-\vep'$.
Noting that
  \begin{equation*}
d_{\infty,T}(f,g_{j_0})\le d_{\infty,T}(f,f_{j_0})+d_{\infty,T}(f_{j_0},g_{j_0})
\le(\vep-\vep')+\vep'=\vep,
  \end{equation*}
we get, by \eq{e:av}, $V_\vep(f,T)\le V(g_{j_0},T)\le V_{\vep'}(f_{j_0},T)+1<\infty$.
\sq

\begin{lemma}[change of variable in the approximate variation] \label{l:chvar}
If $T\subset\Rb$, $\vfi:T\to\Rb$ is a {\sl strictly monotone} function and
 $f:\vfi(T)\to M$, then%
\footnote{Here, as usual, $\vfi(T)=\{\vfi(t):t\in T\}$ is the image of $T$ under $\vfi$,
and $f\circ\vfi$ is the composed function of $\vfi:T\to\Rb$ and $f:\vfi(T)\to M$ given
by $(f\circ\vfi)(t)=f(\vfi(t))$, $t\in T$.}
  \begin{equation*}
V_\vep(f,\vfi(T))=V_\vep(f\circ\vfi,T)\quad\mbox{for \,all \,\,$\vep>0$.}
  \end{equation*}
\end{lemma}

\proof
We need the following `change of variable' formula for Jordan's variation (cf.\ 
\cite[Theorem~2.20]{Var}, \cite[Proposition~2.1(V4)]{MatSb}):
if $T\subset\Rb$, $\vfi:T\to\Rb$
is a (not necessarily strictly) \emph{monotone} function and $g:\vfi(T)\to M$, then
  \begin{equation} \label{e:chava}
V(g,\vfi(T))=V(g\circ\vfi,T).
  \end{equation}

($\ge$) Suppose $V_\vep(f,\vfi(T))\!<\!\infty$. By definition \eq{e:av}, for every
$\eta\!>\!V_\vep(f,\vfi(T))$ there is $g\in\BV(\vfi(T);M)$ such that
$d_{\infty,\vfi(T)}(f,g)\le\vep$ and $V(g,\vfi(T))\le\eta$. We have $g\circ\vfi\in M^T$,
  \begin{equation} \label{e:444}
d_{\infty,T}(f\circ\vfi,g\circ\vfi)=d_{\infty,\vfi(T)}(f,g)\le\vep,
  \end{equation}
and, by \eq{e:chava}, $V(g\circ\vfi,T)=V(g,\vfi(T))\le\eta$. Thus,
by \eq{e:av} and \eq{e:444},
  \begin{equation*}
\mbox{$V_\vep(f\circ\vfi,T)\le V(g\circ\vfi,T)\le\eta$ \,\,\,for all 
\,\,$\eta>V_\vep(f,\vfi(T))$,}
  \end{equation*}
 and so,
$V_\vep(f\circ\vfi,T)\le V_\vep(f,\vfi(T))<\infty$.

($\le$) Now, suppose $V_\vep(f\circ\vfi,T)<\infty$. Then, for every
$\eta>V_\vep(f\circ\vfi,T)$ there exists $g\in\BV(T;M)$ such that
$d_{\infty,T}(f\circ\vfi,g)\le\vep$ and $V(g,T)\le\eta$. Denote by
$\vfi^{-1}:\vfi(T)\to T$ the inverse function of $\vfi$. Clearly, $\vfi^{-1}$ is
strictly monotone on $\vfi(T)$ in the same sense as $\vfi$ on $T$. Setting
$g_1=g\circ\vfi^{-1}$, we find $g_1:\vfi(T)\to M$ and, by \eq{e:444},
  \begin{align*}
d_{\infty,\vfi(T)}(f,g_1)&=d_{\infty,\vfi(T)}\bigl((f\circ\vfi)\circ\vfi^{-1},%
  g\circ\vfi^{-1}\bigr)\\[3pt]
&=d_{\infty,\vfi^{-1}(\vfi(T))}(f\circ\vfi,g)=d_{\infty,T}(f\circ\vfi,g)\le\vep.
  \end{align*}
Furthermore, by \eq{e:chava},
  \begin{equation*}
V(g_1,\vfi(T))=V(g\circ\vfi^{-1},\vfi(T))=V(g,\vfi^{-1}(\vfi(T)))=V(g,T)\le\eta.
  \end{equation*}
Thus, $V_\vep(f,\vfi(T))\le V(g_1,\vfi(T))\le\eta$ for all $\eta>V_\vep(f\circ\vfi,T)$,
which implies the inequality $V_\vep(f,\vfi(T))\le V_\vep(f\circ\vfi,T)<\infty$.
\sq

Lemma~\ref{l:chvar} will be applied in Example~\ref{ex:midp}
(cf.~Case $\al>1$ on p.~\pageref{p:a>1}).

Under additional assumptions on the metric space $(M,d)$, we get three more
properties of the approximate variation. Recall that $(M,d)$ is called \emph{proper\/}
(or has the \emph{Heine-Borel property\/}) if all \emph{closed bounded\/} subsets
of $M$ are \emph{compact}. \label{p:properms} For instance, if $(M,\|\cdot\|)$ is a
\emph{finite-dimensional\/} normed linear space with induced metric $d$
(cf.~p.~\pageref{p:nls}), then $(M,d)$ is a proper metric space.
Note that a proper metric space is complete. In fact, if $\{x_j\}_{j=1}^\infty$ is a Cauchy
sequence in $M$, then it is bounded and, since $M$ is proper, the set $\{x_j:j\in\Nb\}$
is \rc\ in~$M$. Hence a subsequence of $\{x_j\}_{j=1}^\infty$ converges in $M$ to an
element $x\in M$. Now, since $\{x_j\}_{j=1}^\infty$ is Cauchy, we get $x_j\to x$ as
$j\to\infty$, which proves the completeness of~$M$.

\begin{lemma} \label{l:proper}
Let $(M,d)$ be a {\sl proper} metric space and $f\in M^T$. We have\/{\rm:}
  \begin{itemize}
\renewcommand{\itemsep}{0.0pt plus 0.5pt minus 0.25pt}
\item[{\rm(a)}] the function $\vep\mapsto V_\vep(f,T)$ is continuous from the right
  on $(0,\infty);$
\item[{\rm(b)}] given $\vep>0$, $V_\vep(f,T)<\infty$ if and only if
  $V_\vep(f,T)=V(g,T)$ for some function $g=g_\vep\in G_{\vep,T}(f)$ {\rm(}i.e.,
  the infimum in\/ \eq{e:av} is attained, and so, becomes the minimum{\rm);}
\item[{\rm(c)}] if\/ $\{f_j\}\subset M^T$ and $f_j\to f$ on $T$, then
  $V_\vep(f,T)\le\liminf_{j\to\infty}V_\vep(f_j,T)$ for all $\vep>0$.
  \end{itemize}
\end{lemma}

\proof
(a) By virtue of \eq{e:osli}, it suffices to show that $V_\vep(f,T)\le V_{\vep+0}(f,T)$
provided $V_{\vep+0}(f,T)$ is finite. In fact, given
$\eta\!>\!V_{\vep+0}(f,T)\!=\!\lim_{\vep'\to\vep+0}V_{\vep'}(f,T)$,
there is \mbox{$\delta\!=\!\delta(\eta)\!>\!0$} such that $\eta>V_{\vep'}(f,T)$
 for all $\vep'$ with
$\vep<\vep'\le\vep+\delta$. Let $\{\vep_k\}_{k=1}^\infty$ be a sequence such that
$\vep<\vep_k\le\vep+\delta$ for all $k\in\Nb$ and $\vep_k\to\vep$ as $k\to\infty$.
Given $k\in\Nb$, setting $\vep'=\vep_k$, we find $\eta>V_{\vep_k}(f,T)$, and so,
by definition \eq{e:av}, there is $g_k\in\BV(T;M)$ (also depending on $\eta$) such that
  \begin{equation} \label{e:tfo}
d_{\infty,T}(f,g_k)\le\vep_k\quad\mbox{and}\quad V(g_k,T)\le\eta.
  \end{equation}
By the first inequality in \eq{e:tfo}, the sequence $\{g_k\}$ is pointwise bounded on $T$,
because, given $t\in T$, by the triangle inequality for $d$, we have
  \begin{align}
d(g_k(t),g_j(t))&\le d(g_k(t),f(t))+d(f(t),g_j(t)) \nonumber\\[3pt]
&\le d_{\infty,T}(g_k,f)+d_{\infty,T}(f,g_j) \label{e:tipb} \\[3pt]
&\le\vep_k+\vep_j\le2(\vep+\delta)\quad\mbox{for all}\quad k,j\in\Nb,\nonumber
  \end{align}
and since $(M,d)$ is \emph{proper}, the sequence $\{g_k\}$ is \pw\ \rc\ on~$T$.
So, the second inequality in \eq{e:tfo} and the Helly-type selection principle in
$\BV(T;M)$ (which is property (V.4) on~p.~\pageref{p:V}) imply the existence of a
subsequence of $\{g_k\}$, again denoted by $\{g_k\}$ (and the corresponding
subsequence of $\{\vep_k\}$---again by~$\{\vep_k\}$), and a function $g\in\BV(T;M)$
such that $g_k\to g$ pointwise on~$T$. Noting that, by \eq{e:tfo},
  \begin{equation} \label{e:kove}
d_{\infty,T}(f,g)\le\liminf_{k\to\infty}d_{\infty,T}(f,g_k)\le\lim_{k\to\infty}\vep_k=\vep
  \end{equation}
and, by the lower semicontinuity of $V$ (property (V.3) on p.~\pageref{p:V}),
  \begin{equation} \label{e:mke}
V(g,T)\le\liminf_{k\to\infty}V(g_k,T)\le\eta,
  \end{equation}
we find, from definition \eq{e:av}, that $V_\vep(f,T)\le V(g,T)\le\eta$. It remains
to take into account the arbitrariness of $\eta>V_{\vep+0}(f,T)$. 

Items (b) and (c) were essentially established in \cite{Fr} for $T=[a,b]$ and $M=\Rb^N$
as Propositions 3.5 and 3.6, respectively. For the sake of completeness, we present the
proofs of (b) and (c) in our more general situation.

(b) The sufficiency ($\!\Leftarrow$) is clear. Now we establish the necessity~%
($\Rightarrow\!$). By definition \eq{e:av}, given $k\in\Nb$, there is $g_k\in\BV(T;M)$
such that
  \begin{equation} \label{e:25on}
d_{\infty,T}(f,g_k)\le\vep\quad\mbox{and}\quad
V_\vep(f,T)\le V(g_k,T)\le V_\vep(f,T)+(1/k).
  \end{equation}
From \eq{e:tipb} and \eq{e:25on}, we find $d(g_k(t),g_j(t))\le2\vep$ for all
$k,j\in\Nb$ and~$t\in T$, and so, the sequence $\{g_k\}$ is \pw\ bounded on $T$,
and since $(M,d)$ is \emph{proper}, $\{g_k\}$ is \pw\ \rc\ on~$T$. Moreover,
by \eq{e:25on}, $\sup_{k\in\Nb}V(g_k,T)\le V_\vep(f,T)+1<\infty$.
By the Helly-type selection principle (V.4) in $\BV(T;M)$, there are a subsequence of
$\{g_k\}$, again denoted by $\{g_k\}$, and a function $g\in\BV(T;M)$ such that
$g_k\to g$ on $T$. As in \eq{e:kove}, we get $d_{\infty,T}(f,g)\le\vep$, and so,
\eq{e:av}, \eq{e:mke} and \eq{e:25on} yield
  \begin{equation*}
V_\vep(f,T)\le V(g,T)\le\lim_{k\to\infty}V(g_k,T)=V_\vep(f,T).
  \end{equation*}

(c) Passing to a subsequence of $\{f_j\}$ (if necessary), we may assume that the
right-hand side of the inequality in (c) is equal to $C_\vep=\lim_{j\to\infty}V_\vep(f_j,T)$
and finite. Given $\eta>C_\vep$, there is $j_0=j_0(\eta,\vep)\in\Nb$ such that
$\eta>V_\vep(f_j,T)$ for all $j\ge j_0$. For every $j\ge j_0$, by the definition of
$V_\vep(f_j,T)$, there is $g_j\in\BV(T;M)$ such that
  \begin{equation} \label{e:t25}
d_{\infty,T}(f_j,g_j)\le\vep\quad\mbox{and}\quad V(g_j,T)\le\eta.
  \end{equation}
Since $f_j\to f$ \pw\ on $T$, $\{f_j\}$ is \pw\ \rc\ on $T$, and so, $\{f_j\}$ is \pw\
bounded on $T$, i.e., $B(t)=\sup_{j,k\in\Nb}d(f_j(t),f_k(t))<\infty$ for all $t\in T$.
By the triangle inequality for $d$ and \eq{e:t25}, given $j,k\ge j_0$ and $t\in T$, we have
  \begin{align*}
d(g_j(t),g_k(t))&\le d(g_j(t),f_j(t))+d(f_j(t),f_k(t))+d(f_k(t),g_k(t))\\[3pt]
&\le d_{\infty,T}(g_j,f_j)+B(t)+d_{\infty,T}(f_k,g_k)\le B(t)+2\vep.
  \end{align*}
This implies that the sequence $\{g_j\}_{j=j_0}^\infty$ is \pw\ bounded on $T$, and
since $(M,d)$ is \emph{proper}, it is \pw\ \rc\ on~$T$. It follows from \eq{e:t25} that
$\sup_{j\ge j_0}V(g_j,T)$ does not exceed $\eta<\infty$,
 and so, by the Helly-type selection principle (V.4)
in $\BV(T;M)$, there are a subsequence $\{g_{j_p}\}_{p=1}^\infty$ of
$\{g_j\}_{j=j_0}^\infty$ and a function $g\in\BV(T;M)$ such that $g_{j_p}\to g$ \pw\
on~$T$ as $p\to\infty$. Since $f_{j_p}\to f$ \pw\ on $T$ as $p\to\infty$, we find,
from \eq{e:t25} and property (V.3) on p.~\pageref{p:V}, that
  \begin{equation*}
d_{\infty,T}(f,g)\le\liminf_{p\to\infty}d_{\infty,T}(f_{j_p},g_{j_p})\le\vep
  \end{equation*}
and
  \begin{equation*}
V(g,T)\le\liminf_{p\to\infty}V(g_{j_p},T)\le\eta.
  \end{equation*}
Now, definition \eq{e:av} implies $V_\vep(f,T)\le V(g,T)\le\eta$ for all $\eta>C_\vep$,
and so, $V_\vep(f,T)\le C_\vep=\lim_{j\to\infty}V_\vep(f_j,T)$, which was to be proved.
\sq

\begin{remark}
The inequality in Lemma~\ref{l:proper}(c) agrees with the left-hand side inequality in
Lemma~\ref{l:uc}(a): in fact, if $(M,d)$ is \emph{proper}, $\{f_j\}\subset M^T$ and
$f_j\rra f$ on $T$, then, by Lemma~\ref{l:proper}(a),
  \begin{equation*}
V_\vep(f,T)=V_{\vep+0}(f,T)\le\liminf_{j\to\infty}V_\vep(f_j,T),\quad\vep>0.
  \end{equation*}
The properness of $(M,d)$ in Lemma~\ref{l:proper} is essential:
item (a) is illustrated in Example~\ref{ex:gDf}(e) on p.~\pageref{p:rico},
(b)---in Example~\ref{ex:ims1}, and (c)---in Example~\ref{ex:ims2}.
\end{remark}

\chapter{Examples of approximate variations} \label{s:exav}

This section is devoted to various examples of approximate variations. In particular,
it is shown that all assertions in Section~\ref{ss:pro} are sharp.

\section{Functions with values in a normed linear space}

\begin{example} \label{ex:1} \rm
Let $T\subset\Rb$ and $(M,\|\cdot\|)$ be a normed linear space (cf.~p.~\pageref{p:nls}).
We have $d_{\infty,T}(f,g)=\|f-g\|_{\infty,T}$, $f,g\in M^T$, where the
\emph{uniform norm\/} on $M^T$ is given by
  \begin{equation*}
\|f\|_{\infty,T}=\sup_{t\in T}\|f(t)\|,\quad\,\,f\in M^T.
  \end{equation*}

We are going to estimate (and/or evaluate) the approximate variation
$\{V_\vep(f,T)\}_{\vep>0}$ for the function $f:T\to M$ defined, for $x,y\in M$, $x\ne0$,
by
  \begin{equation} \label{e:fxy}
\mbox{$f(t)=\vfi(t)x+y$, \,$t\in T$, \,\,where \,$\vfi\in\BV(T;\Rb)$ is
\emph{nonconstant}.}
  \end{equation}

To begin with, recall that $0<|\vfi(T)|\le V(\vfi,T)<\infty$ and
  \begin{equation*}
|\vfi(T)|=\sup_{s,t\in T}|\vfi(s)-\vfi(t)|=\mbox{\rm s}(\vfi)-\mbox{\rm i}(\vfi),
  \end{equation*}
where $\mbox{\rm s}(\vfi)\!\equiv\!\mbox{\rm s}(\vfi,T)\!=\!\sup_{t\in T}\vfi(t)$ and
$\mbox{i}(\vfi)\!\equiv\!\mbox{\rm i}(\vfi,T)\!=\!\inf_{t\in T}\vfi(t)$. Moreover,
  \begin{equation} \label{e:sif}
\biggl|\vfi(t)-\frac{\mbox{\rm i}(\vfi)+\mbox{\rm s}(\vfi)}2\biggr|\le
\frac{\mbox{\rm s}(\vfi)-\mbox{\rm i}(\vfi)}2=\frac{|\vfi(T)|}2\quad\,
\mbox{for all \,\,$t\in T$.}
  \end{equation}
In fact, given $t\in T$, we have $\mbox{\rm i}(\vfi)\le\vfi(t)\le\mbox{\rm s}(\vfi)$,
and so, subtracting $(\mbox{\rm i}(\vfi)+\mbox{\rm s}(\vfi))/2$ from
both sides, we get
  \begin{align*}
-\frac{|\vfi(T)|}2=\frac{\mbox{\rm i}(\vfi)-\mbox{\rm s}(\vfi)}2&=
  \mbox{\rm i}(\vfi)-\frac{\mbox{\rm i}(\vfi)+\mbox{\rm s}(\vfi)}2\le\\
&\le\vfi(t)-\frac{\mbox{\rm i}(\vfi)+\mbox{\rm s}(\vfi)}2\le\\
&\le\mbox{\rm s}(\vfi)-\frac{\mbox{\rm i}(\vfi)+\mbox{\rm s}(\vfi)}2=
  \frac{\mbox{\rm s}(\vfi)-\mbox{\rm i}(\vfi)}2=\frac{|\vfi(T)|}2.
  \end{align*}

Since $|f(T)|=|\vfi(T)|\!\cdot\!\|x\|$, by \eq{e:zero}, $V_\vep(f,T)=0$ for all
$\vep\ge|\vfi(T)|\!\cdot\!\|x\|$. Furthermore, if 
$c\equiv c(t)=(\mbox{\rm i}(\vfi)+\mbox{\rm s}(\vfi))(x/2)+y$, $t\in T$, then $c$
is a constant function on $T$ and, by \eq{e:sif}, we have
  \begin{equation*}
\|f(t)-c\|=\biggl|\vfi(t)-\frac{\mbox{\rm i}(\vfi)+\mbox{\rm s}(\vfi)}2\biggr|\!\cdot\!\|x\|
\le\frac{|\vfi(T)|}2\!\cdot\!\|x\|\quad\,\,\mbox{for all \,\,$t\in T$,}
  \end{equation*}
i.e., $\|f-c\|_{\infty,T}\le|\vfi(T)|\!\cdot\!\|x\|/2$. By \eq{e:ze1}, we find
  \begin{equation} \label{e:ov2}
V_\vep(f,T)=0\quad\mbox{for \,all}\quad\vep\ge|\vfi(T)|\!\cdot\!\|x\|/2.
  \end{equation}
Now, assume that $0<\vep<|\vfi(T)|\!\cdot\!\|x\|/2$. Lemma~\ref{l:71}(f) implies
  \begin{equation} \label{e:tvd}
V_\vep(f,T)\ge|f(T)|-2\vep=|\vfi(T)|\!\cdot\!\|x\|-2\vep.
  \end{equation}
Define the function $g\in M^T$ by
  \begin{align}
g(t)&=\biggl[\biggl(1-\frac{2\vep}{V(\vfi,T)\|x\|}\biggr)\vfi(t)+
  \frac{(\mbox{\rm i}(\vfi)+\mbox{\rm s}(\vfi))\,\vep}{V(\vfi,T)\|x\|}\biggr]x+y=
  \label{e:gt1}\\[4pt]
&=\vfi(t)x-\frac{2\vep}{V(\vfi,T)\|x\|}\biggl(\vfi(t)-
  \frac{\mbox{\rm i}(\vfi)+\mbox{\rm s}(\vfi)}2\biggr)x+y,\quad t\in T.
  \label{e:gt2}
  \end{align}
Note that since $|\vfi(T)|\!\le\! V(\vfi,T)$, the assumption on $\vep$ gives
$\vep\!<\!V(\vfi,T)\|x\|/2$, and so, $0<2\vep/(V(\vfi,T)\|x\|)<1$. Given $t\in T$,
\eq{e:gt2} and \eq{e:sif} imply
  \begin{equation*}
\|f(t)\!-\!g(t)\|\!=\!\frac{2\vep}{V(\vfi,T)\|x\|}\!\cdot\!
  \biggl|\vfi(t)-\frac{\mbox{\rm i}(\vfi)+\mbox{\rm s}(\vfi)}2\biggr|\!\cdot\!\|x\|
  \le\frac{2\vep}{V(\vfi,T)}\!\cdot\!\frac{|\vfi(T)|}{2}\le\vep,
  \end{equation*}
and so, $\|f-g\|_{\infty,T}\le\vep$. By \eq{e:gt1}, we find
  \begin{equation*}
V(g,T)=\biggl(1-\frac{2\vep}{V(\vfi,T)\|x\|}\biggr)V(\vfi,T)\|x\|=V(\vfi,T)\|x\|-2\vep.
  \end{equation*}
Hence, by definition \eq{e:av}, $V_\vep(f,T)\le V(g,T)=V(\vfi,T)\|x\|-2\vep$.
From here and \eq{e:tvd} we conclude that
  \begin{equation} \label{e:trdo}
|\vfi(T)|\cdot\|x\|-2\vep\le V_\vep(f,T)\le V(\vfi,T)\|x\|-2\vep
\,\,\,\mbox{if}\,\,\,0\!<\!\vep\!<\!|\vfi(T)|\cdot\|x\|/2.
  \end{equation}

In particular, if $\vfi\in\Rb^T$ is (nonconstant and) \emph{monotone}, then
$V(\vfi,T)=|\vfi(T)|$, and so, \eq{e:trdo} yields
  \begin{equation} \label{e:mntn}
\mbox{if \,\,$0\!<\!\vep\!<\!|\vfi(T)|\!\cdot\!\|x\|/2$, \,then
\,\,$V_\vep(f,T)=|\vfi(T)|\cdot\|x\|-2\vep$.}
  \end{equation}
Clearly, Example \ref{exa:t} is a particular case of \eq{e:mntn} and \eq{e:ov2} with
$T=[0,1]$, $M=\Rb$, $x=1$, $y=0$, and $\vfi(t)=t$, $t\in T$.

However, if $\vfi$ from \eq{e:fxy} is nonmonotone, both inequalities \eq{e:trdo} may be
strict (cf.~Remark~\ref{r:siq}). Note also that assertion \eq{e:mntn} implies the classical
Helly selection principle for monotone functions (cf.\ Remark~\ref{r:cHp}).
\end{example}

\begin{example} \label{ex:2} \rm
Here we show that the inequalities in Lemma~\ref{l:mor} are sharp and may be strict.
In fact, letting $\vfi(t)=t$, $t\in T=[0,1]$, and $y=0$ in \eq{e:fxy}, and setting
$T_1=[0,\frac12]$ and $T_2=[\frac12,1]$, we get, by virtue of \eq{e:mntn}
and \eq{e:ov2},
  \begin{equation*}
V_\vep(f,T)=\left\{
  \begin{tabular}{ccr}
$\!\!\|x\|-2\vep$ & \mbox{if} & $0<\vep<\frac12\|x\|$,\\[3pt]
$\!\!0$ & \mbox{if} & $\vep\ge\frac12\|x\|$,
  \end{tabular}\right.
  \end{equation*}
and, for $i=1,2$,
  \begin{equation*}
V_\vep(f,T_i)=\left\{
  \begin{tabular}{ccr}
$\!\!\frac12\|x\|-2\vep$ & \mbox{if} & $0<\vep<\frac14\|x\|$,\\[3pt]
$\!\!0$ & \mbox{if} & $\vep\ge\frac14\|x\|$.
  \end{tabular}\right.
  \end{equation*}
It remains, in Lemma~\ref{l:mor}, to consider the cases:
(a) $0<\vep<\frac14\|x\|$, (b) $\frac14\|x\|\le\vep<\frac12\|x\|$,
and (c) $\vep\ge\frac12\|x\|$. Explicitly, we have, in place of
  \begin{align*}
\mbox{$V_\vep(f,T_1)+V_\vep(f,T_2)$}&\le\mbox{$V_\vep(f,T)$}\le
  \mbox{$V_\vep(f,T_1)+V_\vep(f,T_2)+2\vep:$}\\[8pt]
\mbox{(a) $(\frac12\|x\|\!-\!2\vep)\!+\!(\frac12\|x\|\!-\!2\vep)$}&
  <\mbox{$\|x\|\!-\!2\vep$}=
  \mbox{$(\frac12\|x\|\!-\!2\vep)\!+\!(\frac12\|x\|\!-\!2\vep)+2\vep;$}\\[2pt]
\mbox{(b)\quad $0+0$}&<\mbox{$\|x\|\!-\!2\vep$}\le\mbox{$0+0+2\vep;$}\\[2pt]
\mbox{(c)\quad $0+0$}&=\mbox{\quad\quad$\!0\,$\quad}<
  \mbox{$0+0+2\vep.$}
  \end{align*}
\end{example}

\begin{example} \label{ex:thr} \rm
Let $\tau\in I=[a,b]$, $(M,d)$ be a metric space, and $x,y\in M$, $x\ne y$.
Define $f\in M^I$ by
  \begin{equation} \label{e:ftau}
f(t)\equiv f_\tau(t)=\left\{
  \begin{tabular}{ccl}
$\!\!x$ & \mbox{if} & $t=\tau$,\\[2pt]
$\!\!y$ & \mbox{if} & $t\in I$, $t\ne\tau$.
  \end{tabular}\right.
  \end{equation}
Clearly, $|f(I)|=d(x,y)$, $V(f,I)=d(x,y)$ if $\tau\in\{a,b\}$, and \mbox{$V(f,I)=2d(x,y)$}
if \mbox{$a\!<\!\tau\!<\!b$}. By \eq{e:zero}, we get $V_\vep(f,I)=0$
for all $\vep\ge d(x,y)$.
Lemma~\ref{l:71}(f) provides the following inequalities for $0<\vep<d(x,y)$:

(a) if $\tau=a$ or $\tau=b$, then
  \begin{align*}
d(x,y)-2\vep&\le V_\vep(f,I)\le d(x,y)\quad\mbox{if}\quad
  0<\vep<\textstyle\frac12d(x,y),\\[3pt]
0&\le V_\vep(f,I)\le d(x,y)\quad\mbox{if}\quad\textstyle\frac12d(x,y)\le\vep<d(x,y);
  \end{align*}

(b) if $a<\tau<b$, then
  \begin{align*}
d(x,y)-2\vep&\le V_\vep(f,I)\le 2d(x,y)\quad\mbox{if}\quad
  0<\vep<\textstyle\frac12d(x,y),\\[3pt]
0&\le V_\vep(f,I)\le 2d(x,y)\quad\mbox{if}\quad\textstyle\frac12d(x,y)\le\vep<d(x,y).
  \end{align*}

Under additional assumptions on the metric space $(M,d)$, the values $V_\vep(f,I)$ for
$0<\vep<d(x,y)$ can be given more exactly. To see this, we consider two cases
(A) and (B) below.

(A) Let $M=\{x,y\}$ be the two-point set with metric $d$ and $0\!<\!\vep\!<\!d(x,y)$.
Since $f(t)=x$ or $f(t)=y$ for all $t\in I$, we have: if $g\in M^I$ and
$d_{\infty,I}(f,g)\le\vep$, then $g=f$ on $I$, i.e., $G_{\vep,I}(f)=\{f\}$. Thus,
$V_\vep(f,I)=V(f,I)$, and so,
  \begin{align*}
V_\vep(f,I)&=d(x,y)\quad\,\,\,\mbox{if}\quad \tau\in\{a,b\},\\[3pt]
V_\vep(f,I)&=2d(x,y)\quad \mbox{if}\quad a<\tau<b.
  \end{align*}

(B) Let $(M,\|\cdot\|)$ be a normed linear space with induced metric $d$ and
$0<\vep<d(x,y)=\|x-y\|$. By \eq{e:zero2}, $V_\vep(f,I)=0$ for all
$\vep\ge\frac12\|x-y\|$. We assert that if $0<\vep<\frac12\|x-y\|$, then
  \begin{align}
V_\vep(f,I)&=\|x-y\|-2\vep\qquad\,\mbox{if}\quad\tau\in\{a,b\},\label{e:tab}\\[3pt]
V_\vep(f,I)&=2(\|x-y\|-2\vep)\quad\mbox{if}\quad a<\tau<b.\label{e:ntab}
  \end{align}

In order to establish these equalities, we first note that the function $f$ from
\eq{e:ftau} can be expressed as (cf.~\eq{e:fxy})
  \begin{equation} \label{e:ftex}
\mbox{$f(t)\!=\!\vfi(t)(x-y)+y$, where $\vfi(t)\!\equiv\!\vfi_\tau(t)\!=\!\left\{
  \begin{tabular}{ccl}
$\!\!1$ & \mbox{if} & $t=\tau$,\\[2pt]
$\!\!0$ & \mbox{if} & $t\ne\tau$,
  \end{tabular}\right.$ \,$t\in I$.}
  \end{equation}

\emph{Proof of}~\eq{e:tab}. If $\tau\in\{a,b\}$, then $\vfi$ is monotone on $I$ with
$\mbox{\rm i}(\vfi)\!=\!0$, $\mbox{\rm s}(\vfi)\!=\!1$, and $V(\vfi,I)\!=\!|\vfi(I)|\!=\!1$.
Now, \eq{e:tab} follows from \eq{e:ftex}~and~\eq{e:mntn}.

\smallbreak
Note that function $g$ from \eq{e:gt1}, used in obtaining \eq{e:tab}, is of the form
  \begin{equation*}
g(t)=\biggl[\biggl(1-\frac{2\vep}{\|x-y\|\,}\biggr)\vfi(t)+
  \frac{\vep}{\|x-y\|}\biggr](x-y)+y,\quad t\in I,
  \end{equation*}
i.e., if $\mbox{\rm e}_{x,y}=(x-y)/\|x-y\|$ is the \emph{unit vector\/} 
(`directed from $y$ to $x$'), then
  \begin{equation} \label{e:unve}
g(\tau)=x-\vep\mbox{\rm e}_{x,y}\quad\mbox{and}\quad
g(t)=y+\vep\mbox{\rm e}_{x,y},\,\,\,t\in I\setminus\{\tau\}.
  \end{equation}
This implies $\|f-g\|_{\infty,I}=\vep$ (for all $\tau\in I$), and we have,
for $\tau\in\{a,b\}$,
  \begin{equation} \label{e:Vgta}
V(g,I)=|g(I)|=\|(x-\vep\mbox{\rm e}_{x,y})-(y+\vep\mbox{\rm e}_{x,y})\|
=\|x-y\|-2\vep.
  \end{equation}

\emph{Proof of}~\eq{e:ntab}. Suppose $a<\tau<b$ and $0<\vep<\frac12\|x-y\|$.
First, consider an arbitrary function $g\in M^I$ such that
$\|f-g\|_{\infty,I}=\sup_{t\in I}\|f(t)-g(t)\|\le\vep$. Since $P=\{a,\tau,b\}$ is a
partition of $I$, by virtue of \eq{e:10} and \eq{e:ftau}, we get
  \begin{align}
V(g,I)&\ge\|g(\tau)-g(a)\|+\|g(b)-g(\tau)\|\nonumber\\[3pt]
&\ge(\|f(\tau)-f(a)\|-2\vep)+(\|f(b)-f(\tau)\|-2\vep)\label{e:ggff2e}\\[3pt]
&=2(\|x-y\|-2\vep).\nonumber
  \end{align}
Due to the arbitrariness of $g$ as above, \eq{e:av} implies
$V_\vep(f,I)\ge2(\|x-y\|-2\vep)$. Now, for the function $g$ from \eq{e:unve},
the additivity of $V$ and \eq{e:Vgta} yield
  \begin{equation*}
V(g,I)\!=\!V(g,[a,\tau])+V(g,[\tau,b])\!=\!|g([a,\tau])|+|g([\tau,b])|
\!=\!2(\|x-y\|-2\vep),
  \end{equation*}
and so, $V_\vep(f,I)\!\le\! V(g,I)\!=\!2(\|x-y\|-2\vep)$. This completes the proof of~\eq{e:ntab}.
\end{example}

\begin{remark} \label{r:siq}
If $\vfi$ from \eq{e:fxy} is nonmonotone, inequalities in \eq{e:trdo} may be
\emph{strict}. In fact, supposing $a<\tau<b$, we find that the function $\vfi=\vfi_\tau$
from \eq{e:ftex} is not monotone, $|\vfi(I)|=1$ and $V(\vfi,I)=2$, and so, by
\eq{e:ntab}, inequalities \eq{e:trdo} for function $f$ from \eq{e:ftex} are of the form:
  \begin{equation*}
\|x-y\|-2\vep<V_\vep(f,I)=2(\|x-y\|-2\vep)<2\|x-y\|-2\vep
  \end{equation*} 
if $0<\vep<\frac12\|x-y\|$.
\end{remark}

\begin{example} \label{ex:midp}
Let $I=[a,b]$, $a<\tau<b$, $(M,\|\cdot\|)$ be a normed linear space, $x,y\in M$,
$x\ne y$, and $\al\in\Rb$. Define $f\in M^I$ by
  \begin{equation} \label{e:alp}
\mbox{$f(t)\!=\!x$ if $a\le t<\tau$, \,$f(\tau)\!=\!(1\!-\!\al)x\!+\!\al y$,
and $f(t)\!=\!y$ if $\tau<t\le b$.}
  \end{equation}
We are going to evaluate the approximate variation $\{V_\vep(f,I)\}_{\vep>0}$
for all $\al\!\in\!\Rb$.
For this, we consider three possibilities: $0\le\al\le1$, $\al<0$, and $\al>1$.

\emph{Case\/ $0\le\al\le1$}. We assert that (independently of $\al\in[0,1]$)
  \begin{equation} \label{e:veo}
V_\vep(f,I)=\left\{
  \begin{tabular}{ccr}
$\!\!\|x-y\|-2\vep$ & \mbox{if} & $0<\vep<\textstyle\frac12\|x-y\|$,\\[3pt]
$\!\!0$ & \mbox{if} & $\vep\ge\textstyle\frac12\|x-y\|$.
  \end{tabular}\right.
  \end{equation}
To see this, we note that $f$ can be represented in the form \eq{e:fxy}:
  \begin{equation*}
f(t)\!=\!\vfi(t)(x-y)+(1-\al)x+\al y\!\quad\mbox{with}\!\quad\vfi(t)=\left\{
  \begin{tabular}{ccc}
$\!\!\al$ & \mbox{if} & $a\le t<\tau$,\\[2pt]
$\!\!0$ & \mbox{if} & $t=\tau$,\\[2pt]
$\!\!\al\!-\!1$ & \mbox{if} & $\tau<t\le b$.
  \end{tabular}\right.
  \end{equation*}
Since $\al\in[0,1]$, $\vfi$ is nonicreasing on $I$ and $|\vfi(I)|=|\al-(\al-1)|=1$.
Hence, \eq{e:mntn} implies the first line in \eq{e:veo}. The second line in \eq{e:veo} is
a consequence of \eq{e:ov2}.

\emph{Case\/ $\al<0$}. The resulting form of $V_\vep(f,I)$ is given by \eq{e:cas1},
\eq{e:cas2} and \eq{e:buda}.
Now we turn to their proofs. We set $x_\al=(1-\al)x+\al y$ in \eq{e:alp} and note that
  \begin{align}
x_\al-x&=(-\al)(x-y)=(-\al)\|x-y\|\mbox{\rm e}_{x,y},\label{e:xax}\\[3pt]
x_\al-y&=(1-\al)(x-y)=(1-\al)\|x-y\|\mbox{\rm e}_{x,y},\label{e:xay}
  \end{align}
where $\mbox{\rm e}_{x,y}=(x-y)/\|x-y\|$.

Let us evaluate $|f(I)|$ and $V(f,I)$. Since $1-\al>-\al$, and  $\al<0$ implies
$1-\al>1$, by \eq{e:xax} and \eq{e:xay}, $\|x_\al-y\|>\|x_\al-x\|$ and
$\|x_\al-y\|>\|x-y\|$, and since $f$ assumes only values $x$, $x_\al$, and $y$,
  \begin{equation*}
|f(I)|=\|x_\al-y\|=(1-\al)\|x-y\|.
  \end{equation*}
For $V(f,I)$, by the additivity (V.2) of $V$, \eq{e:xax} and \eq{e:xay}, we find
  \begin{align*}
V(f,I)&=V(f,[a,\tau])+V(f,[\tau,b])=|f([a,\tau])|+|f([\tau,b])|\\[3pt]
&=\|f(\tau)-f(a)\|+\|f(b)-f(\tau)\|=\|x_\al-x\|+\|y-x_\al\|\\[3pt]
&=(-\al)\|x-y\|+(1-\al)\|x-y\|=(1-2\al)\|x-y\|.
  \end{align*}

Setting $c=c(t)=\frac12(x_\al+y)$ for all $t\in I$, we get, by \eq{e:xay},
  \begin{equation*}
\|x_\al-c\|=\|y-c\|=\textstyle\frac12\|x_\al-y\|=\frac12(1-\al)\|x-y\|=\frac12|f(I)|,
  \end{equation*}
and
  \begin{align*}
\|x-c\|&=\textstyle\|x-\frac12(x_\al+y)\|=\frac12\|(x\!-\!x_\al)+(x\!-\!y)\|
  =\frac12\|\al(x\!-\!y)+(x\!-\!y)\|\\[4pt]
&=\textstyle\frac12|\al+1|\!\cdot\!\|x\!-\!y\|\le\frac12(1+|\al|)\|x\!-\!y\|
  \stackrel{{\scriptscriptstyle(\al<0)}}{=}
  \frac12(1\!-\!\al)\|x\!-\!y\|=\frac12|f(I)|.
  \end{align*}
Hence $\|f-c\|_{\infty,I}\le\frac12|f(I)|$, and it follows from \eq{e:ze1} that
  \begin{equation} \label{e:buda}
V_\vep(f,I)=0\quad\mbox{if}\quad\vep\ge\textstyle\frac12|f(I)|
=\frac{1-\al}2\|x-y\|.
  \end{equation}

It remains to consider the case when $0<\vep<\frac{1-\al}2\|x-y\|$, which we split
into two subcases:
  \begin{equation*}
\mbox{(I) $0<\vep<\frac{(-\al)}2\|x-y\|$, and
(II) $\frac{(-\al)}2\|x-y\|\le\vep<\frac{1-\al}2\|x-y\|$.}
  \end{equation*}

\emph{Subcase\/}~(I). First, given $g\in M^I$ with $\|f-g\|_{\infty,I}\le\vep$, since
$P=\{a,\tau,b\}$ is a partition of $I$, applying \eq{e:ggff2e}, we get
  \begin{align*}
V(g,I)&\ge(\|f(\tau)-f(a)\|-2\vep)+(\|f(b)-f(\tau)\|-2\vep)\\[3pt]
&=((-\al)\|x-y\|-2\vep)+((1-\al)\|x-y\|-2\vep)\\[3pt]
&=(1-2\al)\|x-y\|-4\vep,
  \end{align*}
and so, by \eq{e:av}, $V_\vep(f,I)\ge(1-2\al)\|x-y\|-4\vep$. Now, we define a concrete
(=`test') function $g\in M^I$ by the rule:
  \begin{align}
g(t)&=x+\vep\mbox{\rm e}_{x,y}\quad\mbox{if}\quad a\le t<\tau,\nonumber\\[3pt]
g(\tau)&=x_\al-\vep\mbox{\rm e}_{x,y},\label{e:gco}\\[3pt]
g(t)&=y+\vep\mbox{\rm e}_{x,y}\quad\mbox{if}\quad \tau<t\le b.\nonumber
  \end{align}
Clearly, by \eq{e:alp} and \eq{e:gco}, $\|f-g\|_{\infty,I}=\vep$. Furthermore,
  \begin{align*}
V(g,I)&=\|g(\tau)-g(a)\|+\|g(\tau)-g(b)\|\\[3pt]
&=\|(x_\al-x)-2\vep\mbox{\rm e}_{x,y}\|+\|(x_\al-y)-2\vep\mbox{\rm e}_{x,y}\|\\[3pt]
&=\Bigl\|(-\al)\|x-y\|\mbox{\rm e}_{x,y}-2\vep\mbox{\rm e}_{x,y}\Bigr\|
  +\Bigl\|(1-\al)\|x-y\|\mbox{\rm e}_{x,y}-2\vep\mbox{\rm e}_{x,y}\Bigr\|\\[3pt]
&=\bigl|(-\al)\|x-y\|-2\vep\bigr|+\bigl|(1-\al)\|x-y\|-2\vep\bigr|.
  \end{align*}
Assumption (I) implies $2\vep<(-\al)\|x-y\|<(1-\al)\|x-y\|$, so
  \begin{equation*}
V(g,I)\!=\!((-\al)\|x-y\|-2\vep)+((1-\al)\|x-y\|-2\vep)\!=\!(1-2\al)\|x-y\|-4\vep.
  \end{equation*}
By \eq{e:av}, $V_\vep(f,I)\le V(g,I)=(1-2\al)\|x-y\|-4\vep$. Thus,
  \begin{equation} \label{e:cas1}
V_\vep(f,I)=(1-2\al)\|x-y\|-4\vep\quad\mbox{if}\quad 
0<\vep<\textstyle\frac{(-\al)}2\|x-y\|.
  \end{equation}
Note that, in agreement with Lemma~\ref{l:71}(a), $V_\vep(f,I)\to V(f,I)$ as
$\vep\to+0$.

\emph{Subcase\/}~(II). First, given $g\in M^I$ with $\|f-g\|_{\infty,I}\le\vep$,
by virtue of \eq{e:10} and \eq{e:xay}, we get
  \begin{equation*}
V(g,I)\ge\|g(b)-g(\tau)\|\ge\|f(b)-f(\tau)\|-2\vep=(1-\al)\|x-y\|-2\vep,
  \end{equation*}
and so, definition \eq{e:av} implies $V_\vep(f,I)\ge(1-\al)\|x-y\|-2\vep$. Now, define
a test function $g\in M^I$ by
  \begin{equation} \label{e:gspe}
\mbox{$g(t)=x_\al-\vep\mbox{\rm e}_{x,y}$ \,if \,$a\le t\le\tau$, and
\,$g(t)=y+\vep\mbox{\rm e}_{x,y}$ \,if \,$\tau<t\le b$.}
  \end{equation}

Let us show that $\|f-g\|_{\infty,I}\le\vep$. Clearly, by \eq{e:alp}, $\|f(t)-g(t)\|=\vep$
for all $\tau\le t\le b$. Now, suppose $a\le t<\tau$. We have, by \eq{e:xax},
  \begin{align*}
\|f(t)-g(t)\|&=\|x-x_\al+\vep\mbox{\rm e}_{x,y}\|=
  \Bigl\|\al\|x-y\|\mbox{\rm e}_{x,y}+\vep\mbox{\rm e}_{x,y}\Bigr\|\\[2pt]
&=\bigl|\al\|x-y\|+\vep\bigr|\equiv A_\al.
  \end{align*}
Suppose first that $\al>-1$ (i.e., $x_\al$ is closer to $x$ than $x$ to $y$ in the sense that
$\|x_\al-x\|=(-\al)\|x-y\|<\|x-y\|$). Then $(-\al)<\frac12(1-\al)$, and so, for $\vep$
from subcase (II) we have either
  \begin{equation*}
\mbox{(II${}_1$) $\frac{(-\al)}2\|x\!-\!y\|\le\vep\!<\!(-\al)\|x\!-\!y\|$, or
(II{$_2$}) $(-\al)\|x\!-\!y\|\le\vep\!<\!\frac{1-\al}2\|x\!-\!y\|$.}
  \end{equation*}
In case (II${}_1$), $\al\|x-y\|+\vep<0$, which implies $A_\al=(-\al)\|x-y\|-\vep$.
Hence, the left-hand side inequality in (II${}_1$) gives $A_\al\le\vep$. In case (II${}_2$),
$\al\|x-y\|+\vep\ge0$, which implies $A_\al=\al\|x-y\|+\vep<\vep$ (because $\al<0$).

Now, suppose $\al\le-1$ (i.e., $\|x-y\|\le(-\al)\|x-y\|=\|x_\al-x\|$, which means that
$x_\al$ is farther from $x$ than $x$ from $y$), so that $\frac12(1-\al)\le(-\al)$.
In this case, assumption (II) implies only condition (II${}_1$), and so, as above,
$A_\al=(-\al)\|x-y\|-\vep\le\vep$. This completes the proof of $\|f-g\|_{\infty,I}\le\vep$.

For the variation $V(g,I)$ of function $g$ from \eq{e:gspe}, we have, by \eq{e:xay},
  \begin{align*}
V(g,I)&=\|(x_\al-\vep\mbox{\rm e}_{x,y})-(y+\vep\mbox{\rm e}_{x,y})\|
  =\|(x_\al-y)-2\vep\mbox{\rm e}_{x,y}\|\\[3pt]
&=\Bigl\|(1-\al)\|x-y\|\mbox{\rm e}_{x,y}-2\vep\mbox{\rm e}_{x,y}\Bigr\|
  =(1-\al)\|x-y\|-2\vep.
  \end{align*}
Hence $V_\vep(f,I)\le V(g,I)=(1-\al)\|x-y\|-2\vep$. Thus, we have shown that
  \begin{equation} \label{e:cas2}
V_\vep(f,I)=(1-\al)\|x-y\|-2\vep\quad\mbox{if}\quad 
\textstyle\frac{(-\al)}2\|x-y\|\le\vep<\frac{(1-\al)}2\|x-y\|.
  \end{equation}

\emph{Case\/ $\al>1$}. \label{p:a>1}
We reduce this case to the case $\al<0$ and apply Lemma~\ref{l:chvar}.
Set $T=[a',b']$ with $a'=2\tau-b$ and $b'=2\tau-a$, so that $a'<\tau<b'$, and define
$\vfi:T\to\Rb$ by $\vfi(t)=2\tau-t$, $a'\le t\le b'$. Clearly, $\vfi$ is strictly decreasing
on $T$, $\vfi(T)=[a,b]=I$, and $\vfi(\tau)=\tau$. Let us show that the composed
function $f'=f\circ\vfi\in M^T$ is of the same form as \eq{e:alp}.

If $a'\le t<\tau$, then $\tau<\vfi(t)\le b$, and so, by \eq{e:alp},
$f'(t)=f(\vfi(t))=y$; if $t=\tau$, then $f'(\tau)=f(\vfi(\tau))=f(\tau)=x_\al$;
and if $\tau<t\le b'$, then $a\le\vfi(t)<\tau$, and so, $f'(t)=f(\vfi(t))=x$.
Setting $x'=y$, $y'=x$, and $\al'=1-\al$, we get $\al'<0$,
  \begin{equation*}
\mbox{$f'(t)=x'$ \,if \,$a'\le t<\tau$,\quad $f'(t)=y'$ \,if \,$\tau<t\le b'$,}
  \end{equation*}
and
  \begin{equation*}
f'(\tau)=x_\al=(1-\al)x+\al y=\al'y'+(1-\al')x'=(1-\al')x'+\al'y'\equiv x'_{\al'}.
  \end{equation*}

By Lemma~\ref{l:chvar}, given $\vep>0$,
  \begin{equation*}
V_\vep(f,I)=V_\vep(f,[a,b])=V_\vep(f,\vfi(T))=V_\vep(f\circ\vfi,T)=V_\vep(f',[a',b']),
  \end{equation*}
where, since $f'$ is of the form \eq{e:alp}, $V_\vep(f',[a',b'])$ is given by
\eq{e:cas1}, \eq{e:cas2} and \eq{e:buda} with $f$, $x$, $y$, and $\al$ replaced by
$f'$, $x'$, $y'$, and $\al'$, respectively. Noting that $\|x'-y'\|=\|x-y\|$, $1-\al'=\al$,
$1-2\al'=2\al-1$, and $(-\al')=\al-1$, we get, for $\al>1$:
  \begin{equation*}
V_\vep(f,I)=\left\{
  \begin{tabular}{ccr}
$\!\!(2\al-1)\|x-y\|-4\vep$ & \mbox{if} &
   \mbox{$0<\,\vep<\frac{\al-1}2\|x-y\|$},\\[3pt]
$\!\!\al\|x-y\|-2\vep$ & \mbox{if} & 
  \mbox{$\frac{\al-1}2\|x-y\|\le\,\,\vep<\,\frac\al2\|x-y\|$,\,\,\,\,}\\[3pt]
$\!\!0$ & \mbox{if} & \mbox{$\,\,\vep\,\ge\,\frac\al2\|x-y\|$.\,\,\,}
  \end{tabular}\right.
  \end{equation*}
Finally, we note that, for $\al>1$, we have, by \eq{e:xax} and \eq{e:xay},
  \begin{equation*}
V(f,I)\!=\!\|x-x_\al\|+\|x_\al-y\|\!=\!\al\|x-y\|+(\al-1)\|x-y\|\!=\!(2\al-1)\|x-y\|,
  \end{equation*}
and so, $V_\vep(f,I)\to V(f,I)$ as $\vep\to+0$.
\end{example}

\section{The generalized Dirichlet function}

\begin{example}[generalized Dirichlet function] \label{ex:gDf} \rm
This is an illustration of Lem\-ma~\ref{l:Regc} illuminating several specific features
of the approximate variation.

(a) Let $T=I=[a,b]$, $(M,d)$ be a metric space, and $\Qb$ denote (as usual) the set
of all rational numbers. We set $I_1=I\cap\Qb$ and $I_2=I\setminus\Qb$. A function
$f\in M^I$ is said to be a \emph{generalized Dirichlet function\/} if $f\in\Bd(I;M)$ and
  \begin{equation*}
\Delta f\equiv\Delta f(I_1,I_2)=\inf_{s\in I_1,t\in I_2}d(f(s),f(t))>0.
  \end{equation*}

Clearly, $f\!\notin\!\Reg(I;M)$ (in fact, if, say, $a\!<\!\tau\!\le\! b$, then for all
\mbox{$\delta\!\in\!(0,\tau-a)$},
$s\in(\tau-\delta,\tau)\cap\Qb$ and $t\in(\tau-\delta,\tau)\setminus\Qb$, we have
$d(f(s),f(t))\ge\Delta f>0$).

Setting $|f(I_1,I_2)|=\sup_{s\in I_1,t\in I_2}d(f(s),f(t))$, we find
  \begin{equation*}
|f(I_1,I_2)|\le|f(I_1)|+d(f(s_0),f(t_0))+|f(I_2)|,\quad s_0\in I_1,\,\,t_0\in I_2,
  \end{equation*}
and
  \begin{equation*}
0<\Delta f\le|f(I_1,I_2)|\le|f(I)|=\max\{|f(I_1)|,|f(I_2)|,|f(I_1,I_2)|\}.
  \end{equation*}
Furthermore (cf.~Lemma~\ref{l:Regc}), we have
  \begin{equation} \label{e:Dirin}
\mbox{$V_\vep(f,I)=\infty$ \,\,if \,\,$0<\vep<\Delta f/2$, \,and 
\,\,$V_\vep(f,I)=0$ \,\,if \,\,$\vep\ge|f(I)|$;}
  \end{equation}
the values of $V_\vep(f,I)$ for $\Delta f/2\le\vep<|f(I)|$ depend on (the structure of)
the metric space $(M,d)$ in general (see items (b), (c) and (d) below). The second
assertion in \eq{e:Dirin} is a consequence of \eq{e:zero}.
In order to prove the first assertion in \eq{e:Dirin}, we show that if $0<\vep<\Delta f/2$,
$g\in M^I$ and $d_{\infty,I}(f,g)\le\vep$, then $V(g,I)=\infty$ (cf.~\eq{e:besk}).
In fact, given $n\in\Nb$, let $P=\{t_i\}_{i=0}^{2n}$ be a partition of $I$ (i.e.,
$a\le t_0<t_1<t_2<\dots<t_{2n-1}<t_{2n}\le b$) such that
$\{t_{2i}\}_{i=0}^n\subset I_1$ and $\{t_{2i-1}\}_{i=1}^n\subset I_2$. Given
$i\in\{1,2,\dots,n\}$, by the triangle inequality for $d$, we have
  \begin{align}
d(f(t_{2i}),f(t_{2i-1}))&\le d(f(t_{2i}),g(t_{2i}))\!+\!d(g(t_{2i}),g(t_{2i-1}))
  \!+\!d(g(t_{2i-1}),f(t_{2i-1})) \nonumber\\[3pt]
&\le d_{\infty,I_1}(f,g)+d(g(t_{2i}),g(t_{2i-1}))+d_{\infty,I_2}(g,f) \nonumber\\[3pt]
&\le\vep+d(g(t_{2i}),g(t_{2i-1}))+\vep. \label{e:twe1}
  \end{align}
It follows from the definition of $V(g,I)$ that
  \begin{align}
V(g,I)&\ge\sum_{i=1}^{2n}d(g(t_i),g(t_{i-1}))\ge\sum_{i=1}^nd(g(t_{2i}),g(t_{2i-1}))
  \nonumber\\
&\ge\sum_{i=1}^n\Bigl(d(f(t_{2i}),f(t_{2i-1}))-2\vep\Bigr)\ge (\Delta f-2\vep)n.
  \label{e:twe2}
  \end{align}
It remains to take into account the arbitrariness of $n\in\Nb$.

In a particular case of the classical \emph{Dirichlet function\/} $f=\Dc_{x,y}:I\to M$
defined, for $x,y\in M$, $x\ne y$, by
  \begin{equation} \label{e:Dir}
\mbox{$\Dc_{x,y}(t)=x$ \,\,if \,\,$t\in I_1$, \,and 
\,\,$\Dc_{x,y}(t)=y$ \,\,if \,\,$t\in I_2$,}
  \end{equation}
we have $\Delta f=\Delta\Dc_{x,y}=d(x,y)$ and $|f(I)|=|\Dc_{x,y}(I_1,I_2)|=d(x,y)$,
and so, \eq{e:Dirin} assumes the form
(which  was established in \cite[assertion~(4.4)]{Studia17}):
  \begin{equation} \label{e:Dirass}
\mbox{$V_\vep(f,I)\!=\!\infty$ \,\,if \,\,$0\!<\!\vep\!<\!d(x,y)/2$, \,and 
\,\,$V_\vep(f,I)\!=\!0$ \,\,if \,\,$\vep\!\ge\! d(x,y)$.}
  \end{equation}

(b) \label{p:L71b} This example and items (c) and (d) below illustrate the sharpness of
assertions in  Lemma~\ref{l:71}(b),\,(d). Let $(M,\|\cdot\|)$ be a normed linear space
with induced metric $d$ (cf.\ p.~\pageref{p:nls}) and $f=\Dc_{x,y}$ be the Dirichlet
function \eq{e:Dir}. Setting $c=c(t)=(x+y)/2$, $t\in I$, we find
  \begin{equation*}
2d_{\infty,I}(f,c)=2\max\{\|x-c\|,\|y-c\|\}=\|x-y\|=d(x,y)=|f(I)|,
  \end{equation*}
and so, by \eq{e:zero2} and \eq{e:2max}, the second equality in \eq{e:Dirass}
is refined as follows:
  \begin{equation} \label{e:refi}
V_\vep(f,I)=0\quad\mbox{for all}\quad\vep\ge\frac{\|x-y\|}2=\frac{d(x,y)}2.
  \end{equation}
This shows the sharpness of the inequality in Lemma~\ref{l:71}(b). Inequalities in
Lemma~\ref{l:71}(d) assume the form:
  \begin{equation*}
\inf_{\vep>0}(V_\vep(f,I)+\vep)=\frac{\|x-y\|}2<|f(I)|=\|x-y\|=
\inf_{\vep>0}(V_\vep(f,I)+2\vep).
  \end{equation*}

More generally, \eq{e:Dirass} and \eq{e:refi} hold for a complete and \emph{metrically
convex\/} (in the sense of K.~Menger \cite{Menger}) metric space $(M,d)$ (see
\cite[Example~1]{Studia17}).

(c) In the context of \eq{e:Dir}, assume that $M=\{x,y\}$ is the two-point set with
metric $d$. If $0<\vep<d(x,y)$, $g\in M^I$ and $d_{\infty,I}(f,g)\le\vep$, then
$g=f=\Dc_{x,y}$ on $I$, and so, $V(g,I)=\infty$. By \eq{e:besk}, the first assertion
in \eq{e:Dirass} can be expressed more exactly as $V_\vep(f,I)=\infty$ for all
$0<\vep<d(x,y)$. Now, (in)equalities in Lemma~\ref{l:71}(d) are of the form:
  \begin{equation*}
\inf_{\vep>0}(V_\vep(f,I)+\vep)=d(x,y)=|f(I)|<2d(x,y)=
\inf_{\vep>0}(V_\vep(f,I)+2\vep).
  \end{equation*}

(d) \label{p:L71d} Given $x,y\in\Rb$, $x\ne y$, and $0\le r\le|x-y|/2$, we set
  \begin{equation*}
M_r=\Rb\setminus\bigl(\textstyle\frac12(x+y)-r,\frac12(x+y)+r\bigr)
\quad\mbox{and}\quad d(u,v)=|u-v|,\,\,u,v\in M_r.
  \end{equation*}
Note that $(M_r,d)$ is a proper metric space (cf.~p.~\pageref{p:properms}). If
$f=\Dc_{x,y}:I\to M_r$ is the Dirichlet function \eq{e:Dir} on $I$, we claim that
  \begin{equation} \label{e:Mr}
\mbox{$V_\vep(f,I)=\infty$ \,\,if \,\,$0<\vep<\frac12|x-y|+r$, \,and 
\,$V_\vep(f,I)=0$ \,otherwise.}
  \end{equation}

\emph{Proof of~\eq{e:Mr}}. Since $M_0=\Rb$, assertion \eq{e:Mr} for $r=0$ follows
from \eq{e:Dirass} and \eq{e:refi}. Now, suppose $r>0$. From \eq{e:Dirass}, we find
$V_\vep(f,I)=\infty$ if $0<\vep<\frac12|x-y|$, and $V_\vep(f,I)=0$ if $\vep\ge|x-y|$.
So, only the case when $\frac12|x-y|\le\vep<|x-y|$ is to be considered. We split this
case into two subcases:
  \begin{equation*}
\mbox{(I) $\frac12|x-y|\le\vep<\frac12|x-y|+r$, and
(II) $\frac12|x-y|+r\le\vep<|x-y|$.}
  \end{equation*}

\emph{Case}~(I). Let us show that if $g:I\to M_r$ and $d_{\infty,I}(f,g)\le\vep$, then
$V(g,I)=\infty$. Given $t\in I=I_1\cup I_2$, the inclusion $g(t)\in M_r$ is equivalent to
  \begin{equation} \label{e:gtimr}
g(t)\le\textstyle\frac12(x+y)-r\quad\mbox{or}\quad g(t)\ge\textstyle\frac12(x+y)+r,
  \end{equation}
and condition $d_{\infty,I}(f,g)=|f-g|_{\infty,I}\le\vep$ is equivalent to
  \begin{equation} \label{e:difg}
\mbox{$|x-g(s)|\le\vep$ \,$\forall\,s\in I_1$, \,and \,$|y-g(t)|\le\vep$ \,$\forall\,t\in I_2$.}
  \end{equation}
Due to the symmetry in $x$ and $y$ everywhere, we may assume that $x<y$.

Suppose $s\in I_1$. The first condition in \eq{e:difg} and assumption (I) imply
  \begin{equation*}
x-\vep\le g(s)\le x+\vep<\textstyle x+\frac12|x-y|+r=x+\frac12(y-x)+r=\frac12(x+y)+r,
  \end{equation*}
and so, by \eq{e:gtimr}, we find $g(s)\le\frac12(x+y)-r$. Note that, by (I),
  \begin{equation*}
-\vep\le g(s)-x\le\textstyle\frac12(x+y)-r-x=\frac12(y-x)-r=\frac12|y-x|-r\le\vep-r<\vep.
  \end{equation*}

Given $t\in I_2$, the second condition in \eq{e:difg} and assumption (I) yield
  \begin{equation*}
y+\vep\ge g(t)\ge y-\vep>\textstyle y-\frac12|x-y|-r=y-\frac12(y-x)-r=\frac12(x+y)-r,
  \end{equation*}
and so, by \eq{e:gtimr}, we get $g(t)\ge\frac12(x+y)+r$. Note also that, by (I),
  \begin{equation*}
\vep\ge g(t)-y\ge\textstyle\frac12(x+y)+r-y=\frac12(x-y)+r=-\frac12|x-y|+r
\ge-\vep+r>-\vep.
  \end{equation*}

Thus, we have shown that, given $s\in I_1$ and $t\in I_2$,
  \begin{equation} \label{e:ggts}
g(t)-g(s)\ge\textstyle\frac12(x+y)+r-\bigl(\frac12(x+y)-r\bigr)=2r.
  \end{equation}
Given $n\in\Nb$, let $\{t_i\}_{i=0}^{2n}$ be a partition of $I$ such that
$\{t_{2i}\}_{i=0}^n\subset I_1$ and $\{t_{2i-1}\}_{i=1}^n\subset I_2$. Taking into
account \eq{e:ggts} with $s=t_{2i}$ and $t=t_{2i-1}$, we get
  \begin{equation*}
V(g,I)\ge\sum_{i=1}^{2n}|g(t_i)-g(t_{i-1})|\ge\sum_{i=1}^n\bigl(
g(t_{2i-1})-g(t_{2i})\bigr)\ge2rn.
  \end{equation*}

\emph{Case}~(II). We set $c=c(t)=\vep+\min\{x,y\}$, $t\in I$; under our
 assumption
$x<y$, we have $c=\vep+x$. Note that $c\in M_r$: in fact, (II) and $x<y$ imply
$\frac12(y-x)+r\le\vep<y-x$, and so, $\frac12(x+y)+r\le c=\vep+x<y$.
If $s\in I_1$, we find $|x-c(s)|=\vep$, and if $t\in I_2$, we get, by assumption~(II),
  \begin{equation*}
|y-c(t)|=|y-x-\vep|=y-x-\vep\le|x-y|-\textstyle\frac12|x-y|-r\le\frac12|x-y|+r\le\vep.
  \end{equation*}
It follows that (cf.~\eq{e:difg}) $d_{\infty,I}(f,c)\le\vep$, and since $c$ is constant
on $I$, we conclude from \eq{e:ze1} that $V_\vep(f,I)=0$. This completes
the proof of \eq{e:Mr}.
\sq

Two conclusions from \eq{e:Mr} are in order. First, given $0\le r\le\frac12|x-y|$ and
$\vep>0$, $V_\vep(f,I)=0$ if and only if $|f(I)|=|x-y|\le2\vep-2r$ (cf.\ \eq{e:zero}
and Lemma~\ref{l:71}(b)). Second, the inequalities in Lemma~\ref{l:71}(d) are
as follows:
  \begin{align*}
\inf_{\vep>0}(V_\vep(f,I)+\vep)&=\textstyle\frac12|x-y|+r\le|f(I)|=|x-y|\\[2pt]
&\le|x-y|+2r=\inf_{\vep>0}(V_\vep(f,I)+2\vep).
  \end{align*}
The inequalities at the left and at the right become equalities for $r=\frac12|x-y|$ and
$r=0$, respectively; otherwise, the mentioned inequalities are strict.
\label{p:36d}

(e) \label{p:rico} Let $x,y\in\Rb$, $x\ne y$, and $0\le r<|x-y|/2$. We set
  \begin{equation*}
M_r=\Rb\setminus\bigl[\textstyle\frac12(x+y)-r,\frac12(x+y)+r\bigr]
\quad\mbox{and}\quad d(u,v)=|u-v|,\,\,u,v\in M_r.
  \end{equation*}
Note that $(M_r,d)$ is an improper metric space. For the
Dirichlet function $f=\Dc_{x,y}:I\to M_r$ from \eq{e:Dir}, we have:
  \begin{equation} \label{e:Mrim}
\mbox{$V_\vep(f,I)=\infty$ \,\,if \,\,$0<\vep\le\frac12|x-y|+r$, \,and 
\,$V_\vep(f,I)=0$ \,otherwise.}
  \end{equation}
Clearly, the function $\vep\mapsto V_\vep(f,I)$ is not continuous from the right at
$\vep=\frac12|x-y|+r$ (cf.\ Lemma~\ref{l:proper}(a)). The proof of \eq{e:Mrim}
follows the same lines as those of \eq{e:Mr}, so we present only the necessary
modifications. We split the case when $\frac12|x-y|\le\vep<|x-y|$ into two subcases:
  \begin{equation*}
\mbox{(I) $\frac12|x-y|\le\vep\le\frac12|x-y|+r$, and
(II) $\frac12|x-y|+r<\vep<|x-y|$.}
  \end{equation*}

\emph{Case}~(I). Given $g:I\to M_r$ with $d_{\infty,I}(f,g)\le\vep$, to see that
$V(g,I)=\infty$, we have \emph{strict\/} inequalities in \eq{e:gtimr}, conditions
\eq{e:difg}, and assume that $x<y$. If $s\in I_1$, then (as above)
$g(s)\le\frac12(x+y)+r$, and so, by (strict) \eq{e:gtimr}, $g(s)<\frac12(x+y)-r$.
If $t\in I_2$, then $g(t)\ge\frac12(x+y)-r$, and so, by \eq{e:gtimr},
$g(t)>\!\frac12(x+y)+r$. Thus, $g$ is discontinuous at every point of $I=I_1\cup I_2$,
and so, $V(g,I)=\infty$ (in fact, if, on the contrary, $g$ is continuous at, say,
a point $s\in I_1$, then the inequality $g(s)\!<\!\frac12(x+y)-r$ holds in a
neighbourhood~of~$s$, and since the neighbourhood contains an irrational point
$t\in I_2$, we get $g(t)>\!\frac12(x+y)+r$, which is a contradiction; recall also that
a $g\in\BV(I;\Rb)$ is continuous on $I$ apart, possibly, an at most countable subset
of $I$).

\emph{Case}~(II). It is to be noted only that $\frac12(x+y)+r<c=\vep+x<y$, and so,
$c\in M_r$; in fact, by (II) and assumption $x<y$, $\frac12(y-x)+r<\vep<y-x$. \qed
\end{example}

\section{Examples with convergent sequences} \label{ss:ecnv}

\begin{example} \label{ex:rieq} \rm
The left limit $V_{\vep-0}(f,T)$ in Lemma~\ref{l:uc}(a) cannot, in general, be replaced
by $V_\vep(f,T)$. To see this, we let $T=I$, $(M,\|\cdot\|)$ be a normed linear space,
$\{x_j\}$, $\{y_j\}\subset M$ be two sequences, $x,y\in M$, $x\ne y$, and
$x_j\to x$ and $y_j\to y$ in $M$ as $j\to\infty$. If $f_j=\Dc_{x_j,y_j}$, $j\in\Nb$,
and $f=\Dc_{x,y}$ are Dirichlet functions \eq{e:Dir} on $I$, then $f_j\rra f$ on $I$,
which follows from
  \begin{equation*}
\|f_j-f\|_{\infty,I}=\max\{\|x_j-x\|,\|y_j-y\|\}\to0\quad\mbox{as}\quad j\to\infty.
  \end{equation*}
The values $V_\vep(f,I)$ are given by \eq{e:Dirass} and \eq{e:refi}, and, similarly,
if $j\in\Nb$,
  \begin{equation} \label{e:steen}
\mbox{$V_\vep(f_j,I)=\infty$ \,if \,$0<\vep<\frac12\|x_j-y_j\|$, \,\,\,
$V_\vep(f_j,I)=0$ \,if \,$\vep\ge\frac12\|x_j-y_j\|$.}
  \end{equation}
Setting $\vep=\frac12\|x-y\|$, $\al_j=1+(1/j)$, $x_j=\al_jx$ and $y_j=\al_jy$,
$j\in\Nb$, we find
  \begin{equation*}
V_{\vep+0}(f,I)=V_\vep(f,I)=0<\infty=V_{\vep-0}(f,I),
  \end{equation*}
whereas, since $\vep<\frac12\al_j\|x-y\|=\frac12\|x_j-y_j\|$ for all $j\in\Nb$,
  \begin{equation*}
\mbox{$V_\vep(f_j,I)=\infty$ for all $j\in\Nb$, and so,
$\D\lim_{j\to\infty}V_\vep(f_j,I)=\infty$.}
  \end{equation*}
\end{example}

\begin{example} \label{ex:voo} \rm
The right-hand side inequality in Lemma~\ref{l:uc}(a) may not hold if $\{f_j\}\subset M^T$
converges to $f\in M^T$ only \emph{pointwise\/} on $T$. To see this, suppose
$C\equiv\inf_{j\in\Nb}|f_j(T)|>0$ and $f=c$ (is a constant function) on~$T$. Given
$0<\vep<C/2$, Lemma~\ref{l:71}(f) implies
  \begin{equation*}
\mbox{$V_\vep(f_j,T)\ge|f_j(T)|-2\vep\ge C-2\vep>0=V_\vep(c,T)=V_\vep(f,T)$,
\quad $j\in\Nb$.}
  \end{equation*}
For instance, given a sequence $\{\tau_j\}\subset(a,b)\subset I=[a,b]$ such that
$\tau_j\to a$ as $j\to\infty$, and $x,y\in M$, $x\ne y$, defining $\{f_j\}\subset M^I$
(as in Example~\ref{ex:thr}) by $f_j(\tau_j)=x$ and $f_j(t)=y$ if
$t\in I\setminus\{\tau_j\}$, $j\in\Nb$, we have $C=d(x,y)>0$ and
$f_j\to c\equiv y$ \pw\ on~$I$.

The arguments above are not valid for the uniform convergence: in fact, if
$f_j\rra f=c$ on $T$, then, by \eq{e:1s2}, $|f_j(T)|\le2d_{\infty,T}(f_j,c)\to0$ as
$j\to\infty$, and so, $C=0$.
\end{example}

\begin{example} \label{ex:ucbw} \rm
Lemma~\ref{l:uc}(b) is wrong for the \pw\ convergence \mbox{$f_j\to f$}.
To see this, let $T=I=[a,b]$, $(M,d)$ be a metric space,
$x,y\in M$, $x\ne y$, and, given $j\in\Nb$, define
$f_j\in M^I$ at $t\in I$ by: $f_j(t)=x$ if $j!t$ is integer, and $f_j(t)=y$ otherwise.
Each $f_j$ is a step function on $I$, so it is regulated and, hence, by Lemma~\ref{l:Regc},
$V_\vep(f_j,I)<\infty$ for all $\vep>0$. At the same time, the sequence $\{f_j\}$
converges (only) pointwise on $I$ to the Dirichlet function $f=\Dc_{x,y}$ (cf.~\eq{e:Dir}),
and so, by \eq{e:Dirass}, $V_\vep(f,I)=\infty$ for all $0<\vep<\frac12d(x,y)$.
\end{example}

\section{Examples with improper metric spaces} \label{ss:exims}

\pagebreak
\begin{example} \label{ex:ims1} \rm
This example is similar to Example~\ref{ex:gDf}(e) (p.~\pageref{p:rico}), but with
\emph{finite\/} values of $V_\vep(f,I)$. It shows that the assumption on the
\emph{proper\/} metric space $(M,d)$ in Lemma~\ref{l:proper}(b) is essential.

Let $x,y\in\Rb$, $x\ne y$, $M=\Rb\setminus\{\frac12(x+y)\}$ with metric
$d(u,v)=|u-v|$ for $u,v\in M$, $I=[a,b]$, $\tau=a$ or $\tau=b$, and $f\in M^I$ be
given by (cf.\ \eq{e:ftau}): $f(\tau)=x$ and $f(t)=y$ if $t\in I$, $t\ne\tau$.
We claim that (as in \eq{e:tab})
  \begin{equation} \label{e:511}
V_\vep(f,I)=\left\{
  \begin{tabular}{ccr}
$\!\!|x-y|-2\vep$ & \mbox{if} & $0<\vep<\frac12|x-y|$,\\[3pt]
$\!\!0$ & \mbox{if} & $\vep\ge\frac12|x-y|$.
  \end{tabular}\right.
  \end{equation}

In order to verify this, we note that $|f(I)|=|x-y|$, and so, by \eq{e:zero},
$V_\vep(f,I)=0$ for all $\vep\ge|x-y|$. We split the case $0<\vep<|x-y|$ into
  \begin{equation*}
\mbox{(I) $0<\vep<\frac12|x-y|$; \,(II) $\vep=\frac12|x-y|$;
\,(III) $\frac12|x-y|<\vep<|x-y|$.}
  \end{equation*}
Due to the symmetry (in $x$ and $y$), we may consider only the case $x<y$.

\emph{Case\/}~(I). Given $g\in M^I$ with $d_{\infty,I}(f,g)\le\vep$, inequality
\eq{e:10} implies
  \begin{equation*}
V(g,I)\ge|g(t)-g(\tau)|\ge|f(t)-f(\tau)|-2\vep=|x-y|-2\vep\quad(t\ne\tau),
  \end{equation*}
and so, by \eq{e:av}, $V_\vep(f,I)\ge|x-y|-2\vep$. Now, following \eq{e:unve}, we set
  \begin{equation} \label{e:givep}
g_\vep(\tau)=x+\vep\quad\mbox{and}\quad g_\vep(t)=y-\vep\,\,\,\,\mbox{if}\,\,\,\,
t\in I\setminus\{\tau\}.
  \end{equation}
We have $g_\vep:I\to M$, because assumption $0<\vep<\frac12(y-x)$ yields
  \begin{equation*}
g_\vep(\tau)=x+\vep<x+\textstyle\frac12(y-x)=\frac12(x+y)
  \end{equation*}
and, if $t\in I$, $t\ne\tau$,
  \begin{equation*}
g_\vep(t)=y-\vep>y-\textstyle\frac12(y-x)=\frac12(x+y).
  \end{equation*}
Moreover, $d_{\infty,I}(f,g_\vep)=\vep$ and
  \begin{equation*}
V(g_\vep,I)=|g_\vep(I)|=|(y-\vep)-(x+\vep)|\stackrel{\mbox{\tiny(I)}}{=}
y-x-2\vep=|x-y|-2\vep.
  \end{equation*}
Hence $V_\vep(f,I)\!\le\! V(g_\vep,I)\!=\!|x\!-\!y|\!-\!2\vep$.
This proves the upper line in \eq{e:511}.

\emph{Case\/}~(II). Here we rely on the full form of \eq{e:zer}. Let a sequence
$\{\vep_k\}_{k=1}^\infty$ be such that $0<\vep_k<\vep=\frac12|x-y|$ for all
$k\in\Nb$ and $\vep_k\to\vep$ as $k\to\infty$. We set $g_k=g_{\vep_k}$, $k\in\Nb$,
where $g_{\vep_k}$ is defined in \eq{e:givep} (with $\vep=\vep_k$). By Case~(I),
given $k\in\Nb$, $g_k\in\BV(I;M)$, $V(g_k,I)=|x-y|-2\vep_k$ and
$d_{\infty,I}(f,g_k)=\vep_k<\vep$. Since $V(g_k,I)\to0$ as $k\to\infty$, we conclude
from \eq{e:zer} that $V_\vep(f,I)=0$.

\emph{Case\/}~(III). We set $c(t)=\vep+\min\{x,y\}$, $t\in I$, and argue as in
Example \ref{ex:gDf}(e) (in~Case (II) for $r=0$). This gives $V_\vep(f,I)=0$, and
completes the proof of \eq{e:511}.

Clearly, the metric space $(M,d)$ in this example is \emph{not proper}. Let us show that
Lemma~\ref{l:proper}(b) is wrong. In fact, by contradition, assume that there is
$g\in\BV(I;M)$ with $d_{\infty,I}(f,g)\le\vep=\frac12|x-y|$ such that
$V_\vep(f,I)=V(g,I)$. By \eq{e:511}, $V(g,I)=0$, and so, $g=c$ is a constant
function $c:I\to M$. From $d_{\infty,I}(f,c)\le\vep$, we find $|x-c|=|f(\tau)-c|\le\vep$,
and so (as above, $x<y$),
  \begin{equation*}
c\le x+\vep=x+\textstyle\frac12(y-x)=\frac12(x+y),
  \end{equation*}
and, if $t\ne\tau$, then $|y-c|=|f(t)-c|\le\vep$, which implies
  \begin{equation*}
c\ge y-\vep=y-\textstyle\frac12(y-x)=\frac12(x+y).
  \end{equation*}
Hence, $c=c(t)=\frac12(x+y)$, $t\in I$, but $g=c\notin M^I$, which is a contradiction.
\end{example}

\begin{example} \label{ex:ims2} \rm
Here we show that the assumption that the metric space $(M,d)$ is \emph{proper\/} in
Lemma~\ref{l:proper}(c) is essential.

Let $x,y\in\Rb$, $x\ne y$, $M=\Rb\setminus\{\frac12(x+y)\}$ with metric
$d(u,v)=|u-v|$, $u,v\in M$, $I=[0,1]$, and the sequence $\{f_j\}\subset M^I$ be
given by
  \begin{equation} \label{e:seqfj}
f_j(t)=\left\{
  \begin{tabular}{ccl}
$\!\!x$ & \mbox{if} & $j!t$ is integer,\\[3pt]
$\!\!y$ & & otherwise,
  \end{tabular}\right.\,\,\quad t\in I,\,\,\,j\in\Nb.
  \end{equation}
We claim that, for all $j\in\Nb$,
  \begin{equation} \label{e:jfac}
V_\vep(f_j,I)=\left\{
  \begin{tabular}{ccr}
$\!\!2\!\cdot\! j!\,(|x-y|-2\vep)$ & \mbox{if} & $0<\vep<\frac12|x-y|$,\\[3pt]
$\!\!0$ & \mbox{if} & $\vep\ge\frac12|x-y|$.
  \end{tabular}\right.
  \end{equation}

Suppose that we have already established \eq{e:jfac}. The sequence $\{f_j\}$ from
\eq{e:seqfj} converges pointwise on $I$ to the Dirichlet function $f=\Dc_{x,y}$ from
\eq{e:Dir}. Let $\vep=\frac12|x-y|$. By \eq{e:Mrim} with $r=0$, we have
$V_\vep(f,I)=\infty$, while, by \eq{e:jfac}, we get $V_\vep(f_j,I)=0$ for all $j\in\Nb$,
and so, $\lim_{j\to\infty}V_\vep(f_j,I)=0$. Thus, the properness of metric space
$(M,d)$ in Lemma~\ref{l:proper}(c) is indispensable.

\emph{Proof of\/}~\eq{e:jfac}.
In what follows, we fix $j\in\Nb$. By \eq{e:zero}, $V_\vep(f_j,I)=0$ for all
$\vep\ge|f_j(I)|=|x-y|$. Now, we consider cases (I)--(III) from Example~\ref{ex:ims1}.

\emph{Case\/}~(I). We set $t_k=k/j!$ (so that $f_j(t_k)=x$) for $k=0,1,\dots,j!$, and
$s_k=\frac12(t_{k-1}+t_k)=(k-\frac12)/j!$ (so that $f_j(s_k)=y$) for $k=1,2,\dots,j!$.
So, we have the following partition of the interval $I=[0,1]$:
  \begin{equation} \label{e:partj}
0=t_0<s_1<t_1<s_2<t_2\dots<s_{j!-1}<t_{j!-1}<s_{j!}<t_{j!}=1.
  \end{equation}
If $g\in M^I$ is arbitrary with $d_{\infty,I}(f_j,g)\le\vep$, then, applying \eq{e:10}, we get
  \begin{align*}
V(g,I)&\ge\,\sum_{k=1}^{j!}\bigl(|g(t_k)-g(s_k)|+|g(s_k)-g(t_{k-1})|\bigr)\\
&\ge\,\sum_{k=1}^{j!}\bigl(|f_j(t_k)-f_j(s_k)|-2\vep+|f_j(s_k)-f_j(t_{k-1})|-2\vep\bigr)
  \\[3pt]
&=\,2\!\cdot\!j!\,(|x-y|-2\vep),
  \end{align*}
and so, by definition \eq{e:av}, $V_\vep(f_j,I)\ge2\!\cdot\!j!\,(|x-y|-2\vep)$.

Now, we define a test function $g_\vep$ on $I$ by (cf.~\eq{e:unve}): given $t\in I$,
 \begin{equation} \label{e:geep}
\mbox{$g_\vep(t)\!=\!x\!-\!\vep\mbox{\rm e}_{x,y}$\,\,\,if\, $j!t$ is integer, and\,
$g_\vep(t)\!=\!y\!+\!\vep\mbox{\rm e}_{x,y}$\,\,\,otherwise,}
  \end{equation}
where $\mbox{\rm e}_{x,y}=(x-y)/|x-y|$. Due to the symmetry in $x$ and $y$,
we may assume that $x<y$, and so, $g_\vep(t)=x+\vep$ if $j!t$ is integer, and
$g_\vep(t)=y-\vep$ otherwise, $t\in I$. We first note that $g_\vep:I\to M$; in fact,
if $j!t$ is integer, then
  \begin{equation*}
g_\vep(t)=x+\vep<x+\textstyle\frac12|x-y|=x+\frac12(y-x)=\frac12(x+y),
  \end{equation*}
and if $j!t$ is not integer, then
  \begin{equation*}
g_\vep(t)=y-\vep>y-\textstyle\frac12|x-y|=y-\frac12(y-x)=\frac12(x+y).
  \end{equation*}
Clearly, $d_{\infty,I}(f_j,g_\vep)=\vep$ and, by the additivity of $V$, for the partition
\eq{e:partj}, we find
  \begin{align*}
V(g_\vep,I)&=\sum_{k=1}^{j!}\bigl(V(g_\vep,[t_{k-1},s_k])+V(g_\vep,[s_k,t_k])\bigr)\\
&=\sum_{k=1}^{j!}\bigl(|g_\vep(s_k)-g_\vep(t_{k-1})|+|g_\vep(t_k)-g_\vep(s_k)|\bigr)\\
&=\sum_{k=1}^{j!}\bigl(|(y-\vep)-(x+\vep)|+|(x+\vep)-(y-\vep)|\bigr)\\[3pt]
&=2\!\cdot\!j!\,(|x-y|-2\vep).
  \end{align*}
Thus, $V_\vep(f_j,I)\le V(g_\vep,I)$, and this implies the upper line in \eq{e:jfac}.

\emph{Case\/}~(II). Let a sequence $\{\vep_k\}_{k=1}^\infty$ be such that
$0<\vep_k<\vep=\frac12|x-y|$, $k\in\Nb$, and $\vep_k\to\vep$ as $k\to\infty$.
Set $g_k=g_{\vep_k}$, $k\in\Nb$, where $g_{\vep_k}$ is given by \eq{e:geep}
(with $x<y$). We know from Case~(I) that, for every $k\in\Nb$, $g_k\in\BV(I;M)$,
$V(g_k,I)=2\!\cdot\!j!\,(|x-y|-2\vep_k)$, and $d_{\infty,I}(f_j,g_k)=\vep_k<\vep$.
Since $V(g_k,I)\to0$ as $k\to\infty$, we conclude from \eq{e:zer} that $V_\vep(f_j,I)=0$.

\emph{Case\/}~(III). We set $c=c(t)=\vep+\min\{x,y\}$ for all $t\in I$, i.e., under
our assumption $x<y$, $c=\vep+x$. Note that $c\in M$, because assumption (III) and 
$x<y$ imply $\vep>\frac12(y-x)$, and so, $c=\vep+x>\frac12(y-x)+x=\frac12(x+y)$.
Furthermore, $d_{\infty,I}(f_j,c)\le\vep$; in fact, given $t\in I$, if $j!t$ is integer, then
$|f_j(t)-c(t)|=|x-c|=\vep$, and if $j!t$ is not integer, then
  \begin{equation*}
|f_j(t)-c(t)|=|y-x-\vep|\stackrel{\mbox{\tiny(III)}}{=}y-x-\vep<
|x-y|-\textstyle\frac12|x-y|=\frac12|x-y|<\vep.
  \end{equation*}
Since $c$ is a constant function from $M^I$, we get $V_\vep(f_j,I)=0$.
This completes the proof of \eq{e:jfac}.
\end{example}

\chapter{Pointwise selection principles} \label{s:sp}

\section{Functions with values in a metric space} \label{ss:metsp}

Our first main result,  an extension of Theorem~3.8 from \cite{Fr}, is a
\emph{\pw\ selection principle\/} for metric space valued univariate functions in terms
of the approximate variation (see Theorem~\ref{t:SP}). 

In order to formulate it, we slightly generalize the notion of a regulated function
(cf.~p.~\pageref{p:reg}). If $T\subset\Rb$ is an arbitrary set and $(M,d)$ is a metric
space, a function $f\in M^T$ is said to be \emph{regulated\/} on $T$ (in symbols,
$f\in\Reg(T;M)$) if it satisfies the Cauchy condition at every left limit point of $T$ and
every right limit point of $T$. More explicitly, given $\tau\in T$, which is a \emph{left
limit point\/} of $T$ (i.e., $T\cap(\tau-\delta,\tau)\ne\es$ for all $\delta>0$), we have
$d(f(s),f(t))\to0$ as $T\ni s,t\to\tau-0$; and given $\tau'\in T$, which is a \emph{right
limit point\/} of $T$ (i.e., $T\cap(\tau',\tau'+\delta)\ne\es$ for all $\delta>0$), we have
$d(f(s),f(t))\to0$ as $T\ni s,t\to\tau'+0$. The proof of Lemma~\ref{l:Regc} in ($\supset$)
shows that
  \begin{equation} \label{e:supReg}
\Reg(T;M)\supset\{f\in M^T:\mbox{$V_\vep(f,T)<\infty$ for all $\vep>0$}\};
  \end{equation}
it suffices to set $\vfi_\vep(t)=V_\vep(f,T\cap(-\infty,t])$, $t\in T$, and treat
$s,t$ from~$T$.

 In contrast to the case when $T=I$ is an
interval (see p.~\pageref{p:reg}), a function $f\in\Reg(T;M)$ may not be bounded
in general: for instance, $f\in\Rb^T$ given on
$T=[0,1]\cup\{2-\frac1n\}_{n=2}^\infty$ by: $f(t)=t$ if $0\le t\le1$ and
$f(2-\frac1n)=n$ if $n\in\Nb$, is regulated in the above sense, but not bounded.

In what follows, we denote by $\Mon(T;\Rb^+)$ the set of all bounded nondecreasing
functions mapping $T$ into $\Rb^+=[0,\infty)$ ($\Rb^+$ may be replaced by~$\Rb$).
It is worthwhile to recall the classical \emph{Helly selection principle\/} for an arbitrary
set $T\subset\Rb$ (e.g., \cite[Proof of Theorem~1.3]{Sovae}): \emph{a uniformly
bounded sequence of functions from $\Mon(T;\Rb)$ contains a subsequence which
converges \pw\ on $T$ to a function from $\Mon(T;\Rb)$.} \label{p:Hellym}

\begin{theorem} \label{t:SP}
Let $\es\ne T\subset\Rb$ and $(M,d)$ be a metric space. If $\{f_j\}\subset M^T$ is a
\pw\ \rc\ sequence of functions on $T$ such that
  \begin{equation} \label{e:sp}
\limsup_{j\to\infty}V_\vep(f_j,T)<\infty\quad\mbox{for all}\quad\vep>0,
  \end{equation}
then there is a subsequence of $\{f_j\}$, which converges \pw\ on $T$ to a
{\sl bounded regulated} function $f\in M^T$. In addition, if $(M,d)$ is proper, then
$V_\vep(f,T)$ does not exceed the $\limsup$ in\/ \eq{e:sp} for all $\vep>0$.
\end{theorem}

\proof
We present a direct proof based only on the properties of the approximate variation
from Section~\ref{ss:pro} (an indirect proof, based on the notion of the \emph{joint
modulus of variation of two functions}, was given in \cite[Theorem~3]{Studia17}).

By Lemma~\ref{l:ele}(b), given $\vep>0$ and $j\in\Nb$, the $\vep$-variation function
defined by the rule $t\mapsto V_\vep(f_j,T\cap(-\infty,t])$ is nondecreasing on $T$.
Note also that, by assumption \eq{e:sp}, for each $\vep>0$ there are $j_0(\vep)\in\Nb$
and a number $C(\vep)>0$ such that $V_\vep(f_j,T)\le C(\vep)$ for all $j\ge j_0(\vep)$.

We divide the rest of the proof into five steps.

1. Let us show that for each decreasing sequence $\{\vep_k\}_{k=1}^\infty$ of
positive numbers $\vep_k\to0$ there are a subsequence of $\{f_j\}$, again denoted by 
$\{f_j\}$, and a sequence of functions $\{\vfi_k\}_{k=1}^\infty\subset\Mon(T;\Rb^+)$
such that
  \begin{equation} \label{e:SP1}
\lim_{j\to\infty}V_{\vep_k}(f_j,T\cap(-\infty,t])=\vfi_k(t)\quad
\mbox{for \,all \,$k\in\Nb$ \,and \,$t\in T$.}
  \end{equation}

In order to prove \eq{e:SP1}, we make use of the Cantor diagonal procedure.
Lemma~\ref{l:ele}(b) and remarks above imply
  \begin{equation*}
\mbox{ $V_{\vep_1}(f_j,T\cap(-\infty,t])\le V_{\vep_1}(f_j,T)\le C(\vep_1)$
for all $t\in T$ and $j\ge j_0(\vep_1)$,}
  \end{equation*}
 i.e., the sequence of functions $\{t\mapsto V_{\vep_1}(f_j,T\cap(-\infty,t])\}%
_{j=j_0(\vep_1)}^\infty\subset\Mon(T;\Rb^+)$ is uniformly bounded on $T$ by
constant $C(\vep_1)$. By the classical Helly selection principle (for monotone functions),
there are a subsequence $\{J_1(j)\}_{j=1}^\infty$ of $\{j\}_{j=j_0(\vep_1)}^\infty$
and a function $\vfi_1\in\Mon(T;\Rb^+)$ such that
$V_{\vep_1}(f_{J_1(j)},T\cap(-\infty,t])$ converges to $\vfi_1(t)$ in $\Rb$ as
$j\to\infty$ for all $t\in T$. Now, choose the least number $j_1\in\Nb$ such that
$J_1(j_1)\ge j_0(\vep_2)$. Inductively, assume that $k\in\Nb$, $k\ge2$, and a
subsequence $\{J_{k-1}(j)\}_{j=1}^\infty$ of $\{j\}_{j=j_0(\vep_1)}^\infty$ and
the number $j_{k-1}\in\Nb$ with $J_{k-1}(j_{k-1})\ge j_0(\vep_k)$ are already
constructed. By Lemma~\ref{l:ele}(b), we get
  \begin{equation*}
\mbox{ $V_{\vep_k}(f_{J_{k-1}(j)},T\cap(-\infty,t])\!\le\!
V_{\vep_k}(f_{J_{k-1}(j)},T)\!\le\! C(\vep_k)$
for all $t\!\in\! T$ and $j\!\ge\! j_{k-1}$,}
  \end{equation*}
and so, by the Helly selection principle, there are a subsequence $\{J_k(j)\}_{j=1}^\infty$
of the sequence $\{J_{k-1}(j)\}_{j=j_{k-1}}^\infty$ and a function 
$\vfi_k\in\Mon(T;\Rb^+)$ such that
  \begin{equation*}
\lim_{j\to\infty}V_{\vep_k}(f_{J_k(j)},T\cap(-\infty,t])=\vfi_k(t)\quad
\mbox{for all}\quad t\in T.
  \end{equation*}
Given $k\in\Nb$, $\{J_j(j)\}_{j=k}^\infty$ is a subsequence of $\{J_k(j)\}_{j=1}^\infty$,
and so, the diagonal sequence $\{f_{J_j(j)}\}_{j=1}^\infty$, again denoted by $\{f_j\}$,
satisfies condition \eq{e:SP1}.

2. Let $Q$ be an at most countable dense subset of $T$. Note that any point $t\in T$,
which is not a limit point for $T$ (i.e., $T\cap(t-\delta,t+\delta)=\{t\}$ for some
$\delta>0$), belongs to~$Q$. Since, for any $k\in\Nb$, $\vfi_k\in\Mon(T;\Rb^+)$,
the set $Q_k\subset T$ of points of discontinuity of $\vfi_k$ is at most countable.
Setting $S=Q\cup\bigcup_{k=1}^\infty Q_k$, we find that $S$ is an at most countable
dense subset of $T$; moreover, if $S\ne T$, then every point $t\in T\setminus S$ is
a limit point for $T$ and 
  \begin{equation} \label{e:SP2}
\mbox{$\vfi_k$ is continuous on $T\setminus S$ for all $k\in\Nb$.}
  \end{equation}
Since $S\subset T$ is at most countable and $\{f_j(s):j\in\Nb\}$ is \rc\ in $M$ for all
$s\in S$, applying the Cantor diagonal procedure and passing to a subsequence of
$\{f_j(s)\}_{j=1}^\infty$ if necessary, with no loss of generality we may assume that,
for each $s\in S$, $f_j(s)$ converges in $M$ as $j\to\infty$ to a (unique) point
denoted by $f(s)\in M$ (so that $f:S\to M$).

If $S=T$, we turn to Step~4 below and complete the proof.

3. Now, assuming that $S\ne T$, we prove that $f_j(t)$ converges in $M$ as $j\to\infty$
for all $t\in T\setminus S$, as well. Let $t\in T\setminus S$ and $\eta>0$ be arbitrarily
fixed. Since $\vep_k\to0$ as $k\to\infty$ (cf.\ Step~1), we pick and fix
$k=k(\eta)\in\Nb$ such that $\vep_k\le\eta$. By \eq{e:SP2}, $\vfi_k$ is continuous
at $t$, and so, by the density of $S$ in $T$, there is $s=s(k,t)\in S$ such that
$|\vfi_k(t)-\vfi_k(s)|\le\eta$. From property \eq{e:SP1}, there is
$j^1=j^1(\eta,k,t,s)\in\Nb$ such that, for all $j\ge j^1$,
  \begin{equation} \label{e:SP3}
|V_{\vep_k}(f_j,T\cap(-\infty,t])\!-\!\vfi_k(t)|\!\le\!\eta\,\,\mbox{and}\,\,
|V_{\vep_k}(f_j,T\cap(-\infty,s])\!-\!\vfi_k(s)|\!\le\!\eta.
  \end{equation}
Assuming that $s<t$ (with no loss of generality) and applying Lemma~\ref{l:mor}
(where $T$ is replaced by $T\cap(-\infty,t]$, $T_1$---by $T\cap(-\infty,s]$,
and $T_2$---by $T\cap[s,t]$), we get
  \begin{align*}
V_{\vep_k}(f_j,T\cap[s,t])&\le V_{\vep_k}(f_j,T\cap(-\infty,t])-
  V_{\vep_k}(f_j,T\cap(-\infty,s])\\[3pt]
&\le|V_{\vep_k}(f_j,T\cap(-\infty,t])-\vfi_k(t)|+|\vfi_k(t)-\vfi_k(s)|\\[3pt]
&\qquad+|\vfi_k(s)-V_{\vep_k}(f_j,T\cap(-\infty,s])|\\[3pt]
&\le\eta+\eta+\eta=3\eta\quad\mbox{for all}\quad j\ge j^1.
  \end{align*}
By the definition of $V_{\vep_k}(f_j,T\cap[s,t])$, for each $j\ge j^1$, there is
$g_j\in\BV(T\cap[s,t];M)$ (also depending on $\eta$, $k$, $t$, and $s$) such that
  \begin{equation*}
d_{\infty,T\cap[s,t]}(f_j,g_j)\le\vep_k\quad\mbox{and}\quad
V(g_j,T\cap[s,t])\le V_{\vep_k}(f_j,T\cap[s,t])+\eta.
  \end{equation*}
These inequalities, \eq{e:10} and property (V.1) on p.~\pageref{p:V} yield,
for all $j\ge j^1$,
  \begin{align}
d(f_j(s),f_j(t))&\le d(g_j(s),g_j(t))+2d_{\infty,T\cap[s,t]}(f_j,g_j)\nonumber\\[3pt]
&\le V(g_j,T\cap[s,t])+2\vep_k\le(3\eta+\eta)+2\eta=6\eta.\label{e:SP4}
  \end{align}
Being convergent, the sequence $\{f_j(s)\}_{j=1}^\infty$ is Cauchy in $M$, and so,
there is a natural number $j^2=j^2(\eta,s)$ such that $d(f_j(s),f_{j'}(s))\le\eta$ for all
$j,j'\ge j^2$. Since the number $j^3=\max\{j^1,j^2\}$ depends only on $\eta$
(and $t$) and
  \begin{align*}
d(f_j(t),f_{j'}(t))&\le d(f_j(t),f_j(s))+d(f_j(s),f_{j'}(s))+d(f_{j'}(s),f_{j'}(t))\\[3pt]
&\le 6\eta+\eta+6\eta=13\eta\quad\mbox{for all}\quad j,j'\ge j^3,
  \end{align*}
the sequence $\{f_j(t)\}_{j=1}^\infty$ is Cauchy in $M$. Taking into account that the set
$\{f_j(t):j\in\Nb\}$ is \rc\ in $M$, we conclude that $f_j(t)$ converges in $M$ as
$j\to\infty$ to a (unique) point denoted by $f(t)\in M$ (so, $f:T\setminus S\to M$).

4. At the end of Steps 2 and 3, we have shown that the function
$f$ mapping $T=S\cup(T\setminus S)$ into $M$ is the \pw\ limit on $T$ of a subsequence
$\{f_{j_p}\}_{p=1}^\infty$ of the original sequence $\{f_j\}_{j=1}^\infty$. By virtue
of Lemma~\ref{l:71}(b) and assumption \eq{e:sp}, given $\vep>0$, we get
  \begin{align*}
|f(T)|&\le\liminf_{p\to\infty}|f_{j_p}(T)|
  \le\liminf_{p\to\infty}V_\vep(f_{j_p},T)+2\vep\\[3pt]
&\le\limsup_{j\to\infty}V_\vep(f_j,T)+2\vep<\infty,
  \end{align*}
and so, $f$ is a \emph{bounded\/} function on $T$, i.e., $f\in\Bd(T;M)$.

Now, we prove that $f$ is \emph{regulated\/} on $T$. Given $\tau\in T$, which is a
left limit point for $T$, let us show that $d(f(s),f(t))\to0$ as $T\ni s,t\to\tau-0$ (similar
arguments apply if $\tau'\in T$ is a right limit point for $T$). This is equivalent to showing
that for every $\eta>0$ there is $\delta=\delta(\eta)>0$ such that
$d(f(s),f(t))\le7\eta$ for all $s,t\in T\cap(\tau-\delta,\tau)$ with $s<t$.
Recall that the (finally) extracted subsequence of the original sequence $\{f_j\}$,
here again denoted by $\{f_j\}$, satisfies condition \eq{e:SP1} and $f_j\to f$ \pw\ on~$T$.

Let $\eta>0$ be arbitrarily fixed. Since $\vep_k\to0$, pick and fix $k=k(\eta)\in\Nb$
such that $\vep_k\le\eta$. Furthermore, since $\vfi_k\in\Mon(T;\Rb^+)$ and $\tau\in T$
is a left limit point of $T$, the left limit $\lim_{T\ni t\to\tau-0}\vfi_k(t)\in\Rb^+$ exists.
Hence, there is $\delta=\delta(\eta,k)>0$ such that $|\vfi_k(t)-\vfi_k(s)|\le\eta$ for
all $s,t\in T\cap(\tau-\delta,\tau)$. Now, let $s,t\in T\cap(\tau-\delta,\tau)$ be arbitrary.
By \eq{e:SP1}, there is $j^1=j^1(\eta,k,s,t)\in\Nb$ such that if $j\ge j^1$, the
inequalities \eq{e:SP3} hold. Arguing exactly the same way as between lines \eq{e:SP3}
and \eq{e:SP4}, we find that $d(f_j(s),f_j(t))\le6\eta$ for all $j\ge j^1$. Noting that
$f_j(s)\to f(s)$ and $f_j(t)\to f(t)$ in $M$ as $j\to\infty$, by the triangle inequality
for $d$, we have, as $j\to\infty$,
  \begin{equation*}
|d(f_j(s),f_j(t))-d(f(s),f(t))|\le d(f_j(s),f(s))+d(f_j(t),f(t))\to0.
  \end{equation*}
So, there is $j^2=j^2(\eta,s,t)\in\Nb$ such that $d(f(s),f(t))\le d(f_j(s),f_j(t))+\eta$
for all $j\ge j^2$. Thus, if $j\ge\max\{j^1,j^2\}$, we get
$d(f(s),f(t))\le6\eta+\eta=7\eta$.

5. Finally, assume that $(M,d)$ is a \emph{proper\/} metric space. Once again (as at the
beginning of Step~4) it is convenient to denote the \pw\ convergent subsequence of
$\{f_j\}$ by $\{f_{j_p}\}_{p=1}^\infty$. So, since $f_{j_p}\to f$ \pw\ on $T$ as
$p\to\infty$, we may apply Lemma~\ref{l:proper}(c) and assumption \eq{e:sp}
and get, for all $\vep>0$,
  \begin{equation*}
V_\vep(f,T)\le\liminf_{p\to\infty}V_\vep(f_{j_p},T)\le
\limsup_{j\to\infty}V_\vep(f,T)<\infty.
  \end{equation*}
This and \eq{e:supReg} (or Lemma~\ref{l:Regc} if $T=I$) also imply $f\in\Reg(T;M)$.

This completes the proof of Theorem~\ref{t:SP}.
\sq

A few remarks concerning Theorem~\ref{t:SP} are in order (see also Remarks
\ref{r:cHp} and \ref{r:neces}).

\begin{remark} \label{r:four2} \rm
If $(M,d)$ is a \emph{proper\/} metric space, then the assumption that
`$\{f_j\}\subset M^T$
is \pw\ \rc\ on $T$' in Theorem~\ref{t:SP} can be replaced by an (seemingly weaker,
but, actually) equivalent condition `$\{f_j\}\subset M^T$ and $\{f_j(t_0)\}$ is
\emph{eventually bounded\/} in $M$ for some $t_0\in T$' in the sense that there are
$J_0\in\Nb$ and a constant $C_0>0$ such that $d(f_j(t_0),f_{j'}(t_0))\le C_0$ for all
$j,j'\ge J_0$. In fact, fixing $\vep>0$, e.g., $\vep=1$, by Lemma~\ref{l:71}(b), we get
$|f_j(T)|\le V_1(f_j,T)+2$ for all $j\in\Nb$, and so, applying assumption \eq{e:sp},
  \begin{equation*}
\limsup_{j\to\infty}|f_j(T)|\le\limsup_{j\to\infty}V_1(f_j,T)+2<\infty.
  \end{equation*}
Hence, there are $J_1\in\Nb$ and a constant $C_1>0$ such that
$|f_j(T)|\le C_1$ for all $j\ge J_1$. By the triangle inequality for $d$,
given $t\in T$, we find, for all $j,j'\ge\max\{J_0,J_1\}$,
  \begin{align}
d(f_j(t),f_{j'}(t))&\le d(f_j(t),f_j(t_0))+d(f_j(t_0),f_{j'}(t_0))+d(f_{j'}(t_0),f_{j'}(t))
  \nonumber\\[3pt]
&\le|f_j(T)|+C_0+|f_{j'}(T)|\le C_1+C_0+C_1, \label{e:Cio}
  \end{align}
i.e., $\{f_j(t)\}$ is eventually bounded uniformly in $t\in T$. Thus, since $M$ is proper,
$\{f_j(t)\}$ is \rc\ in $M$ for all $t\in T$.
In the case under consideration, an alternative proof of Theorem~\ref{t:SP},
worth mentioning of, can be given (see Theorem~\ref{t:SPprop} and its proof).

However, for a general metric space $(M,d)$, the relative compactness of $\{f_j(t)\}$
at all points $t\in T$ cannot be replaced by their (closedness and) boundedness even
at a single point of $T$. To see this, let $T=I=[a,b]$ and $M=\ell^1\subset\Rb^\Nb$
be the (infinite-dimensional) Banach space of all summable sequences
$u=\{u_n\}_{n=1}^\infty\in\ell^1$ equipped with the norm
$\|u\|=\sum_{n=1}^\infty|u_n|<\infty$. If $j\in\Nb$, denote by
$e_j=\{u_n\}_{n=1}^\infty$ the unit vector from $\ell^1$ given by
$u_n=0$ if $n\ne j$ and $u_j=1$. Now, define the sequence $\{f_j\}\subset M^T$
by $f_j(a)=e_j$ and $f_j(t)=0$ if $a<t\le b$, $j\in\Nb$. We have: the set
$\{f_j(a)\}_{j=1}^\infty=\{e_j:j\in\Nb\}$ is closed and bounded in $M$,
$\{f_j(t)\}_{j=1}^\infty=\{0\}$ is compact in $M$ if $a<t\le b$, and
(cf. Example~\ref{ex:thr} and \eq{e:tab}), given $j\in\Nb$, $V_\vep(f_j,T)=1-2\vep$
if $0<\vep<1/2$, and $V_\vep(f_j,T)=0$ if $\vep\ge1/2$. Clearly, condition \eq{e:sp}
is satisfied for $\{f_j\}$, but no subsequence of $\{f_j\}$ converges in $M$ at the
point $t=a$.
\end{remark}

\begin{theorem} \label{t:SPprop}
Suppose $T\subset\Rb$, $(M,d)$ is a {\sl proper} metric space, and a sequence of
functions $\{f_j\}\subset M^T$ is such that $\{f_j(t_0)\}$ is eventually bounded in $M$
for some $t_0\in T$ and condition\/ \eq{e:sp} holds. Then, a subsequence of $\{f_j\}$
converges \pw\ on $T$ to a {\sl bounded} function $f\in M^T$ such that
$V_\vep(f,T)\le\limsup_{j\to\infty}V_\vep(f_j,T)$ for all $\vep>0$.
\end{theorem}

\proof
1. Let $\{\vep_k\}_{k=1}^\infty\subset(0,\infty)$ be such that $\vep_k\to0$ as
$k\to\infty$. Given $k\in\Nb$, condition \eq{e:sp} implies the existence of
$j_0'(\vep_k)\in\Nb$ and a constant $C(\vep_k)>0$ such that
$V_{\vep_k}(f_j,T)<C(\vep_k)$ for all $j\ge j_0'(\vep_k)$. By definition \eq{e:av},
for each $j\ge j_0'(\vep_k)$, there is $g_j^{(k)}\in\BV(T;M)$ such that%
\footnote{Conditions $\{f_j\}\subset M^T$ is \rc\ on $T$ and $\{g_j\}\subset M^T$ is
such that $d_{\infty,T}(f_j,g_j)\le\vep$ for all $j\in\Nb$ \emph{do not\/} imply in general
that $\{g_j\}$ is also \rc\ on $T$: e.g.\ (cf.\ notation in Remark~\ref{r:four2}), $T=[0,1]$,
$M=\ell^1$ (which is not proper), $f_j(t)=0$ and $g_j(t)=\vep t e_j$ for all $j\in\Nb$
and $t\in T$.}
  \begin{equation} \label{e:dnVC}
d_{\infty,T}(f_j,g_j^{(k)})\le\vep_k\quad\mbox{ and }\quad V(g_j^{(k)},T)\le C(\vep_k).
  \end{equation}
Since $\{f_j(t_0)\}$ is eventually bounded and \eq{e:sp} holds, we get inequality
\eq{e:Cio} for all $j,j'\ge\max\{J_0,J_1\}$. It follows that if $t\in T$, $k\in\Nb$, and
$j,j'\ge j_0(\vep_k)\equiv\max\{j_0'(\vep_k),J_0,J_1\}$, we find, by the triangle
inequality for $d$, \eq{e:dnVC},~and~\eq{e:Cio},
  \begin{align*}
d(g_j^{(k)}(t),g_{j'}^{(k)}(t))&\le d(g_j^{(k)}(t),f_j(t))+d(f_j(t),f_{j'}(t))
  +d(f_{j'}(t),g_{j'}^{(k)}(t))\\[2pt]
&\le d_{\infty,T}(g_j^{(k)},f_j)+d(f_j(t),f_{j'}(t))+d_{\infty,T}(f_{j'},g_{j'}^{(k)})\\[2pt]
&\le\vep_k+(C_0+2C_1)+\vep_k.
  \end{align*}
In this way, we have shown that
  \begin{equation} \label{e:djjk}
\sup_{j,j'\ge j_0(\vep_k)}d(g_j^{(k)}(t),g_{j'}^{(k)}(t))\le2\vep_k+C_0+2C_1\,\,\,
\mbox{for all $k\in\Nb$ and $t\in T$,}
  \end{equation}
and, by the second inequality in \eq{e:dnVC},
  \begin{equation} \label{e:Vgj}
\sup_{j\ge j_0(\vep_k)}V(g_j^{(k)},T)\le C(\vep_k)\quad\mbox{for \,all}\quad k\in\Nb.
  \end{equation}

2. Applying Cantor's diagonal procedure, let us show the following: given $k\in\Nb$, there
exist a subsequence of $\{g_j^{(k)}\}_{j=j_0(\vep_k)}^\infty$, denoted by
$\{g_j^{(k)}\}_{j=1}^\infty$, and $g^{(k)}\in\BV(T;M)$ such that
  \begin{equation} \label{e:gkkt}
\lim_{j\to\infty}d(g_j^{(k)}(t),g^{(k)}(t))=0\quad\mbox{for \,all}\quad t\in T.
  \end{equation}

Setting $k=1$ in  \eq{e:djjk} and \eq{e:Vgj}, we find that the sequence
$\{g_j^{(1)}\}_{j=j_0(\vep_1)}^\infty$ has uniformly bounded (by $C(\vep_1)$)
Jordan variations and is uniformly bounded on $T$ (by $2\vep_1+C_0+2C_1$), and so,
since $M$ is a \emph{proper\/} metric space, the sequence is \pw\ \rc\ on~$T$.
By the Helly-type \pw\ selection principle in $\BV(T;M)$ (cf.\ property (V.4) on
p.~\pageref{p:V}), there are a subsequence $\{J_1(j)\}_{j=1}^\infty$ of
$\{j\}_{j=j_0(\vep_1)}^\infty$ and a function $g^{(1)}\in\BV(T;M)$ such that
$g_{J_1(j)}^{(1)}(t)\to g^{(1)}(t)$ in $M$ as $j\to\infty$ for all $t\in T$.
Pick the least number $j_1\in\Nb$ such that $J_1(j_1)\ge j_0(\vep_2)$. Inductively,
if $k\in\Nb$ with $k\ge2$, a subsequence $\{J_{k-1}(j)\}_{j=1}^\infty$ of
$\{j\}_{j=j_0(\vep_1)}^\infty$, and the number $j_{k-1}\in\Nb$ such that
$J_{k-1}(j_{k-1})\ge j_0(\vep_k)$ are already chosen, we get the sequence of functions
$\{g_{J_{k-1}(j)}^{(k)}\}_{j=j_{k-1}}^\infty\subset\BV(T;M)$, which, by virtue of
\eq{e:djjk} and \eq{e:Vgj}, satisfies conditions
  \begin{equation*}
\sup_{j,j'\ge j_{k-1}}d\bigl(g_{J_{k-1}(j)}^{(k)}(t),g_{J_{k-1}(j')}^{(k)}(t)\bigr)\le
2\vep_k+C_0+2C_1\quad\mbox{for \,all}\quad t\in T
  \end{equation*}
and
  \begin{equation*}
\sup_{j\ge j_{k-1}}V\bigl(g_{J_{k-1}(j)}^{(k)},T\bigr)\le C(\vep_k).
  \end{equation*}
By Helly's-type selection principle (V.4) in $\BV(T;M)$, there are a subsequence
$\{J_k(j)\}_{j=1}^\infty$ of $\{J_{k-1}(j)\}_{j=j_{k-1}}^\infty$ and a function
$g^{(k)}\in\BV(T;M)$ such that $g_{J_k(j)}^{(k)}(t)\to g^{(k)}(t)$ in $M$ as
$j\to\infty$ for all $t\in T$. Since, for each $k\in\Nb$, $\{J_j(j)\}_{j=k}^\infty$ is a
subsequence of $\{J_k(j)\}_{j=j_{k-1}}^\infty\subset\{J_k(j)\}_{j=1}^\infty$,
we conclude that the diagonal sequence $\{g_{J_j(j)}^{(k)}\}_{j=1}^\infty$,
(which was) denoted by $\{g_j^{(k)}\}_{j=1}^\infty$ (at the beginning of step~2),
satisfies condition \eq{e:gkkt}.

We denote the corresponding diagonal subsequence $\{f_{J_j(j)}\}_{j=1}^\infty$ of
$\{f_j\}$ again by $\{f_j\}$.

3. Since $\BV(T;M)\subset\Bd(T;M)$ (by (V.1) on p.~\pageref{p:V}),
$\{g^{(k)}\}_{k=1}^\infty\subset\Bd(T;M)$. We are going to show that
$\{g^{(k)}\}_{k=1}^\infty$ is a Cauchy sequence with respect to the
uniform metric $d_{\infty,T}$. For this, we employ an idea from \cite[p.~49]{Fr}.

Let $\eta>0$ be arbitrary. From $\vep_k\to0$, we find $k_0=k_0(\eta)\in\Nb$ such that
$\vep_k\le\eta$ for all $k\ge k_0$. Now, suppose $k,k'\in\Nb$ be (arbitrary) such that
$k,k'\ge k_0$. By virtue of \eq{e:gkkt}, for each $t\in T$, there is a number
$j^1=j^1(t,\eta,k,k')\in\Nb$ such that if $j\ge j^1$, we have
  \begin{equation*}
d\bigl(g_j^{(k)}(t),g^{(k)}(t)\bigr)\le\eta\quad\mbox{ and }\quad
d\bigl(g_j^{(k')}(t),g^{(k')}(t)\bigr)\le\eta.
  \end{equation*}
Now, it follows from the triangle inequality for $d$ and the first inequality in
\eq{e:dnVC} that if $j\ge j^1$,
  \begin{align*}
d\bigl(g^{(k)}(t),g^{(k')}(t)\bigr)&\le d\bigl(g^{(k)}(t),g_j^{(k)}(t)\bigr)
 +d\bigl(g_j^{(k)}(t),f_j(t)\bigr)\\[2pt]
&\qquad+d\bigl(f_j(t),g_j^{(k')}(t)\bigr)+d\bigl(g_j^{(k')}(t),g^{(k')}(t)\bigr)\\[2pt]
&\le\eta+\vep_k+\vep_{k'}+\eta\le4\eta.
  \end{align*}
By the arbitrariness of $t\in T$, $d_{\infty,T}d(g^{(k)},g^{(k')})\le4\eta$ for all
$k,k'\ge k_0$.

4. Being proper, $(M,d)$ is complete, and so, $\Bd(T;M)$ is complete with respect to
the uniform metric $d_{\infty,T}$. By step~3, there is $g\in\Bd(T;M)$ such that
$g^{(k)}\rra g$ on $T$ (i.e., $d_{\infty,T}(g^{(k)},g)\to0$) as $k\to\infty$.
Let us prove that $f_j\to g$ pointwise on $T$ as $j\to\infty$ ($\{f_j\}$ being from
the end of step~2).

Let $t\in T$ and $\eta>0$ be arbitrary. Choose and fix a number $k=k(\eta)\in\Nb$
such that $\vep_k\le\eta$ and $d_{\infty,T}(g^{(k)},g)\le\eta$. By \eq{e:gkkt},
there is $j^2=j^2(t,\eta,k)\in\Nb$ such that $d(g_j^{(k)}(t),g^{(k)}(t))\le\eta$
for all $j\ge j^2$, and so, \eq{e:dnVC} implies
  \begin{align*}
d(f_j(t),g(t))&\le d\bigl(f_j(t),g_j^{(k)}(t)\bigr)+d\bigl(g_j^{(k)}(t),g^{(k)}(t)\bigr)
  +d\bigl(g^{(k)}(t),g(t)\bigr)\\[2pt]
&\le\vep_k+d\bigl(g_j^{(k)}(t),g^{(k)}(t)\bigr)+d_{\infty,T}(g^{(k)},g)\\[2pt]
&\le\eta+\eta+\eta=3\eta\quad\,\mbox{for \,all}\quad j\ge j^2,
  \end{align*}
which proves our assertion.

Thus, we have shown that a suitable (diagonal) subsequence $\{f_{j_p}\}_{p=1}^\infty$
of the original sequence $\{f_j\}_{j=1}^\infty$ converges pointwise on $T$ to the
function $g$ from $\Bd(T;M)$. Setting $f=g$ and applying Lemma~\ref{l:proper}(c),
we conclude that
  \begin{equation*}
V_\vep(f,T)=V_\vep(g,T)\le\liminf_{p\to\infty}V_\vep(f_{j_p},T)
\le\limsup_{j\to\infty}V_\vep(f_j,T)\quad\forall\,\vep>0.
  \end{equation*}

This completes the proof of Theorem~\ref{t:SPprop}.
\sq

A simple consequence of Theorem~\ref{t:SPprop} is the following

\begin{corollary}
Assume that assumption \eq{e:sp} in Theorem~{\rm\ref{t:SPprop}} is replaced by
condition $\lim_{j\to\infty}|f_j(T)|=0$. Then, a subsequence of $\{f_j\}$ converges
pointwise on $T$ to a constant function on $T$.
\end{corollary}

\proof
In fact, given $\vep>0$, there is $j_0=j_0(\vep)\in\Nb$ such that $|f_j(T)|\le\vep$
for all $j\ge j_0$, and so, by \eq{e:zero}, $V_\vep(f_j,T)=0$ for all $j\ge j_0$.
This implies
  \begin{equation*}
\limsup_{j\to\infty}V_\vep(f_j,T)\le\sup_{j\ge j_0}V_\vep(f_j,T)=0\quad
\mbox{for \,all}\quad\vep>0.
  \end{equation*}
By Theorem~\ref{t:SPprop}, a subsequence of $\{f_j\}$ converges \pw\ on $T$ to
a function $f\in M^T$ such that $V_\vep(f,T)=0$ for all $\vep>0$. Lemma~\ref{l:71}(e)
yields $|f(T)|=0$, i.e., $f$ is a constant function on $T$.
\sq

\begin{remark} \label{r:cHp} \rm
The classical Helly selection principle for monotone functions (p.~\pageref{p:Hellym})
is a particular case of Theorem~\ref{t:SP}. In fact, suppose $\{\vfi_j\}\subset\Rb^T$
is a sequence of monotone functions, for which there is a constant $C>0$ such that
$|\vfi_j(t)|\le C$ for all $t\in T$ and $j\in\Nb$. Setting $(M,\|\cdot\|)=(\Rb,|\cdot|)$,
$x=1$ and $y=0$ in Example~\ref{ex:1}, for every $j\in\Nb$
 we find, from \eq{e:mntn} and \eq{e:ov2},
that $V_\vep(\vfi_j,T)=|\vfi_j(T)|-2\vep$ if $0<\vep<\frac12|\vfi_j(T)|$ and
$V_\vep(\vfi_j,T)=0$ if $\vep\ge\frac12|\vfi_j(T)|$. Since $|\vfi_j(T)|\le2C$, we get
$V_\vep(\vfi_j,T)\le2C$ for all $j\in\Nb$ and $\vep>0$, and so, \eq{e:sp} is satisfied.

Similarly, Theorem~\ref{t:SP} implies Helly's selection principle for functions of bounded
variation (cf.\ property (V.4) on p.~\pageref{p:V}). In fact, if $\{f_j\}\subset M^T$
and $C=\sup_{j\in\Nb}V(f_j,T)$ is finite, then, by Lemma~\ref{l:71}(a),
$V_\vep(f_j,T)\le C$ for all $j\in\Nb$ and $\vep>0$, and so, \eq{e:sp} is fulfilled.
Now, if a subsequence of $\{f_j\}$ converges pointwise on $T$ to $f\in M^T$, then
property (V.3) (p.~\pageref{p:V}) implies $f\in\BV(T;M)$ with $V(f,T)\le C$.
\end{remark}

\begin{remark} \label{r:neces} \rm
(a) Condition \eq{e:sp} is \emph{necessary\/} for the \emph{uniform convergence\/}
in the following sense: if $\{f_j\}\subset M^T$, $f_j\rra f$ on $T$ and
$V_\vep(f,T)<\infty$ for all $\vep>0$, then, by Lemma~\ref{l:uc}(a),
  \begin{equation*}
\limsup_{j\to\infty}V_\vep(f_j,T)\le V_{\vep-0}(f,T)\le V_{\vep'}(f,T)<\infty\quad
\mbox{for all}\quad 0<\vep'<\vep.
  \end{equation*}

(b) Contrary to this, \eq{e:sp} \emph{is not\/} necessary for the \pw\
convergence (see Examples~\ref{ex:notnec} and \ref{ex:poico}).
 On the other hand, condition \eq{e:sp} is
`almost necessary' for the \pw\ convergence $f_j\to f$ on $T$ in the following sense.
Assume that $T\subset\Rb$ is a measurable set with \emph{finite\/} Lebesgue
measure $\mathcal{L}(T)$ and $\{f_j\}\subset M^T$ is a sequence of measurable
functions such that $f_j\to f$ on $T$ (or even $f_j$ converges almost everywhere
on $T$ to~$f$) and $V_\vep(f,T)<\infty$ for all $\vep>0$. By Egorov's Theorem
(e.g., \cite[Section~3.2.7]{Rao}), given $\eta>0$, there is a measurable set
$T_\eta\subset T$ such that $\mathcal{L}(T\setminus T_\eta)\le\eta$ and
$f_j\rra f$ on $T_\eta$. By (a) above and Lemma~\ref{l:ele}(b), we have
  \begin{equation*}
\limsup_{j\to\infty}V_\vep(f_j,T_\eta)\le V_{\vep'}(f,T_\eta)\le V_{\vep'}(f,T)<\infty
\quad\mbox{for all}\quad 0<\vep'<\vep.
  \end{equation*}
\end{remark}

\section{Examples illustrating Theorem~\protect\ref{t:SP}} \label{ss:illSP}

\begin{example} \label{ex:sinejt}
The main assumption \eq{e:sp} in Theorem~\ref{t:SP} is essential. In fact, it is
well known that the sequence of functions $\{f_j\}\subset\Rb^T$ on the interval
$T=[0,2\pi]$ defined by $f_j(t)=\sin(jt)$, $0\le t\le2\pi$, has no subsequence convergent
at all points of $T$ (cf.\ \cite[Example~3]{JMAA05}; more explicitly this is revived in
Remark~\ref{r:sinjt} below). Let us show that $\{f_j\}$ does not satisfy condition~%
\eq{e:sp}.

Let us fix $j\in\Nb$. First, note that, given $t,s\in[0,2\pi]$, we have $\sin(jt)=0$ if and
only if $t=t_k=k\pi/j$, $k=0,1,2,\dots,2j$, and $|\sin(js)|=1$ if and only if
$s=s_k=\frac12(t_{k-1}+t_k)=(k-\frac12)\pi/j$, $k=1,2,\dots,2j$. Setting
$I_k=[t_{k-1},s_k]$ and $I_k'=[s_k,t_k]$, we find
  \begin{equation*}
T=[0,2\pi]=\bigcup_{k=1}^{2j}[t_{k-1},t_k]=\bigcup_{k=1}^{2j}(I_k\cup I_k')\quad
\mbox{(non-overlapping intervals),}
  \end{equation*}
and $f_j$ is strictly monotone on each interval $I_k$ and $I_k'$, $k=1,2\dots,2j$.
By virtue of Lemma~\ref{l:mor}, given $\vep>0$, we have
  \begin{equation} \label{e:sut}
\sum_{k=1}^{2j}\!\bigl(V_\vep(f_j,I_k)+V_\vep(f_j,I_k')\bigr)\le V_\vep(f_j,T)
\le\sum_{k=1}^{2j}\!\bigl(V_\vep(f_j,I_k)+V_\vep(f_j,I_k')\bigr)+(4j-1)2\vep.
  \end{equation}
It suffices to calculate $V_\vep(f_j,I_k)$ for $k=1$, where $I_1=[t_0,s_1]=[0,\pi/2j]$
(the other $\vep$-variations in \eq{e:sut} are calculated similarly and give the same value).
Since $f_j$ is strictly increasing on $I_1$, $f_j(I_1)=[0,1]$ and $|f_j(I_1)|=1$,
\eq{e:mntn} and \eq{e:ov2} imply $V_\vep(f_j,I_1)=1-2\vep$ if $0<\vep<\frac12$
(and $V_\vep(f_j,I_1)=0$ if $\vep\ge\frac12$). Hence, $V_\vep(f_j,I_k)=%
V_\vep(f_j,I_k')=1-2\vep$ for all $0<\vep<\frac12$ and $k=1,2,\dots,2j$, and
it follows from \eq{e:sut} that
  \begin{equation*}
4j(1-2\vep)\le V_\vep(f_j,T)\le4j(1-2\vep)+(4j-1)2\vep=4j-2\vep,
\quad 0<\vep<\textstyle\frac12.
  \end{equation*}
Thus, condition \eq{e:sp} is not satisfied by $\{f_j\}$.
\end{example}

\begin{remark} \label{r:sinjt} \rm
Since the sequence of functions $\{f_j\}_{j=1}^\infty$ from Example~\ref{ex:sinejt},
i.e., $f_j(t)=\sin(jt)$ for $t\in[0,2\pi]$, plays a certain role in the sequel as well, for the
sake of completeness, we recall here the proof of the fact (e.g.,
\cite[Chapter~7, Example~7.20]{Rudin}) that no subsequence of
$\{\sin(jt)\}_{j=1}^\infty$
converges in $\Rb$ for all $t\in[0,2\pi]$; note that $\{f_j\}$ is a uniformly bounded
sequence of continuous functions on the compact set $[0,2\pi]$. On the contrary,
assume that there is an increasing sequence $\{j_p\}_{p=1}^\infty\subset\Nb$
such that $\sin(j_pt)$ converges as $p\to\infty$ for all $t\in[0,2\pi]$. Given $t\in[0,2\pi]$,
this implies $\sin(j_pt)-\sin(j_{p+1}t)\to0$, and so, $(\sin(j_pt)-\sin(j_{p+1}t))^2\to0$
as $p\to\infty$. By Lebesgue's dominated convergence theorem,
$I_p\equiv\int_0^{2\pi}(\sin(j_pt)-\sin(j_{p+1}t))^2dt\to0$ as $p\to\infty$.
However, a straightforward computation of the integral $I_p$ (note that $j_p\!<\!j_{p+1}$)
gives the value $I_p=2\pi$ for all $p\in\Nb$, which is a contradiction.

More precisely (cf.~\cite[Chapter~10, Exercise~16]{Rudin}), the set $E\subset[0,2\pi]$ of
all points $t\in[0,2\pi]$, for which $\sin(j_pt)$ converges as $p\to\infty$ (with
$\{j_p\}_{p=1}^\infty$ as above), is of Lebesgue measure zero, $\mathcal{L}(E)=0$.
To see this, it suffices to note that, for a measurable set $A\subset E$,
$\int_A\sin(j_pt)dt\to0$ and
  \begin{equation*}
\int_A(\sin(j_pt))^2dt=\frac12\int_A(1-\cos(2j_pt))dt\to\frac12\mathcal{L}(A)\quad
\mbox{as}\quad p\to\infty.
  \end{equation*}

To illustrate the assertion in the previous paragraph, let us show that, given $t\in\Rb$,
$\sin(jt)$ converges as $j\to\infty$ if and only if $\sin t=0$ (i.e., $t=\pi k$ for some
integer~$k$). Since the sufficiency is clear, we prove the necessity. Suppose $t\in\Rb$
and the limit $\phi(t)=\lim_{j\to\infty}\sin(jt)$ exists in~$\Rb$. To show that $\sin t=0$,
we suppose, by contadiction, that $\sin t\ne0$. Passing to the limit as $j\to\infty$ in $\sin(j+2)t+\sin(jt)=2\sin(j+1)t\cdot\cos t$, we get
$\phi(t)+\phi(t)=2\phi(t)\cos t$, which is equivalent to $\phi(t)=0$ or $\cos t=1$.
Since $\sin(j+1)t=\sin(jt)\cdot\cos t+\sin t\cdot\cos(jt)$, we find
  \begin{equation*}
\cos(jt)=\frac{\sin(j+1)t-\sin(jt)\cdot\cos t}{\sin t},
  \end{equation*}
and so, $\lim_{j\to\infty}\cos(jt)=\phi(t)(1-\cos t)/\sin t$. Hence
  \begin{align*}
1&=\lim_{j\to\infty}\bigl(\sin^2(jt)+\cos^2(jt)\bigr)=
(\phi(t))^2+(\phi(t))^2\cdot\biggl(\frac{1\!-\!\cos t}{\sin t}\biggr)^{\!2}\\
&\qquad\qquad\qquad=(\phi(t))^2\cdot\frac{2(1\!-\!\cos t)}{\sin^2t},
  \end{align*}
and so, $\phi(t)\ne0$ and $1\ne\cos t$, which is a contradition. (In a similar manner,
one may show that, given $t\in\Rb$, $\cos(jt)$ converges as $j\to\infty$ if and only if
$\cos t=1$, i.e., $t=2\pi k$ for some integer~$k$; see \cite[p.~233]{Bridger}).

Returning to the convergence set $E\subset[0,2\pi]$ of the sequence
$\{\sin(j_pt)\}_{p=1}^\infty$, as a consequence of the previous assertion, we find that
if $j_p=p$, then $E=\{0,\pi,2\pi\}$. In general, the set $E$ may be
`quite large': for instance, if $j_p=p!$, then $E=\pi\cdot(\Qb\cap[0,2])$, which is
countable and dense in~$[0,2\pi]$.
\end{remark}

\begin{example} \label{ex:notnec} \rm
That condition \eq{e:sp} in Theorem~\ref{t:SP} is \emph{not necessary\/} for the
pointwise convergence $f_j\to f$ can be illustrated by the sequence $\{f_j\}$ from
Example~\ref{ex:ucbw}, where $I=[a,b]=[0,1]$. We assert that if
$0<\vep<\frac12d(x,y)$, then $\lim_{j\to\infty}V_\vep(f_j,I)=\infty$. To see this,
given $j\in\Nb$, we consider a partition of $I$ as defined in \eq{e:partj}. Supposing
that  $g\in M^I$ is arbitrary such that $d_{\infty,I}(f_j,g)\le\vep$, we find, by virtue
of \eq{e:10},
  \begin{equation*}
V(g,I)\ge\sum_{k=1}^{j!}d(g(t_k),g(s_k))\ge
\sum_{k=1}^{j!}\bigl(d(f(t_k),f(s_k))\!-\!2\vep\bigr)\!=\!j!\bigl(d(x,y)\!-\!2\vep\bigr).
  \end{equation*}
By definition \eq{e:av}, $V_\vep(f_j,I)\ge j!(d(x,y)-2\vep)$, which proves our assertion.
\end{example}

\begin{example} \label{ex:equim} \rm
The choice of an appropriate (equivalent) metric on $M$ is essential in Theorem~\ref{t:SP}.
(Recall that two metrics $d$ and $d'$ on $M$ are \emph{equivalent\/} if, given
a sequence $\{x_j\}\subset M$ and $x\in M$, conditions $d(x_j,x)\to0$ and
$d'(x_j,x)\to0$ are equivalent.)

Let $d$ be an unbounded metric on $M$, i.e., $\sup_{x,y\in M}d(x,y)=\infty$ (for
instance, given $N\in\Nb$ and $q\ge1$,
 $M=\Rb^N$ and $d(x,y)=\|x-y\|$, where $x=(x_1,\dots,x_N)$,
$y=(y_1,\dots,y_N)\in\Rb^N$ and $\|x\|=\bigl(\sum_{i=1}^N|x_i|^q\bigr)^{1/q}$).
The unboundedness of $d$ is equivalent to $\sup_{x\in M}d(x,y)=\infty$ for all
$y\in M$, so let us fix $y_0\in M$ and pick $\{x_j\}\subset M$ such that
$d(x_j,y_0)\to\infty$ as $j\to\infty$ (e.g., in the case $M=\Rb^N$ we may set
$x_j=(j,j,\dots,j)$ and $y_0=(0,0,\dots,0)$). Given a sequence $\{\tau_j\}\subset(a,b)%
\subset I=[a,b]$ such that $\tau_j\to a$ as $j\to\infty$, define $\{f_j\}\subset M^I$ by
(cf.\ Example~\ref{ex:thr})
  \begin{equation*}
\mbox{$f_j(\tau_j)=x_j$ \,\,and \,\,$f_j(t)=y_0$ \,if\, 
$t\in I\setminus\{\tau_j\}$, $j\in\Nb$.}
  \end{equation*}
Clearly, $f_j\to c(t)\equiv y_0$ pointwise on $I$, and so, $\{f_j\}$ is \pw\ \rc\ on $I$
(this can be seen directly by noting that the set $\{f_j(t):j\in\Nb\}$ is equal to
$\{x_k:\mbox{$k\in\Nb$ and $\tau_k=t$}\}\cup\{y_0\}$, which is finite for all $t\in I$).
Given $\vep>0$, there is $j_0=j_0(\vep)\in\Nb$ such that $|f_j(I)|=d(x_j,y_0)>2\vep$
for all $j\ge j_0$, and so, by Lemma~\ref{l:71}(f),
  \begin{equation*}
V_\vep(f_j,I)\ge|f_j(I)|-2\vep=d(x_j,y_0)-2\vep\quad\mbox{for all}\quad j\ge j_0.
  \end{equation*}
Since $\lim_{j\to\infty}d(x_j,y_0)=\infty$, this implies
$\lim_{j\to\infty}V_\vep(f_j,I)=\infty$, and so, Theorem~\ref{t:SP} is inapplicable
in this context.

On the other hand, the metric $d'$ on $M$, given by $d'(x,y)=\frac{d(x,y)}{1+d(x,y)}$,
$x,y\in M$, is equivalent to~$d$. Let us denote by $V'(f_j,I)$ and $V'_\vep(f_j,I)$ the
variation and the $\vep$-variation of $f_j$ on $I$ with respect to metric $d'$, respectively.
The variation $V'(f_j,I)$ is equal to (by virtue of the additivity of $V'$)
  \begin{equation*}
V'(f_j,I)\!=\!V'(f_j,[a,\tau_j])+V'(f_j,[\tau_j,b])\!=\!d'(x_j,y_0)+d'(x_j,y_0)\!=\!
2\,\frac{d(x_j,y_0)}{1+d(x_j,y_0)}.
  \end{equation*}
Now, if $\vep>0$, by Lemma~\ref{l:71}(a), $V_\vep'(f_j,I)\le V'(f_j,I)$ for all $j\in\Nb$,
and so,
  \begin{equation*}
\limsup_{j\to\infty}V_\vep'(f_j,I)\le\lim_{j\to\infty}V'(f_j,I)=2.
  \end{equation*}
Thus, the main assumption \eq{e:sp} in Theorem~\ref{t:SP} is satisfied, and this
Theorem is applicable to the sequence $\{f_j\}$.

Another interpretation of this example is that the main condition \eq{e:sp} is
\emph{not invariant\/} under equivalent metrics on $M$.
\end{example}

\begin{example} \label{ex:poico} \rm
Here we show that condition \eq{e:sp} in Theorem~\ref{t:SP} is \emph{not necessary\/}
for the pointwise convergence $f_j\to f$ on $T$, although we have $V_\vep(f_j,T)<\infty$
and $V_\vep(f,T)<\infty$ for all $\vep>0$. Furthermore, condition \eq{e:sp} may not
hold with respect to \emph{any\/} (equivalent) metric on $M$ such that $d(f_j(t),f(t))\to0$
as $j\to\infty$ for all $t\in T$. In fact, let $T=[0,2\pi]$ and $M=\Rb$, and define
$\{f_j\}\subset M^T$ by (cf.\ \cite[Example~4]{JMAA05}): $f_j(t)=\sin(j^2t)$ if
$0\le t\le2\pi/j$ and $f_j(t)=0$ if $2\pi/j<t\le2\pi$, $j\in\Nb$. Clearly, $\{f_j\}$ converges
\pw\ on $T$ to the function $f\equiv0$ with respect to \emph{any\/} metric $d$ on $M$,
which is equivalent to the usual metric $(x,y)\mapsto|x-y|$ on $\Rb$. Since $f_j$ is
continuous on $T=[0,2\pi]$ with respect to metric $|x-y|$, and so, with respect to $d$,
we find, by Lemma~\ref{l:Regc}, $V_\vep(f_j,T)<\infty$ and $V_\vep(f,T)=0$ with
respect to $d$ for all $j\in\Nb$ and $\vep>0$. Now, given $j,n\in\Nb$, we set
  \begin{equation*}
s_{j,n}=\frac1{j^2}\Bigl(2\pi n-\frac{3\pi}2\Bigr)\quad\mbox{and}\quad
t_{j,n}=\frac1{j^2}\Bigl(2\pi n-\frac\pi2\Bigr),
  \end{equation*}
so that $f_j(s_{j,n})=1$ and $f_j(t_{j,n})=-1$. Note also that
  \begin{equation*}
0<s_{j,1}<t_{j,1}<s_{j,2}<t_{j,2}<\dots<s_{j,j}<t_{j,j}<2\pi/j\quad
\mbox{for all}\quad j\in\Nb.
  \end{equation*}
Let $0<\vep<\frac12d(1,-1)$. Given $j\in\Nb$, suppose $g\!\in\! M^T$ is arbitrary
such that \mbox{$d_{\infty,T}(f_j,g)\!\le\!\vep$}.
The definition of $V(g,T)$ and \eq{e:10} give
  \begin{align*}
V(g,T)&\ge\sum_{n=1}^jd(g(s_{j,n}),g(t_{j,n}))\ge
\sum_{n=1}^j\bigl(d(f_j(s_{j,n}),f_j(t_{j,n}))-2\vep\bigr)\\
&=\bigl(d(1,-1)-2\vep\bigr)j.
  \end{align*}
By the arbitrariness of $g$ as above and \eq{e:av}, $V_\vep(f_j,T)\ge(d(1,-1)-2\vep)j$,
and so, condition \eq{e:sp} is not fulfilled for $0<\vep<\frac12d(1,-1)$.
\end{example}

\begin{example} \label{ex:irreg} \rm
(a) Theorem~\ref{t:SP} is inapplicable to the sequence $\{f_j\}$ from Example~%
\ref{ex:rieq}, because (although $f_j\rra f=\Dc_{x,y}$ on $I$)
$\lim_{j\to\infty}V_\vep(f_j,I)=\infty$ for $\vep=\frac12\|x-y\|$. The reason is that
the limit function $\Dc_{x,y}$ is not regulated (if $x\ne y$). However, see Remark~%
\ref{r:neces}(a) if the limit function is regulated.

(b) Nevertheless, Theorem~\ref{t:SP} can be successfully applied to sequences of
\emph{nonregulated\/} functions. To see this, we again use the context of
Example~\ref{ex:rieq}, where we suppose $x=y\in M$, so that $f(t)=\Dc_{x,x}(t)=x$,
$t\in I$. Recall also that we have $f_j=\Dc_{x_j,y_j}$ with $x_j\ne y_j$, $j\in\Nb$,
$x_j\to x$ and $y_j\to y=x$ in $M$, and $f_j\rra f\equiv x$ on $I$. Given $\vep>0$,
there is $j_0=j_0(\vep)\in\Nb$ such that $\|x_j-y_j\|\le2\vep$ for all $j\ge j_0$, which
implies, by virtue of \eq{e:steen}, $V_\vep(f_j,I)=0$ for all $j\ge j_0$. This yields
condition \eq{e:sp}:
  \begin{equation*}
\limsup_{j\to\infty}V_\vep(f_j,I)\le\sup_{j\ge j_0}V_\vep(f_j,I)=0
  \end{equation*}
(cf.\ also \cite[Example~3]{Studia17}).

On the other hand, for a fixed $k\in\Nb$ and $0<\vep<\frac12\|x_k-y_k\|$, we have,
from \eq{e:steen}, $V_\vep(f_k,I)=\infty$, and so, $\sup_{j\in\Nb}V_\vep(f_j,I)\ge%
V_\vep(f_k,I)=\infty$. Thus, condition of uniform boundedness of $\vep$-variations
$\sup_{j\in\Nb}V_\vep(f_j,I)<\infty$, which was assumed in \cite[Theorem~3.8]{Fr},
is more restrictive than condition~\eq{e:sp}.
\end{example}

\section{Two extensions of Theorem~\protect\ref{t:SP}}

Applying Theorem~\ref{t:SP} and the diagonal procedure over expanding intervals,
we get the following \emph{local\/} version of Theorem~\ref{t:SP}.

\begin{theorem} \label{t:SPloc}
If $T\subset\Rb$, $(M,d)$ is a metric space and $\{f_j\}\subset M^T$ is a \pw\ \rc\
sequence of functions such that
  \begin{equation*}
\mbox{$\D\limsup_{j\to\infty}V_\vep(f_j,T\cap[a,b])<\infty$ \,for all \,$a,b\in T$,
$a\le b$, and $\vep>0$,}
  \end{equation*}
then a subsequence of $\{f_j\}$ converges pointwise on $T$ to a {\sl regulated} function
$f\!\in\!\Reg(T;M)\!$ such that $f$ is {\sl bounded} on $T\cap[a,b]$ for all $a,b\in T$,
$a\le b$.
\end{theorem}

\proof
With no loss of generality, we may assume that sequences $\{a_k\}$ and $\{b_k\}$
from $T$ are such that $a_{k+1}\!<\!a_k\!<\!b_k\!<\!b_{k+1}$ for all $k\!\in\!\Nb$,
$a_k\to\inf T\notin T$ and $b_k\to\sup T\notin T$ as $k\to\infty$.
By Theorem~\ref{t:SP}, applied to $\{f_j\}$ on $T\cap[a_1,b_1]$, there is a subsequence
$\{J_1(j)\}_{j=1}^\infty$ of $\{J_0(j)\}_{j=1}^\infty=\{j\}_{j=1}^\infty$ such that
$\{f_{J_1(j)}\}_{j=1}^\infty$ converges \pw\ on $T\cap[a_1,b_1]$ to a bounded
regulated function $f_1':T\cap[a_1,b_1]\to M$. Since
  \begin{equation*}
\limsup_{j\to\infty}V_\vep(f_{J_1(j)},T\cap[a_2,b_2])\le
\limsup_{j\to\infty}V_\vep(f_j,T\cap[a_2,b_2])\!<\!\infty\,\,\,\mbox{for all}\,\,\,\vep>0,
  \end{equation*}
applying Theorem~\ref{t:SP} to $\{f_{J_1(j)}\}_{j=1}^\infty$ on $T\cap[a_2,b_2]$,
we find a subsequence $\{J_2(j)\}_{j=1}^\infty$ of $\{J_1(j)\}_{j=1}^\infty$ such that
$\{f_{J_2(j)}\}_{j=1}^\infty$ converges \pw\ on the set $T\cap[a_2,b_2]$ to a bounded
regulated function $f_2':T\cap[a_2,b_2]\to M$. Since $[a_1,b_1]\subset[a_2,b_2]$, we get
$f_2'(t)=f_1'(t)$ for all $t\in T\cap[a_1,b_1]$. Proceeding this way, for each $k\in\Nb$ we
obtain a subsequence $\{J_k(j)\}_{j=1}^\infty$ of $\{J_{k-1}(j)\}_{j=1}^\infty$ and a
bounded regulated function $f_k':T\cap[a_k,b_k]\to M$ such that $f_{J_k(j)}\to f_k'$ on
$T\cap[a_k,b_k]$ as $j\to\infty$ and $f_k'(t)=f_{k-1}'(t)$ for all
$t\in T\cap[a_{k-1},b_{k-1}]$. Define $f:T\to M$ as follows: given $t\in T$, we have
$\inf T<t<\sup T$, so there is $k=k(t)\in\Nb$ such that $t\in T\cap[a_k,b_k]$, and so,
we set $f(t)=f_k'(t)$. The diagonal sequence $\{f_{J_j(j)}\}_{j=1}^\infty$ converges
\pw\ on $T$ to the function $f$, which satisfies the \emph{conclusions\/} of
Theorem~\ref{t:SPloc}.
\sq

Theorem~\ref{t:SP} implies immediately that if $(M,d)$ is a \emph{proper\/} metric space,
$\{f_j\}\subset M^T$ is a \pw\ \rc\ sequence and there is $E\subset T$ of measure zero,
$\mathcal{L}(E)=0$, such that $\limsup_{j\to\infty}V_\vep(f_j,T\setminus E)<\infty$ for
all $\vep>0$, then a subsequence of $\{f_j\}$ converges a.e.\ (=\,almost everywhere)
on $T$ to a function $f\in M^T$ such that $V_\vep(f,T\setminus E)<\infty$ for all
$\vep>0$. The following theorem is a \emph{selection principle for the a.e.\ convergence\/}
(it may be considered as subsequence-converse to Remark~\ref{r:neces}(b) concerning
the `almost necessity' of condition \eq{e:sp} for the \pw\ convergence).

\begin{theorem} \label{t:SPae}
Suppose $T\subset\Rb$, $(M,d)$ is a {\sl proper} metric space and $\{f_j\}\subset M^T$
is a \pw\ \rc\/ {\rm(}or a.e.\ \rc{\rm)}~on~$T$ sequence of functions satisfying the
condition\/{\rm:} for every $\eta>0$ there is a measurable set $E_\eta\subset T$ of
Lebesgue measure $\mathcal{L}(E_\eta)\le\eta$ such that
  \begin{equation} \label{e:spae}
\limsup_{j\to\infty}V_\vep(f_j,T\setminus E_\eta)<\infty\quad\mbox{for all}\quad\vep>0.
  \end{equation}
Then a subsequence of $\{f_j\}$ converges a.e.\ on $T$ to a function $f\in M^T$
having the property\/{\rm:} given $\eta>0$, there is a measurable set
$E'_\eta\subset T$ of Lebesgue measure $\mathcal{L}(E'_\eta)\le\eta$ such that
$V_\vep(f,T\setminus E'_\eta)<\infty$ for all $\vep>0$.
\end{theorem}

\proof
We follow the proof of Theorem~6 from \cite{JMAA05} with appropriate modifications.
Let $T_0\subset T$ be a set of Lebesgue measure zero such that the set
$\{f_j(t):j\in\Nb\}$ is \rc\ in $M$ for all $t\in T\setminus T_0$. We employ
Theorem~\ref{t:SP} several times as well as the diagonal procedure. By the assumption,
there is a measurable set $E_1\subset T$ of measure $\mathcal{L}(E_1)\le1$ such that
\eq{e:spae} holds with $\eta=1$. The sequence $\{f_j\}$ is \pw\ \rc\ on
$T\setminus(T_0\cup E_1)$ and, by Lemma~\ref{l:ele}(b), for all $\vep>0$,
  \begin{equation*}
\limsup_{j\to\infty}V_\vep(f_j,T\setminus(T_0\cup E_1))\le
\limsup_{j\to\infty}V_\vep(f_j,T\setminus E_1)<\infty.
  \end{equation*}
By Theorem~\ref{t:SP}, there are a subsequence $\{J_1(j)\}_{j=1}^\infty$ of
$\{j\}_{j=1}^\infty$ and a function $f^{(1)}:T\setminus(T_0\cup E_1)\to M$,
satisfying $V_\vep(f^{(1)},T\setminus(T_0\cup E_1))<\infty$ for all $\vep>0$,
such that $f_{J_1(j)}\to f^{(1)}$ \pw\ on $T\setminus(T_0\cup E_1)$ as $j\to\infty$.
Inductively, if $k\ge2$ and a subsequence $\{J_{k-1}(j)\}_{j=1}^\infty$ of
$\{j\}_{j=1}^\infty$ is already chosen, by the assumption \eq{e:spae}, there is a
measurable set $E_k\subset T$ with $\mathcal{L}(E_k)\le1/k$ such that
$\limsup_{j\to\infty}V_\vep(f_j,T\setminus E_k)<\infty$ for all $\vep>0$.
The sequence $\{f_{J_{k-1}(j)}\}_{j=1}^\infty$ is \pw\ \rc\ on
$T\setminus(T_0\cup E_k)$ and, again by Lemma~\ref{l:ele}(b), for all $\vep>0$,
  \begin{align*}
\!\!\!\!\limsup_{j\to\infty}V_\vep(f_{J_{k-1}(j)},T\setminus(T_0\cup E_k))&\le
\limsup_{j\to\infty}V_\vep(f_{J_{k-1}(j)},T\setminus E_k)\\[3pt]
&\le\limsup_{j\to\infty}V_\vep(f_j,T\setminus E_k)<\infty.
  \end{align*}
Theorem~\ref{t:SP} implies the existence of a subsequence $\{J_k(j)\}_{j=1}^\infty$
of $\{J_{k-1}(j)\}_{j=1}^\infty$ and a function $f^{(k)}:T\setminus(T_0\cup E_k)\to M$,
satisfying $V_\vep(f^{(k)},T\setminus(T_0\cup E_k))<\infty$ for all $\vep>0$, such that
$f_{J_k(j)}\to f^{(k)}$ \pw\ on $T\setminus(T_0\cup E_k)$ as $j\to\infty$.

The set $E=T_0\cup\bigcap_{k=1}^\infty E_k$ is of measure zero, and we have the
equality $T\setminus E=\bigcup_{k=1}^\infty(T\setminus(T_0\cup E_k))$. Define the
function $f:T\setminus E\to M$ as follows: given $t\in T\setminus E$, there is $k\in\Nb$
such that $t\in T\setminus(T_0\cup E_k)$, and so, we set $f(t)=f^{(k)}(t)$. The value
$f(t)$ is well-defined, i.e., it is independent of a particular $k$: in fact, if
$t\in T\setminus(T_0\cup E_{k_1})$ for some $k_1\in\Nb$ with, say, $k<k_1$ (with
no loss of generality), then, by the construction above, $\{J_{k_1}(j)\}_{j=1}^\infty$
is a subsequence of $\{J_k(j)\}_{j=1}^\infty$, which implies
  \begin{equation*}
f^{(k_1)}(t)=\lim_{j\to\infty}f_{J_{k_1}(j)}(t)=\lim_{j\to\infty}f_{J_k(j)}(t)
=f^{(k)}(t)\quad\mbox{in}\quad M.
  \end{equation*}
Let us show that the diagonal sequence $\{f_{J_j}(j)\}_{j=1}^\infty$ (which, of course,
is a subsequence of the original sequence $\{f_j\}$) converges \pw\ on $T\setminus E$
to the function~$f$. To see this, suppose $t\in T\setminus E$. Then
$t\in T\setminus(T_0\cup E_k)$ for some $k\in\Nb$, and so, $f(t)=f^{(k)}(t)$.
Since $\{f_{J_j(j)}\}_{j=k}^\infty$ is a subsequence of $\{f_{J_k(j)}\}_{j=1}^\infty$,
we find
  \begin{equation*}
\lim_{j\to\infty}f_{J_j(j)}(t)=\lim_{j\to\infty}f_{J_k(j)}(t)=f^{(k)}(t)=f(t)\quad
\mbox{in}\quad M.
  \end{equation*}
We extend $f$ from $T\setminus E$ to the whole $T$ arbitrarily and denote this extension
again by $f$. Given $\eta>0$, pick the minimal $k\in\Nb$ such that $1/k\le\eta$ and
set $E'_\eta=T_0\cup E_k$. It follows that $\mathcal{L}(E'_\eta)=\mathcal{L}(E_k)%
\le1/k\le\eta$, $f=f^{(k)}$ on $T\setminus(T_0\cup E_k)=T\setminus E'_\eta$, and
  \begin{equation*}
V_\vep(f,T\setminus E'_\eta)=V_\vep(f^{(k)},T\setminus(T_0\cup E_k))<\infty\quad
\mbox{for all}\quad\vep>0,
  \end{equation*}
which was to be proved.
\sq

\section{Weak pointwise selection principles} \label{ss:weak}

In this section, we establish a variant of Theorem~\ref{t:SP} for functions with values
in a reflexive Banach space taking into account some specific features of this case
(such as the validity of the weak \pw\ convergence of sequences of functions).

Suppose $(M,\|\cdot\|)$ is a normed linear space over the field $\Kb=\Rb$ or $\Cb$
(equipped with the absolute value $|\cdot|$) and $M^*$ is its \emph{dual}, i.e.,
$M^*=L(M;\Kb)$ is the space of all continuous (=\,bounded) linear functionals on~$M$.
Recall that $M^*$ is a Banach space under the norm
  \begin{equation} \label{e:nrm*}
\|x^*\|=\sup\bigl\{|x^*(x)|:\mbox{$x\in M$ and $\|x\|\le1$}\bigr\},\quad x^*\in M^*.
  \end{equation}
The natural duality between $M$ and $M^*$ is determined by the bilinear functional
$\lan\cdot,\cdot\ran:M\times M^*\to\Kb$ defined by $\lan x,x^*\ran=x^*(x)$ for all
$x\in M$ and $x^*\in M^*$. Recall also that a sequence $\{x_j\}\subset M$ is said to
\emph{converge weakly\/} in $M$ to an element $x\in M$, written as $x_j\wto x$ in $M$,
if $\lan x_j,x^*\ran\to\lan x,x^*\ran$ in $\Kb$ as $j\to\infty$ for all $x^*\in M^*$.
It is well known that if $x_j\wto x$ in $M$, then $\|x\|\le\liminf_{j\to\infty}\|x_j\|$.

The notion of the approximate variation $\{V_\vep(f,T)\}_{\vep>0}$ for $f\in M^T$ is
introduced as in \eq{e:av} with respect to the induced metrics $d(x,y)=\|x-y\|$,
$x,y\in M$, and $d_{\infty,T}(f,g)=\|f-g\|_{\infty,T}=\sup_{t\in T}\|f(t)-g(t)\|$,
$f,g\in M^T$.

\begin{theorem} \label{t:SPweak}
Let $T\subset\Rb$ and $(M,\|\cdot\|)$ be a reflexive Banach space with separable dual
$(M^*,\|\cdot\|)$. Suppose the sequence $\{f_j\}\subset M^T$ is such that
  \begin{itemize}
\renewcommand{\itemsep}{0.0pt plus 0.5pt minus 0.25pt}
\item[{\rm(i)}] $\sup_{j\in\Nb}\|f_j(t_0)\|\le C_0$ for some $t_0\in T$ and $C_0\ge0;$
\item[{\rm(ii)}] $v(\vep)\equiv\limsup_{j\to\infty}V_\vep(f_j,T)<\infty$ for all $\vep>0$.
  \end{itemize}
Then, there is a subsequence of $\{f_j\}$, again denoted by $\{f_j\}$, and a function
$f\in M^T$, satisfying $V_\vep(f,T)\le v(\vep)$ for all $\vep>0$ {\rm(}and, a fortiori,
$f$ is bounded and regulated on $T${\rm)}, such that $f_j(t)\wto f(t)$ in $M$
for all $t\in T$.
\end{theorem}

\proof
1. First, we show that there is $j_0\in\Nb$ such that $C(t)\equiv\sup_{j\ge j_0}\|f_j(t)\|$
is finite for all $t\in T$. In fact, by Lemma~\ref{l:71}(b), $|f_j(T)|\le V_\vep(f_j,T)+2\vep$
with, say, $\vep=1$, for all $j\in\Nb$, which implies
  \begin{equation*}
\limsup_{j\to\infty}|f_j(T)|\le\limsup_{j\to\infty}V_1(f_j,T)+2=v(1)+2<\infty\quad
\mbox{by (ii)},
  \end{equation*}
and so, there is $j_0\in\Nb$ and a constant $C_1>0$ such that $|f_j(T)|\le C_1$ for
all $j\ge j_0$. Now, given $j\ge j_0$ and $t\in T$, we get, by (i),
  \begin{equation*}
\|f_j(t)\|\le\|f_j(t_0)\|+\|f_j(t)-f_j(t_0)\|\le C_0+|f_j(T)|\le C_0+C_1,
  \end{equation*}
i.e., $C(t)\le C_0+C_1$ for all $t\in T$.

2. Given $j\in\Nb$ and $x^*\in M^*$, we set $f_j^{x^*}(t)=\lan f_j(t),x^*\ran$ for
all $t\in T$. Let us verify that the sequence $\{f_j^{x^*}\}_{j=j_0}^\infty\subset\Kb^T$
satisfies the assumptions of Theorem~\ref{t:SP}. By \eq{e:nrm*} and Step~1, we have
  \begin{equation} \label{e:bfu*}
\mbox{$|f_j^{x^*}(t)|\le\|f_j(t)\|\!\cdot\!\|x^*\|\le C(t)\|x^*\|$ \,\,\,for all \,\,$t\in T$
\,and \,$j\ge j_0$,}
  \end{equation}
and so, $\{f_j^{x^*}\}_{j=j_0}^\infty$ is \pw\ \rc\ on $T$. If $x^*=0$, then
$f_j^{x^*}=0$ in $\Kb^T$, which implies $V_\vep(f_j^{x^*},T)=0$ for all $j\in\Nb$
and $\vep>0$. Now, we show that if $x^*\ne0$, then
  \begin{equation} \label{e:Veofn}
\mbox{$V_\vep(f_j^{x^*},T)\le V_{\vep/\|x^*\|}(f_j,T)\|x^*\|$ \,\,\,for all 
\,\,\,$j\in\Nb$ \,and \,$\vep>0$.}
  \end{equation}
To prove \eq{e:Veofn}, we may assume that $V_{\vep/\|x^*\|}(f_j,T)<\infty$.
By definition \eq{e:av}, for every $\eta>0$ there is $g_j=g_{j,\eta}\in\BV(T;M)$
(also depending on $\vep$ and $x^*$) such that
  \begin{equation*}
\|f_j-g_j\|_{\infty,T}\le\vep/\|x^*\|\quad\mbox{and}\quad
V(g_j,T)\le V_{\vep/\|x^*\|}(f_j,T)+\eta.
  \end{equation*}
Setting $g_j^{x^*}(t)=\lan g_j(t),x^*\ran$ for all $t\in T$ (and so,
$g_j^{x^*}\in\Kb^T$), we find
  \begin{align*}
\bigl|f_j^{x^*}-g_j^{x^*}\bigr|_{\infty,T}&=\sup_{t\in T}|\lan f_j(t)-g_j(t),x^*\ran|
  \le\sup_{t\in T}\|f_j(t)-g_j(t)\|\!\cdot\!\|x^*\|\\[2pt]
&=\|f_j-g_j\|_{\infty,T}\|x^*\|\le(\vep/\|x^*\|)\|x^*\|=\vep.
  \end{align*}
Furthermore, it is straightforward that $V(g_j^{x^*},T)\!\le\! V(g_j,T)\|x^*\|$. Once
again from definition \eq{e:av}, it follows that
  \begin{equation*}
V_\vep(f_j^{x^*},T)\le V(g_j^{x^*},T)\le V(g_j,T)\|x^*\|\le
\bigl(V_{\vep/\|x^*\|}(f_j,T)+\eta\bigr)\|x^*\|.
  \end{equation*}
Passing to the limit as $\eta\to+0$, we arrive at \eq{e:Veofn}.
Now, by \eq{e:Veofn} and (ii), for every $\vep>0$ and $x^*\in M^*$, $x^*\ne0$,
we have
  \begin{equation} \label{e:vue*}
v_{x^*}(\vep)\equiv\limsup_{j\to\infty}V_\vep(f_j^{x^*},T)
\le v(\vep/\|x^*\|)\|x^*\|<\infty.
  \end{equation}

Taking into account \eq{e:bfu*} and \eq{e:vue*}, given $x^*\in M^*$, we may
apply Theorem~\ref{t:SP} to the sequence $\{f_j^{x^*}\}_{j=j_0}^\infty\subset\Kb^T$
and extract a subsequence $\{f_{j,x^*}\}_{j=1}^\infty$ (depending on $x^*$
as well) of $\{f_j\}_{j=j_0}^\infty$ and find a function $f_{x^*}\in\Kb^T$, satisfying
$V_\vep(f_{x^*},T)\le v_{x^*}(\vep)$ for all $\vep>0$ (and so, $f_{x^*}$ is bounded
and regulated on $T$), such that $\lan f_{j,x^*}(t),x^*\ran\!\to\! f_{x^*}(t)$ in $\Kb$
as $j\!\to\!\infty$ for all \mbox{$t\in T$}.

3. Making use of the diagonal procedure, we are going to get rid of the dependence
of $\{f_{j,x^*}\}_{j=1}^\infty$ on the element $x^*\in M^*$. Since $M^*$ is separable,
$M^*$ contains a countable dense subset $\{x^*_k\}_{k=1}^\infty$. Setting
$x^*=x_1^*$ in \eq{e:bfu*} and \eq{e:vue*} and applying Theorem~\ref{t:SP} to the
sequence of functions $f_j^{x_1^*}=\lan f_j(\cdot),x_1^*\ran\in\Kb^T$, $j\ge j_0$,
we obtain a subsequence $\{J_1(j)\}_{j=1}^\infty$ of $\{j\}_{j=j_0}^\infty$ and a
function $f_{x_1^*}\in\Kb^T$ (both depending on $x_1^*$), satisfying
$V_\vep(f_{x_1^*},T)\le v_{x_1^*}(\vep)$ for all $\vep>0$, such that
$\lan f_{J_1(j)}(t),x_1^*\ran\to f_{x_1^*}(t)$ in $\Kb$ as $j\to\infty$ for all $t\in T$.
Inductively, assume that $k\ge2$ and the subsequence $\{J_{k-1}(j)\}_{j=1}^\infty$
of $\{j\}_{j=j_0}^\infty$ is already chosen. Putting $x^*=x_k^*$, replacing
$j$ by $J_{k-1}(j)$ in \eq{e:bfu*} and taking into account \eq{e:vue*}, we get 
  \begin{equation*}
\sup_{j\in\Nb}\bigl|\lan f_{J_{k-1}(j)}(t),x_k^*\ran\bigr|\le C(t)\|x_k^*\|\quad
\mbox{for all}\quad t\in T
  \end{equation*}
and
  \begin{equation*}
\limsup_{j\to\infty}V_\vep\bigl(\lan f_{J_{k-1}(j)}(\cdot),x_k^*\ran,T\bigr)
\le v_{x_k^*}(\vep)<\infty\quad\mbox{for all}\quad\vep>0.
  \end{equation*}
By Theorem~\ref{t:SP}, applied to the sequence $\{f_{J_{k-1}(j)}^{x_k^*}\}_{j=1}^%
\infty\subset\Kb^T$, there are a subsequence $\{J_k(j)\}_{j=1}^\infty$ of
$\{J_{k-1}(j)\}_{j=1}^\infty$ and a function $f_{x_k^*}\in\Kb^T$, satisfying
$V_\vep(f_{x_k^*},T)\le v_{x_k^*}(\vep)$ for all $\vep>0$, such that
$\lan f_{J_k(j)}(t),x_k^*\ran\to f_{x_k^*}(t)$ in $\Kb$ as $j\to\infty$ for all $t\in T$.
It follows that the diagonal subsequence $\{f_{J_j(j)}\}_{j=1}^\infty$ of
$\{f_j\}_{j=j_0}^\infty$, denoted by $\{f_j\}_{j=1}^\infty$, satisfies the condition:
  \begin{equation} \label{e:fjfu}
\mbox{$\D\lim_{j\to\infty}\lan f_j(t),x_k^*\ran=f_{x_k^*}(t)$ \,\,\,for all\,\,\,
$t\in T$ \,and \,$k\in\Nb$.}
  \end{equation}

4. Let us show that the sequence $f_j^{x^*}(t)=\lan f_j(t),x^*\ran$, $j\in\Nb$, is
Cauchy in $\Kb$ for every $x^*\in M^*$ and $t\in T$. Since the sequence
$\{x_k^*\}_{k=1}^\infty$ is dense in $M^*$, given $\eta>0$, there is
$k=k(\eta)\in\Nb$ such that $\|x^*-x_k^*\|\le\eta/(4C(t)+1)$, and, by \eq{e:fjfu},
there is $j^0=j^0(\eta)\in\Nb$ such that
$\bigl|\lan f_j(t),x_k^*\ran-\lan f_{j'}(t),x_k^*\ran\bigr|\le\eta/2$ for all $j,j'\ge j^0$.
Hence
  \begin{align*}
|f_j^{x^*}(t)-f_{j'}^{x^*}(t)|&\le\|f_j(t)-f_{j'}(t)\|\!\cdot\!\|x^*-x_k^*\|
  +|\lan f_j(t),x_k^*\ran-\lan f_{j'}(t),x_k^*\ran|\\[3pt]
&\le2C(t)\frac{\eta}{4C(t)+1}+\frac\eta2\le\eta\quad\mbox{for all}\quad j,j'\ge j^0.
  \end{align*}
By the completeness of $(\Kb,|\cdot|)$, there is $f_{x^*}(t)\in\Kb$ such that
$f_j^{x^*}(t)\to f_{x^*}(t)$ in $\Kb$ as $j\to\infty$. Thus, we have shown that
for every $x^*\in M^*$ there is a function $f_{x^*}\in\Kb^T$ satisfying,
by virtue of Lemma~\ref{l:proper}(c), \eq{e:Veofn} and \eq{e:vue*},
  \begin{equation*}
V_\vep(f_{x^*},T)\le\liminf_{j\to\infty}V_\vep(f_j^{x^*},T)\le v_{x^*}(\vep)
\quad\mbox{for all}\quad\vep>0
  \end{equation*}
and such that
  \begin{equation} \label{e:uu*}
\mbox{$\D\lim_{j\to\infty}\lan f_j(t),x^*\ran=f_{x^*}(t)$ \,in \,$\Kb$
\,\,\,for all \,\,\,$t\in T$.}
  \end{equation}

5. Now, we show that, for every $t\in T$, the sequence $\{f_j(t)\}$ converges
weakly in $M$ to an element of $M$. The reflexivity of $M$ implies
  \begin{equation*}
f_j(t)\in M=M^{**}=L(M^*;\Kb)\quad\mbox{for all}\quad j\in\Nb.
  \end{equation*}
Define the functional $F_t:M^*\to\Kb$ by $F_t(x^*)=f_{x^*}(t)$ for all $x^*\in M^*$.
It follows from \eq{e:uu*} that
  \begin{equation*}
\lim_{j\to\infty}\lan f_j(t),x^*\ran=f_{x^*}(t)=F_t(x^*)\quad\mbox{for all}
\quad x^*\in M^*,
  \end{equation*} 
i.e., the sequence of functionals $\{f_j(t)\}\subset L(M^*;\Kb)$ converges \pw\
on $M^*$ to the functional $F_t:M^*\to\Kb$. By the uniform boundedness principle,
$F_t\in L(M^*;\Kb)$ and $\|F_t\|\le\liminf_{j\to\infty}\|f_j(t)\|$. Setting $f(t)=F_t$
for all $t\in T$, we find $f\in M^T$ and, for all $x^*\in M^*$ and $t\in T$,
  \begin{equation} \label{e:wto}
\lim_{j\to\infty}\lan f_j(t),x^*\ran=F_t(x^*)=\lan F_t,x^*\ran=\lan f(t),x^*\ran,
  \end{equation}
which means that $f_j(t)\wto f(t)$ in $M$ for all $t\in T$. (Note that \eq{e:uu*} and
\eq{e:wto} imply $f_{x^*}(t)=\lan f(t),x^*\ran$ for all $x^*\in M^*$ and $t\in T$.)

6. It remains to prove that $V_\vep(f,T)\le v(\vep)$ for all $\vep>0$. Recall that the
sequence $\{f_j\}\subset M^T$, we deal with here, is the diagonal sequence
$\{f_{J_j(j)}\}_{j=1}^\infty$ from the end of Step~3, which satisfies conditions
\eq{e:wto} and, in place of (ii),
  \begin{equation} \label{e:newve}
\limsup_{j\to\infty}V_\vep(f_j,T)\le v(\vep)\quad\mbox{for all}\quad\vep>0.
  \end{equation}
Let us fix $\vep>0$. Since $v(\vep)<\infty$ by (ii), for every $\eta>v(\vep)$ condition
\eq{e:newve} implies the existence of $j_1=j_1(\eta,\vep)\in\Nb$ such that
$\eta>V_\vep(f_j,T)$ for all $j\ge j_1$. Hence, for every $j\ge j_1$, by the definition
of $V_\vep(f_j,T)$, there is $g_j\in\BV(T;M)$ such that
  \begin{equation} \label{e:gjj}
\|f_j-g_j\|_{\infty,T}=\sup_{t\in T}\|f_j(t)-g_j(t)\|\le\vep\quad\mbox{and}\quad
V(g_j,T)\le\eta.
  \end{equation}
These conditions and assumption (i) imply $\sup_{j\ge j_1}V(g_j,T)\le\eta$ and
  \begin{equation*}
\|g_j(t_0)\|\le\|g_j(t_0)-f_j(t_0)\|+\|f_j(t_0)\|\le\|g_j-f_j\|_{\infty,T}+C_0\le\vep+C_0
  \end{equation*}
for all $j\ge j_1$. Since $(M,\|\cdot\|)$ is a reflexive Banach space with separable dual
$M^*$, by the weak Helly-type pointwise selection principle (see Theorem~7 and
Remarks (1)--(4) in \cite{JMAA05}, or Theorem~3.5 in \cite[Chapter~1]{Barbu}),
there are a subsequence $\{g_{j_p}\}_{p=1}^\infty$ of $\{g_j\}_{j=j_1}^\infty$ and
a function $g\in\BV(T;M)$ such that $g_{j_p}\wto g(t)$ in $M$ as $p\to\infty$ for all
$t\in T$. Noting that $f_{j_p}(t)\wto f(t)$ in $M$ as $p\to\infty$ for all $t\in T$ as well,
we get $f_{j_p}(t)-g_{j_p}(t)\wto f(t)-g(t)$ in $M$ as $p\to\infty$, and so, taking into
account the first condition in \eq{e:gjj}, we find
  \begin{equation*}
\|f(t)-g(t)\|\le\liminf_{p\to\infty}\|f_{j_p}(t)-g_{j_p}(t)\|\le\vep\quad\mbox{for all}
\quad t\in T,
  \end{equation*}
which implies $\|f-g\|_{\infty,T}\le\vep$. Had we already shown that $V(g,T)\le\eta$,
definition \eq{e:av} would yield $V_\vep(f,T)\le V(g,T)\le\eta$ for every $\eta>v(\vep)$,
which completes the proof of Theorem~\ref{t:SPweak}.

In order to prove that $V(g,T)\le\eta$, suppose $P=\{t_i\}_{i=0}^m\subset T$ is a
partition of $T$. Since $g_{j_p}(t)\wto g(t)$ in $M$ as $p\to\infty$ for all $t\in T$,
given $i\in\{1,2,\dots,m\}$, we have $g_{j_p}(t_i)-g_{j_p}(t_{i-1})\wto g(t_i)-g(t_{i-1})$
in $M$ as $p\to\infty$, and so,
  \begin{equation*}
\|g(t_i)-g(t_{i-1})\|\le\liminf_{p\to\infty}\|g_{j_p}(t_i)-g_{j_p}(t_{i-1})\|.
  \end{equation*}
Summing over $i=1,2,\dots,m$ and taking into account the properties of the limit inferior
and the second condition in \eq{e:gjj}, we get
  \begin{align*}
\sum_{i=1}^m\|g(t_i)-g(t_{i-1})\|&\le\sum_{i=1}^m\liminf_{p\to\infty}%
  \|g_{j_p}(t_i)-g_{j_p}(t_{i-1})\|\\
&\le\liminf_{p\to\infty}\sum_{i=1}^m\|g_{j_p}(t_i)-g_{j_p}(t_{i-1})\|\\[3pt]
&\le\liminf_{p\to\infty}V(g_{j_p},T)\le\eta.
  \end{align*}
Thus, by the arbitrariness of partition $P$ of $T$, we conclude that $V(g,T)\le\eta$,
which was to be proved.
\sq

Assumption (ii) in Theorem~\ref{t:SPweak} can be weakened as the following theorem
shows.

\begin{theorem} \label{t:SPw2}
Under the assumptions of Theorem\/~{\rm\ref{t:SPweak}} on $T$ and $(M,\|\cdot\|)$,
suppose the sequence $\{f_j\}\subset M^T$ is such that
  \begin{itemize}
\renewcommand{\itemsep}{0.0pt plus 0.5pt minus 0.25pt}
\item[{\rm(i)}] $C(t)\equiv\sup_{j\in\Nb}\|f_j(t)\|<\infty$ for all $t\in T;$
\item[{\rm(ii)}] $v_{x^*}(\vep)\equiv\limsup_{j\to\infty}%
  V_\vep(\lan f_j(\cdot),x^*\ran,T)\!<\!\infty$ for all $\vep\!>\!0$ and $x^*\in M^*$.%
\footnote{As in Step~2 of the proof of Theorem~\ref{t:SPweak},
$\lan f_j(\cdot),x^*\ran(t)=\lan f_j(t),x^*\ran=f_j^{x^*}(t)$, $t\in T$.}
  \end{itemize}
Then, there is a subsequence of $\{f_j\}$, again denoted by $\{f_j\}$, and a function
$f\in M^T$, satisfying $V_\vep(\lan f(\cdot),x^*\ran,T)\le v_{x^*}(\vep)$ for all
$\vep>0$ and $x^*\in M^*$, such that $f_j(t)\wto f(t)$ in $M$ as $j\to\infty$
for all $t\in T$.
\end{theorem}

\proof
It suffices to note that assumption (i) implies \eq{e:bfu*} with $j_0=1$, replace
\eq{e:vue*} by assumption (ii), and argue as in Steps~3--5 of the proof of
Theorem~\ref{t:SPweak}.
\sq

The next example illustrates the applicability of Theorems~\ref{t:SPweak} and \ref{t:SPw2}.

\begin{example} \label{ex:} \rm
In examples (a) and (b) below, we assume the following. Let $M=L^2[0,2\pi]$ be the
real Hilbert space of all square Lebesgue summable functions on the interval $[0,2\pi]$
equipped with the \emph{inner product}
  \begin{equation*}
\lan x,y\ran=\int_0^{2\pi}\!\!x(s)y(s)\,ds\quad\!\!\mbox{and the \emph{norm}}
\quad\!\!\|x\|=\sqrt{\lan x,x\ran},\,\,\,x,y\in M.
  \end{equation*}
It is well known that $M$ is separable, self-adjoint ($M=M^*$), and so, reflexive
($M=M^{**}$). Given $j\in\Nb$, define two functions $x_j,y_j\in M$ by
  \begin{equation*}
x_j(s)=\sin(js)\quad\mbox{and}\quad y_j(s)=\cos(js)\quad\mbox{for all}\quad
s\in[0,2\pi].
  \end{equation*}
Clearly, $\|x_j\|=\|y_j\|=\sqrt\pi$ and, by Lyapunov-Parseval's equality,
  \begin{equation*}
\frac{\lan x,1\ran^2}8+\sum_{j=1}^\infty\Bigl(\lan x,x_j\ran^2+
\lan x,y_j\ran^2\Bigr)=\pi\|x\|^2,\quad x\in M,
  \end{equation*}
we find $\lan x,x_j\ran\to0$ and $\lan x,y_j\ran\to0$ as $j\to\infty$ for all $x\in M$,
and so, $x_j\wto0$ and $y_j\wto 0$ in $M$.

In examples (a) and (b) below, we set $T=I=[0,1]$.

(a) This example illustrates Theorem~\ref{t:SPweak}. Define the sequence
$\{f_j\}\subset M^T$ by $f_j(t)=tx_j$, $t\in T$. Clearly, $f_j(t)\wto0$ in $M$ for all
$t\in T$. Note, however, that the sequence $\{f_j(t)\}$ does \emph{not\/} converge
in (the norm of) $M$ at all points $0<t\le1$, because
$\|f_j(t)\!-\!f_k(t)\|^2=(\|x_j\|^2\!+\!\|x_k\|^2)t^2=2\pi t^2$,~\mbox{$j\ne k$}.

Since $f_j(0)=0$ in $M$ for all $j\in\Nb$, we verify only condition (ii) of Theorem~%
\ref{t:SPweak}. Setting $\vfi(t)=t$ for $t\in T$, $x=x_j$ and $y=0$ in \eq{e:fxy}, we find
$|\vfi(T)|=1$ and $\|x\|=\|x_j\|=\sqrt\pi$, and so, by virtue of \eq{e:mntn}, we get
  \begin{equation*}
V_\vep(f_j,T)=\left\{
  \begin{tabular}{ccr}
$\!\!\sqrt\pi-2\vep$ & \mbox{if} & $0<\vep<\sqrt\pi/2$\\[3pt]
$\!\!0$ & \mbox{if} & $\vep\ge\sqrt\pi/2$
  \end{tabular}\right.\quad\mbox{for all}\quad j\in\Nb,
  \end{equation*}
which implies condition (ii) in Theorem~\ref{t:SPweak}. Note that (cf.\
Lemma~\ref{l:71}(a)) $V(f_j,T)=\lim_{\vep\to+0}V_\vep(f_j,T)=\sqrt\pi$ for all
$j\in\Nb$. Also, it is to be noted that Theorem~\ref{t:SP} is inapplicable to $\{f_j\}$,
because the set $\{f_j(t):j\in\Nb\}$ is not \rc\ in (the norm of) $M$ for all $0<t\le1$.

(b) Here we present an example when Theorem~\ref{t:SPw2} is applicable, while
Theorem~\ref{t:SPweak} is not. Taking into account definition \eq{e:Dir} of the
Dirichlet function, we let the sequence $\{f_j\}\subset M^T$ be given by
$f_j(t)=\Dc_{x_j,y_j}(t)$ for all $t\in T$ and $j\in\Nb$. More explicitly,
  \begin{equation*}
f_j(t)(s)=\Dc_{x_j(s),y_j(s)}(t)=\left\{
  \begin{tabular}{ccl}
$\!\!\sin(js)$ & \!\!\mbox{if}\!\! & $t\in I_1\equiv[0,1]\cap\Qb$,\\[3pt]
$\!\!\cos(js)$ & \!\!\mbox{if}\!\! & $t\in I_2\equiv[0,1]\setminus\Qb$,
  \end{tabular}\right.\quad\!\! s\in[0,2\pi].
  \end{equation*}
Note that
  \begin{equation} \label{e:D01}
f_j(t)=\Dc_{x_j,0}(t)+\Dc_{0,y_j}(t)=\Dc_{1,0}(t)x_j+\Dc_{0,1}(t)y_j,\quad t\in T,
  \end{equation}
where $\Dc_{1,0}$ and $\Dc_{0,1}$ are the corresponding real-valued Dirichlet
functions on $T=[0,1]$. By \eq{e:D01}, $f_j(t)\wto0$ in $M$ for all $t\in T$.
On the other hand, the sequence $\{f_j(t)\}$ \emph{diverges\/} in (the norm of) $M$
at all points $t\in T$: in fact,
  \begin{equation*}
\|x_j-x_k\|^2=\lan x_j-x_k,x_j-x_k\ran=\|x_j\|^2+\|x_k\|^2=2\pi,\quad j\ne k,
  \end{equation*}
and, similarly, $\|y_j-y_k\|^2=2\pi$, $j\ne k$, from which we get
  \begin{equation*}
\|f_j(t)-f_k(t)\|=\left\{
  \begin{tabular}{ccl}
$\!\!\|x_j-x_k\|$ & \!\!\mbox{if}\!\! & $t\in I_1$\\[3pt]
$\!\!\|y_j-y_k\|$ & \!\!\mbox{if}\!\! & $t\in I_2$
  \end{tabular}\right.\!\!=\sqrt{2\pi},\quad j\ne k.
  \end{equation*}
(It already follows that Theorem~\ref{t:SP} is inapplicable to $\{f_j\}$.)

Given $t\in T$ and $j\in\Nb$, we have
  \begin{equation*}
\|f_j(t)\|=\|\Dc_{x_j,y_j}(t)\|=\left\{
  \begin{tabular}{ccl}
$\!\!\|x_j\|$ & \!\!\mbox{if}\!\! & $t\in I_1$\\[3pt]
$\!\!\|y_j\|$ & \!\!\mbox{if}\!\! & $t\in I_2$
  \end{tabular}\right.\!\!=\sqrt\pi,
  \end{equation*}
and so, conditions (i) in Theorems~\ref{t:SPweak} and \ref{t:SPw2} are satisfied.

Let us show that condition (ii) in Theorem~\ref{t:SPweak} does not hold. In fact,
by \eq{e:Dirass} and \eq{e:refi},
  \begin{equation*}
V_\vep(f_j,T)=\left\{
  \begin{tabular}{ccr}
$\!\!\infty$ & \!\!\mbox{if}\!\! & $0<\vep<\textstyle\frac12\|x_j-y_j\|$,\\[3pt]
$\!\!0$ & \!\!\mbox{if}\!\! & $\vep\ge\textstyle\frac12\|x_j-y_j\|$,
  \end{tabular}\right.
  \end{equation*}
where $\|x_j\!-\!y_j\|^2\!=\!\lan x_j\!-\!y_j,x_j-y_j\ran\!=\!%
\|x_j\|^2\!+\!\|y_j\|^2\!=\!2\pi$,
i.e., $\|x_j\!-\!y_j\|\!=\!\sqrt{2\pi}$.

Now, we show that condition (ii) in Theorem~\ref{t:SPw2} is satisfied (cf.\ Example~%
\ref{ex:irreg}). By \eq{e:D01}, for every $x^*\in M^*=M$ and $t\in T$, we have
  \begin{align*}
\lan f_j(t),x^*\ran&=\lan\Dc_{x_j,y_j}(t),x^*\ran=
  \lan\Dc_{1,0}(t)x_j+\Dc_{0,1}(t)y_j,x^*\ran\\[3pt]
&=\Dc_{1,0}(t)\lan x_j,x^*\ran+\Dc_{0,1}(t)\lan y_j,x^*\ran=\Dc_{x_j',y_j'}(t),
  \end{align*}
where $x_j'=\lan x_j,x^*\ran$ and $y_j'=\lan y_j,x^*\ran$. Again, by \eq{e:Dirass}
and \eq{e:refi},
  \begin{equation} \label{e:xpyp}
V_\vep(\lan f_j(\cdot),x^*\ran,T)=V_\vep(\Dc_{x_j',y_j'},T)=\left\{
  \begin{tabular}{ccr}
$\!\!\infty$ & \!\!\mbox{if}\!\! & $0<\vep<\textstyle\frac12|x_j'-y_j'|$,\\[3pt]
$\!\!0$ & \!\!\mbox{if}\!\! & $\vep\ge\textstyle\frac12|x_j'-y_j'|$,
  \end{tabular}\right.
  \end{equation}
where $|x_j'-y_j'|=|\lan x_j,x^*\ran-\lan y_j,x^*\ran|\to0$ as $j\to\infty$. Hence, given
$\vep>0$,
there is $j_0=j_0(\vep,x^*)\in\Nb$ such that $|x_j'-y_j'|\le2\vep$ for all
$j\ge j_0$, and  so, \eq{e:xpyp} implies $V_\vep(\lan f_j(\cdot),x^*\ran,T)=0$ for all
$j\ge j_0$. Thus,
  \begin{equation*}
\limsup_{j\to\infty}V_\vep(\lan f_j(\cdot),x^*\ran,T)\le\sup_{j\ge j_0}
V_\vep(\lan f_j(\cdot),x^*\ran,T)=0
  \end{equation*}
(i.e., $V_\vep(\lan f_j(\cdot),x^*\ran,T)\!\to\!0$, $j\!\to\!\infty$), which yields
condition (ii) in \mbox{Theorem~\ref{t:SPw2}}.
\end{example}

\section{Irregular pointwise selection principles} \label{ss:irreg}

In what follows, we shall be dealing with double sequences of the form
$\al:\Nb\times\Nb\to[0,\infty]$ having the property that $\al(j,j)=0$ for all $j\in\Nb$
(e.g., \eq{e:spir}). The \emph{limit superior\/} of $\al(j,k)$ as $j,k\to\infty$ is defined by
  \begin{equation*}
\limsup_{j,k\to\infty}\al(j,k)=\lim_{n\to\infty}\sup\bigl\{\al(j,k):
\mbox{$j\ge n$ and $k\ge n$}\bigr\}.
  \end{equation*}
For a number $\al_0\ge0$, we say that $\al(j,k)$ \emph{converges\/} to $\al_0$ as
$j,k\to\infty$ and write $\lim_{j,k\to\infty}\al(j,k)=\al_0$ if for every $\eta>0$ there is
$J=J(\eta)\in\Nb$ such that $|\al(j,k)-\al_0|\le\eta$ for all $j\ge J$ and $k\ge J$ with
$j\ne k$.

The main result of this section is the following \emph{irregular \pw\ selection principle\/}
in terms of the approximate variation (see also Example~\ref{ex:nrg}).

\begin{theorem} \label{t:SPir}
Suppose $T\subset\Rb$, $(M,\|\cdot\|)$ is a normed linear space, and
$\{f_j\}\subset M^T$ is a \pw\ \rc\ sequence of functions such that
  \begin{equation} \label{e:spir}
\limsup_{j,k\to\infty}V_\vep(f_j-f_k,T)<\infty\quad\mbox{for all}\quad\vep>0.
  \end{equation}
Then $\{f_j\}$ contains a subsequence which converges pointwise on $T$.
\end{theorem}

In order to prove this theorem, we need a lemma.

\begin{lemma} \label{l:RT}
Suppose $\vep>0$, $C>0$, and a sequence $\{F_j\}_{j=1}^\infty\subset M^T$ of
{\sl distinct} functions are such that
  \begin{equation} \label{e:Fjk}
V_\vep(F_j-F_k,T)\le C\quad\mbox{for all}\quad j,k\in\Nb.
  \end{equation}
Then, there exist a subsequence $\{F_j^\vep\}_{j=1}^\infty$ of $\{F_j\}_{j=1}^\infty$
and a nondecreasing function $\vfi^\vep:T\to[0,C]$ such that
  \begin{equation} \label{e:fek}
\lim_{j,k\to\infty}V_\vep(F_j^\vep-F_k^\vep,T\cap(-\infty,t])=\vfi^\vep(t)\quad
\mbox{for all}\quad t\in T.
  \end{equation}
\end{lemma}

Since the proof of Lemma~\ref{l:RT} is rather lengthy and involves certain ideas from
formal logic (Ramsey's Theorem \ref{t:Ramsey}), for the time being we postpone
it until the end of the proof of Theorem~\ref{t:SPir}.

\begin{proof}[Proof of Theorem~\protect\ref{t:SPir}]
First, we may assume that $T$ is \emph{uncountable}. In fact, if $T$ is (at most)
countable, then, by the relative compactness of sets $\{f_j(t):j\in\Nb\}\subset M$
for all $t\in T$, we may apply the standard diagonal procedure to extract a subsequence
of $\{f_j\}$ which converges \pw\ on~$T$. Second, we may assume that all functions
in the sequence $\{f_j\}$ are \emph{distinct}. To see this, we argue as follows. If there
are only finitely many distinct functions in $\{f_j\}$, then we may choose a constant
subsequence of $\{f_j\}$ (which is, clearly, \pw\ convergent on $T$). Otherwise,
we may pick a subsequence of $\{f_j\}$ (if necessary) consisting of distinct functions.

Given $\vep>0$, we set (cf.\ \eq{e:spir})
  \begin{equation*}
C(\vep)=1+\limsup_{j,k\to\infty}V_\vep(f_j-f_k,T)<\infty.
  \end{equation*}
So, there is $j_0(\vep)\in\Nb$ such that
  \begin{equation} \label{e:ejkC}
\mbox{$V_\vep(f_j-f_k,T)\le C(\vep)$ \,\,\,for all \,\,\,$j\ge j_0(\vep)$
\,and \,$k\ge j_0(\vep)$.}
  \end{equation}
Let $\{\vep_n\}_{n=1}^\infty\!\subset\!(0,\infty)$ be a decreasing sequence such that
$\vep_n\!\to\!0$ as $n\!\to\!\infty$.

We divide the rest of the proof into two main steps for clarity.

\emph{Step~1.} There is a subsequence of $\{f_j\}$, again denoted by $\{f_j\}$, and
for each $n\in\Nb$ there is a nondecreasing function $\vfi_n:T\to[0,C(\vep_n)]$ such that
  \begin{equation} \label{e:fint}
\lim_{j,k\to\infty}V_{\vep_n}(f_j-f_k,T\cap(-\infty,t])=\vfi_n(t)\quad\mbox{for all}
\quad t\in T.
  \end{equation}

In order to prove \eq{e:fint}, we apply Lemma~\ref{l:RT}, induction and the diagonal
procedure. Setting $\vep=\vep_1$, $C=C(\vep_1)$ and $F_j=f_{J_0(j)}$ with
$J_0(j)=j_0(\vep_1)+j-1$, $j\in\Nb$, we find that condition \eq{e:ejkC} implies
\eq{e:Fjk}, and so, by Lemma~\ref{l:RT}, there are a subsequence
$\{J_1(j)\}_{j=1}^\infty$ of $\{J_0(j)\}_{j=1}^\infty=\{j\}_{j=j_0(\vep_1)}^\infty$
and a nondecreasing function $\vfi_1=\vfi^{\vep_1}:T\to[0,C(\vep_1)]$ such that
  \begin{equation*}
\lim_{j,k\to\infty}V_{\vep_1}(f_{J_1(j)}-f_{J_1(k)},T\cap(-\infty,t])=\vfi_1(t)\quad
\mbox{for all}\quad t\in T.
  \end{equation*}
Let $j_1\in\Nb$ be the least number such that $J_1(j_1)\ge j_0(\vep_2)$. Inductively,
suppose $n\in\Nb$, $n\ge2$, and a subsequence $\{J_{n-1}(j)\}_{j=1}^\infty$ of
$\{j\}_{j=j_0(\vep_1)}^\infty$ and the number $j_{n-1}\in\Nb$ with
$J_{n-1}(j_{n-1})\ge j_0(\vep_n)$ are already chosen. To apply Lemma~\ref{l:RT}
once again, we set $\vep=\vep_n$, $C=C(\vep_n)$ and $F_j=f_{J(j)}$ with
$J(j)=J_{n-1}(j_{n-1}+j-1)$, $j\in\Nb$. Since for every $j\in\Nb$ we have
$J(j)\ge J_{n-1}(j_{n-1})\ge j_0(\vep_n)$, we get, by \eq{e:ejkC},
  \begin{equation*}
V_{\vep_n}(F_j-F_k,T)\le C(\vep_n)\quad\mbox{for all}\quad j,k\in\Nb.
  \end{equation*}
By Lemma~\ref{l:RT}, there are a subsequence $\{J_n(j)\}_{j=1}^\infty$ of the
sequence $\{J(j)\}_{j=1}^\infty$, (more explicitly) the latter being equal to
$\{J_{n-1}(j)\}_{j=j_{n-1}}^\infty$, and a nondecreasing function 
$\vfi_n=\vfi^{\vep_n}:T\to[0,C(\vep_n)]$ such that
  \begin{equation} \label{e:JJ}
\lim_{j,k\to\infty}V_{\vep_n}(f_{J_n(j)}-f_{J_n(k)},T\cap(-\infty,t])=\vfi_n(t)\quad
\mbox{for all}\quad t\in T.
  \end{equation}
We assert that the diagonal subsequence $\{f_{J_j(j)}\}_{j=1}^\infty$ of $\{f_j\}$,
again denoted by $\{f_j\}$, satisfies \eq{e:fint} for all $n\in\Nb$. In order to see this,
let us fix $n\in\Nb$ and $t\in T$. By \eq{e:JJ}, given $\eta>0$, there is a number
$J^0=J^0(\eta,n,t)\in\Nb$ such that if $j',k'\ge J^0$, $j'\ne k'$, we have
  \begin{equation} \label{e:pp}
\bigl|V_{\vep_n}(f_{J_n(j')}-f_{J_n(k')},T\cap(-\infty,t])-\vfi_n(t)\bigr|\le\eta.
  \end{equation}
Since $\{J_j(j)\}_{j=n}^\infty$ is a subsequence of $\{J_n(j)\}_{j=1}^\infty$, there is a
strictly increasing natural sequence $q:\Nb\to\Nb$ such that $J_j(j)=J_n(q(j))$ for all
$j\ge n$. Define $J^*=\max\{n,J^0\}$. Now, for arbitrary $j,k\ge J^*$, $j\ne k$, we
set $j'=q(j)$ and $k'=q(k)$. Since $j,k\ge J^*\ge n$, we find $J_j(j)=J_n(j')$ and
$J_k(k)=J_n(k')$, where $j'\ne k'$, $j'=q(j)\ge j\ge J^*\ge J^0$ and, similarly,
$k'\ge J^0$. It follows from \eq{e:pp} that
  \begin{equation*}
\bigl|V_{\vep_n}(f_{J_j(j)}-f_{J_k(k)},T\cap(-\infty,t])-\vfi_n(t)\bigr|\le\eta.
  \end{equation*}
which proves our assertion.

\emph{Step~2.} Let $Q$ denote an at most countable dense subset of $T$. Clearly,
$Q$ contains every point of $T$ which is not a limit point for $T$. Since, for any $n\in\Nb$,
the function $\vfi_n$ from \eq{e:fint} is nondecreasing on $T$, the set $Q_n\subset T$
of its points of discontinuity is at most countable. We set
$S=Q\cup\bigcup_{n=1}^\infty Q_n$. The set $S$ is an at most countable dense
subset of $T$ and has the property:
  \begin{equation} \label{e:TmS}
\mbox{for each $n\in\Nb$, the function $\vfi_n$ is continuous on $T\setminus S$.}
  \end{equation}
By the relative compactness of the set $\{f_j(t):j\in\Nb\}$ for all $t\in T$ and at most
countability of $S\subset T$, we may assume (applying the diagonal procedure and
passing to a subsequence of $\{f_j\}$ if necessary) that, for every $s\in S$, $f_j(s)$
converges in $M$ as $j\to\infty$ to a point of $M$ denoted by $f(s)$ (hence $f:S\to M$).

It remains to show that the sequence $\{f_j(t)\}_{j=1}^\infty$ is Cauchy in $M$ for
every $t\in T\setminus S$. In fact, this and the relative compactness of
$\{f_j(t):j\in\Nb\}$ imply the convergence of $f_j(t)$ as $j\to\infty$ to a point of $M$
denoted by $f(t)$. In other words, $f_j$ converges \pw\ on $T$ to the function
\mbox{$f:T=S\cup(T\setminus S)\to M$}.

Let $t\in T\setminus S$ and $\eta>0$ be arbitrary. Since $\vep_n\to0$ as $n\to\infty$,
choose and fix $n=n(\eta)\in\Nb$ such that $\vep_n\le\eta$. The definition of $S$ implies
that $t$ is a limit point for $T$ and a point of continuity of $\vfi_n$, and so, by the
density of $S$ in $T$, there is $s=s(n,t)\in S$ such that $|\vfi_n(t)-\vfi_n(s)|\le\eta$.
Property \eq{e:fint} yields the existence of $j^1=j^1(\eta,n,t,s)\in\Nb$ such that if
$j,k\ge j^1$, $j\ne k$,
  \begin{equation*}
\bigl|V_{\vep_n}(f_j\!-\!f_k,T\cap(-\infty,\tau])\!-\!\vfi_n(\tau)\bigr|\le\eta
\quad\mbox{for}\quad\mbox{$\tau=t$ \,and \,$\tau=s$.}
  \end{equation*}
Suppose $s<t$ (the case when $s>t$ is treated similarly). Applying Lemma~\ref{l:mor}
(with $T$ replaced by $T\cap(-\infty,t]$, $T_1$---by $T\cap(-\infty,s]$, and
$T_2$---by $T\cap[s,t]$), we get
  \begin{align*}
V_{\vep_n}(f_j\!-\!f_k,T\cap[s,t])&\le V_{\vep_n}(f_j\!-\!f_k,T\cap(-\infty,t])
  -V_{\vep_n}(f_j\!-\!f_k,T\cap(-\infty,s])\\[3pt]
&\le|V_{\vep_n}(f_j\!-\!f_k,T\cap(-\infty,t])\!-\!\vfi_n(t)|+|\vfi_n(t)\!-\!\vfi_n(s)|\\[3pt]
&\qquad+|\vfi_n(s)\!-\!V_{\vep_n}(f_j\!-\!f_k,T\cap(-\infty,s])|\\[3pt]
&\le\eta+\eta+\eta=3\eta\quad\mbox{for all}\quad j,k\ge j^1\,\,\mbox{with}\,\, j\ne k. 
  \end{align*}
Now, given $j,k\ge j^1$, $j\ne k$, by the definition of $V_{\vep_n}(f_j-f_k,T\cap[s,t])$,
there is $g_{j,k}\in\BV(T\cap[s,t];M)$, also depending on $\eta$, $n$, $t$ and $s$,
such that
  \begin{equation*}
\|(f_j-f_k)-g_{j,k}\|_{\infty,\,T\cap[s,t]}\le\vep_n
  \end{equation*}
and
  \begin{equation*}
V(g_{j,k},T\cap[s,t])\le V_{\vep_n}(f_j-f_k,T\cap[s,t])+\eta.
  \end{equation*}
\par\medbreak\noindent
These inequalities and \eq{e:10} imply, for all $j,k\ge j^1$ with $j\ne k$,
  \begin{align*}
\|(f_j\!-\!f_k)(s)\!-\!(f_j\!-\!f_k)(t)\|&\le\|g_{j,k}(s)\!-\!g_{j,k}(t)\|
  +2\|(f_j\!-\!f_k)\!-\!g_{j,k}\|_{\infty,\,T\cap[s,t]}\\[3pt]
&\le V(g_{j,k},T\cap[s,t])+2\vep_n\le(3\eta+\eta)+2\eta=6\eta.
  \end{align*}
Since the sequence $\{f_j(s)\}_{j=1}^\infty$ is convergent in $M$, it is Cauchy, and so,
there is $j^2=j^2(\eta,s)\in\Nb$ such that $\|f_j(s)-f_k(s)\|\le\eta$ for all $j,k\ge j^2$.
It follows that $j^3=\max\{j^1,j^2\}$ depends only on $\eta$ (and $t$), and we have
  \begin{align*}
\|f_j(t)-f_k(t)\|&\le\|(f_j-f_k)(t)-(f_j-f_k)(s)\|+\|(f_j-f_k)(s)\|\\[3pt]
&\le6\eta+\eta=7\eta\quad\mbox{for all}\quad j,k\ge j^3.
  \end{align*}
Thus, $\{f_j(t)\}_{j=1}^\infty$ is a Cauchy sequence in $M$, which completes the proof.
\end{proof}

Various remarks and examples concerning Theorem~\ref{t:SPir} follow after the proof
of Lemma~\ref{l:RT}.

Now we turn to the proof of Lemma~\ref{l:RT}. We need Ramsey's Theorem from
formal logic \cite[Theorem~A]{Ramsey}, which we are going to recall now.

Let $\Gamma$ be a set, $n\in\Nb$, and $\gamma_1,\gamma_2,\dots,\gamma_n$ be
(pairwise) distinct elements of $\Gamma$. The (non-ordered) collection
$\{\gamma_1,\gamma_2,\dots,\gamma_n\}$ is said to be an \emph{$n$-combination\/}
of elements of $\Gamma$ (note that an $n$-combination may be generated by $n!$
different injective functions $\gamma:\{1,2,\dots,n\}\to\Gamma$ with
$\gamma_i=\gamma(i)$ for all $i=1,2,\dots,n$). We denote by $\Gamma[n]$
the family of all $n$-combinations of elements of~$\Gamma$.

\begin{theorem}[Ramsey \cite{Ramsey}] \label{t:Ramsey}
Suppose $\Gamma$ is an infinite set, $n,m\in\Nb$, and
$\Gamma[n]=\bigcup_{i=1}^mG_i$ is a {\sl disjoint} union of $m$ nonempty sets
$G_i\subset\Gamma[n]$. Then, under the Axiom of Choice, there are an infinite
set $\Delta\subset\Gamma$ and $i_0\in\{1,2,\dots,m\}$ such that
$\Delta[n]\subset G_{i_0}$.
\end{theorem}

This theorem will be applied several times in the proof of Lemma~\ref{l:RT} with $\Gamma$
a subset of $\{F_j:j\in\Nb\}$ and $n=m=2$.

The application of Ramsey's Theorem in the context of pointwise selection principles was
initiated by Schrader \cite{Schrader} and later on was extended by several authors
(Di Piazza and Maniscalco \cite{Piazza}, Maniscalco \cite{Manisc}, Chistyakov and
Maniscalco \cite{JMAA08}, Chistyakov, Maniscalco and Tretyachenko \cite{waterman80},
Chistyakov and Tretyachenko \cite{JMAA13}) for real- and metric space-valued functions
of one and several real variables.

\begin{proof}[Proof of Lemma~\protect\ref{l:RT}]
We divide the proof into three steps.

\emph{Step~1.} Let us show that for every $t\in T$ there is a subsequence
$\{F_j^{(t)}\}_{j=1}^\infty$ of $\{F_j\}_{j=1}^\infty$, depending on $t$ and $\vep$,
such that the double limit
  \begin{equation} \label{e:loC}
\lim_{j,k\to\infty}V_\vep(F_j^{(t)}-F_k^{(t)},T\cap(-\infty,t])\quad\mbox{exists in}
\quad [0,C]
  \end{equation}
(clearly, the sequence $\{F_j^{(t)}\}_{j=1}^\infty$ satisfies the uniform estimate
\eq{e:Fjk}).

Given $t\in T$, for the sake brevity, we set $T_t^-=T\cap(-\infty,t]$.
By Lemma~\ref{l:ele}(b) and \eq{e:Fjk}, we have
  \begin{equation*}
0\le V_\vep(F_j-F_k,T_t^-)\le V_\vep(F_j-F_k,T)\le C\quad\mbox{for all}\quad j,k\in\Nb.
  \end{equation*}
In order to apply Theorem~\ref{t:Ramsey}, we set $\Gamma=\{F_j:j\in\Nb\}$,
$c_0=C/2$, and denote by $G_1$ the set of those pairs $\{F_j,F_k\}$ with $j,k\in\Nb$,
$j\ne k$, for which $V_\vep(F_j-F_k,T_t^-)\in[0,c_0)$, and by $G_2$---the set of all
pairs $\{F_j,F_k\}$ with $j,k\in\Nb$, $j\ne k$, such that
$V_\vep(F_j-F_k,T_t^-)\in[c_0,C]$. Clearly, $\Gamma[2]=G_1\cup G_2$ and
$G_1\cap G_2=\es$. If $G_1$ and $G_2$ are both nonempty, then, by
Theorem~\ref{t:Ramsey}, there is a subsequence $\{F_j^1\}_{j=1}^\infty$ of
$\{F_j\}_{j=1}^\infty$ (cf.\ Remark~\ref{r:Fj1}) \label{p:Fj1} such that either
\par(i${}_1$)\, $\{F_j^1,F_k^1\}\in G_1$ for all $j,k\in\Nb$, $j\ne k$, or
\par(ii${}_1$) $\{F_j^1,F_k^1\}\in G_2$ for all $j,k\in\Nb$, $j\ne k$.

In the case when $G_1\ne\es$ and (i${}_1$) holds, or $G_2=\es$, we set
$[a_1,b_1]=[0,c_0]$, while if $G_2\ne \es$ and (ii${}_1$) holds, or $G_1=\es$,
we set $[a_1,b_1]=[c_0,C]$.

Inductively, assume that $p\in\Nb$, $p\ge2$, and a subsequence
$\{F_j^{p-1}\}_{j=1}^\infty$ of $\{F_j\}_{j=1}^\infty$ and an interval
$[a_{p-1},b_{p-1}]\subset[0,C]$ such that
  \begin{equation*}
V_\vep(F_j^{p-1}-F_k^{p-1},T_t^-)\in[a_{p-1},b_{p-1}]\quad\mbox{for all}
\quad j,k\in\Nb,\,\,j\ne k,
  \end{equation*}
are already chosen. To apply Theorem~\ref{t:Ramsey}, we set
$\Gamma=\{F_j^{p-1}:j\in\Nb\}$, define $c_{p-1}=\frac12(a_{p-1}+b_{p-1})$, and
denote by $G_1$ the set of all pairs $\{F_j^{p-1},F_k^{p-1}\}$ with $j,k\in\Nb$, $j\ne k$,
such that $V_\vep(F_j^{p-1}-F_k^{p-1},T_t^-)\in[a_{p-1},c_{p-1})$, and by
$G_2$---the set of all pairs $\{F_j^{p-1},F_k^{p-1}\}$ with $j,k\in\Nb$, $j\ne k$,
for which $V_\vep(F_j^{p-1}-F_k^{p-1},T_t^-)\in[c_{p-1},b_{p-1}]$. We have the union
$\Gamma[2]=G_1\cup G_2$ of disjoint sets. If $G_1$ and $G_2$ are both nonempty,
then, by Ramsey's Theorem, there is a subsequence $\{F_j^p\}_{j=1}^\infty$ of
$\{F_j^{p-1}\}_{j=1}^\infty$ such that either
\par(i${}_p$)\, $\{F_j^p,F_k^p\}\in G_1$ for all $j,k\in\Nb$, $j\ne k$, or
\par(ii${}_p$) $\{F_j^p,F_k^p\}\in G_2$ for all $j,k\in\Nb$, $j\ne k$.
\par\noindent
If $G_1\ne\es$ and (i${}_p$) holds, or $G_2=\es$, we set
$[a_p,b_p]=[a_{p-1},c_{p-1}]$, while if $G_2\ne\es$ and (ii${}_p$) holds, or
$G_1=\es$, we set $[a_p,b_p]=[c_{p-1},b_{p-1}]$.

In this way for each $p\in\Nb$ we have nested intervals
$[a_p,b_p]\subset[a_{p-1},b_{p-1}]$ in $[a_0,b_0]=[0,C]$ with
$b_p-a_p=C/2^p$ and a subsequence $\{F_j^p\}_{j=1}^\infty$ of
$\{F_j^{p-1}\}_{j=1}^\infty$ (where $F_j^0=F_j$, $j\in\Nb$) such that
  \begin{equation*}
V_\vep(F_j^p-F_k^p,T_t^-)\in[a_p,b_p]\quad\mbox{for all}\quad j,k\in\Nb,\,\,j\ne k.
  \end{equation*}
Let $\ell\in[0,C]$ be the common limit of $a_p$ and $b_p$ as $p\to\infty$ (note that
$\ell$ depends on $t$ and $\vep$). Denoting the diagonal sequence
$\{F_j^j\}_{j=1}^\infty$ by $\{F_j^{(t)}\}_{j=1}^\infty$ we infer that the limit
in \eq{e:loC} is equal to~$\ell$. In fact, given $\eta>0$, there is $p(\eta)\in\Nb$
such that $a_{p(\eta)},b_{p(\eta)}\in[\ell-\eta,\ell+\eta]$ and, since
$\{F_j^{(t)}\}_{j=p(\eta)}^\infty$ is a subsequence of
$\{F_j^{p(\eta)}\}_{j=1}^\infty$, we find, for all $j,k\ge p(\eta)$ with $j\ne k$, that
  \begin{equation*}
V_\vep(F_j^{(t)}-F_k^{(t)},T_t^-)\in[a_{p(\eta)},b_{p(\eta)}]\subset[\ell-\eta,\ell+\eta].
  \end{equation*}

\emph{Step~2.} Given a set $A\subset\Rb$, we denote by $\ov A$ its closure in $\Rb$.

Let $Q$ be an at most countable dense subset of $T$ (hence $Q\subset T\subset\ov Q$).
The set $T_L=\{t\in T:\mbox{$T\cap(t-\delta,t)=\es$ for some $\delta>0$}\}$ of
points from $T$, which are isolated from the left for $T$, is at most countable, and
the same is true for the set $T_R=\{t\in T:\mbox{$T\cap(t,t+\delta)=\es$ for some
$\delta>0$}\}$ of points from $T$ isolated from the right for $T$. Clearly,
$T_L\cap T_R\subset Q$, and the set $Z=Q\cup T_L\cup T_R$ is an at most
countable dense subset of~$T$.

We assert that there are a subsequence $\{F_j^*\}_{j=1}^\infty$ of
$\{F_j\}_{j=1}^\infty$ and a nondecreasing function $\vfi:Z\to[0,C]$ (both depending
on $\vep$) such that
  \begin{equation} \label{e:fiz}
\lim_{j,k\to\infty}V_\vep(F_j^*-F_k^*,T\cap(-\infty,s])=\vfi(s)\quad\mbox{for all}
\quad s\in Z.
  \end{equation}
With no loss of generality, we may assume that $Z=\{s_p\}_{p=1}^\infty$. By Step~1,
there are a subsequence $\{F_j^{(s_1)}\}_{j=1}^\infty$ of $\{F_j\}_{j=1}^\infty$,
denoted by $\{F_j^{(1)}\}_{j=1}^\infty$, and a number from $[0,C]$, denoted by
$\vfi(s_1)$, such that
  \begin{equation*}
\lim_{j,k\to\infty}V_\vep(F_j^{(1)}-F_k^{(1)},T\cap(-\infty,s_1])=\vfi(s_1).
  \end{equation*}
Inductively, if $p\in\Nb$, $p\ge2$, and a subsequence $\{F_j^{(p-1)}\}_{j=1}^\infty$ of
$\{F_j\}_{j=1}^\infty$ is already chosen, we apply Step~1 once again to pick a
subsequence $\{F_j^{(p)}\}_{j=1}^\infty$ of $\{F_j^{(p-1)}\}_{j=1}^\infty$ and a
number $\vfi(s_p)\in[0,C]$ such that
  \begin{equation*}
\lim_{j,k\to\infty}V_\vep(F_j^{(p)}-F_k^{(p)},T\cap(-\infty,s_p])=\vfi(s_p).
  \end{equation*}
Denoting by $\{F_j^*\}_{j=1}^\infty$ the diagonal subsequence $\{F_j^{(j)}\}_{j=1}^%
\infty$of $\{F_j\}_{j=1}^\infty$, we establish \eq{e:fiz}. It remains to note that, by
Lemma~\ref{l:ele}(b), the function $\vfi:Z\!\to\![0,C]$, defined by the left-hand side of
\eq{e:fiz}, is nondecreasing~on~$Z$.

\emph{Step~3.} In this step, we finish the proof of \eq{e:fek}. Applying Saks' idea
\cite[Chapter~7, Section~4, Lemma~(4.1)]{Saks}, we extend the function $\vfi$,
defined by \eq{e:fiz}, from the set $Z$ to the whole $\Rb$ as follows: given $t\in\Rb$,
  \begin{equation*}
\wt\vfi(t)=\sup\{\vfi(s):s\in Z\cap(-\infty,t]\}\quad\mbox{if}\quad Z\cap(-\infty,t]\ne\es
  \end{equation*}
and
  \begin{equation*}
\wt\vfi(t)=\inf\{\vfi(s):s\in Z\}\quad\mbox{otherwise.}
  \end{equation*}
Clearly, $\wt\vfi:\Rb\to[0,\infty)$ is nondecreasing and $\wt\vfi(\Rb)\subset\ov{\vfi(Z)}%
\subset[0,C]$. Therefore, the set $D\subset\Rb$ of points of discontinuity of $\wt\vfi$
is at most countable.

Let us show that if $\{F_j^*\}_{j=1}^\infty$ is the sequence from \eq{e:fiz}, then
  \begin{equation} \label{e:tiw}
\lim_{j,k\to\infty}V_\vep(F_j^*-F_k^*,T\cap(-\infty,t])=\wt\vfi(t)\quad
\mbox{for all}\quad t\in T\setminus D.
  \end{equation}

By virtue of \eq{e:fiz}, we may assume that $t\in T\setminus(D\cup Z)$. Let $\eta>0$ be
fixed. Since $t$ is a point of continuity of $\wt\vfi$, there is $\delta=%
\delta(\eta)>0$ such that
  \begin{equation} \label{e:stw}
\mbox{$\wt\vfi(s)\in[\wt\vfi(t)-\eta,\wt\vfi(t)+\eta]$ \,for all \,$s\in\Rb$ \,such that
\,$|s-t|\le\delta$.}
  \end{equation}
Since $T\subset\ov Z$ and $t\notin T_L$, we find $\ov Z\cap(t-\delta,t)\supset%
T\cap(t-\delta,t)\ne\es$, and so, there is $s'\in Z$ with $t-\delta<s'<t$.
By \eq{e:fiz}, there is $j^1=j^1(\eta)\in\Nb$ such that, for all $j,k\ge j^1$, $j\ne k$,
  \begin{equation} \label{e:juke}
\mbox{$V_\vep(F_j^*-F_k^*,T\cap(-\infty,s'])\in[\vfi(s')-\eta,\vfi(s')+\eta]$.}
  \end{equation}
Similarly, $t\notin T_R$ implies the existence of $s''\in Z$ with $t<s''<t+\delta$, and so,
by \eq{e:fiz}, for some $j^2=j^2(\eta)\in\Nb$, we have, for all $j,k\ge j^2$, $j\ne k$,
  \begin{equation} \label{e:keju}
\mbox{$V_\vep(F_j^*-F_k^*,T\cap(-\infty,s''])\in[\vfi(s'')-\eta,\vfi(s'')+\eta]$.}
  \end{equation}
Since $s'<t<s''$, $T\cap(-\infty,s']\subset T\cap(-\infty,t]\subset T\cap(-\infty,s'']$,
and so, by Lemma~\ref{l:ele}(b), we get, for all $j,k\in\Nb$,
  \begin{equation*}
V_\vep(F_j^*-F_k^*,T\cap(-\infty,s'])\!\le\! V_\vep(F_j^*-F_k^*,T\cap(-\infty,t])
\!\le\! V_\vep(F_j^*-F_k^*,T\cap(-\infty,s'']).
  \end{equation*}
Setting $j^3=\max\{j^1,j^2\}$ and noting that $\wt\vfi(s')=\vfi(s')$ and
$\wt\vfi(s'')=\vfi(s'')$, we find, from \eq{e:juke}, \eq{e:keju} and \eq{e:stw}, that
  \begin{align*}
V_\vep(F_j^*\!-\!F_k^*,T\!\cap\!(-\infty,t])&\!\in\!
\bigl[V_\vep(F_j^*\!-\!F_k^*,T\!\cap\!(-\infty,s']),
V_\vep(F_j^*\!-\!F_k^*,T\!\cap\!(-\infty,s''])\bigr]\\[3pt]
&\!\subset\![\vfi(s')-\eta,\vfi(s'')+\eta]=[\wt\vfi(s')-\eta,\wt\vfi(s'')+\eta]\\[3pt]
&\!\subset\![\wt\vfi(t)-2\eta,\wt\vfi(t)+2\eta]\quad\mbox{for all}
  \quad j,k\ge j^3,\,\,j\ne k,
  \end{align*}
which proves \eq{e:tiw}.

Finally, we note that $T=(T\setminus D)\cup(T\cap D)$ where $T\cap D$ is at most
countable. Furthermore, being a subsequence of the original sequence
$\{F_j\}_{j=1}^\infty$, the sequence $\{F_j^*\}_{j=1}^\infty$ from \eq{e:fiz} and
\eq{e:tiw} satisfies the uniform estimate \eq{e:Fjk}. So, arguing as in Step~2 with $Z$
replaced by $T\cap D$, we obtain a subsequence of $\{F_j^*\}_{j=1}^\infty$, denoted
by $\{F_j^\vep\}_{j=1}^\infty$, and a nondecreasing function $\psi:T\cap D\to[0,C]$
such that
  \begin{equation} \label{e:psizh}
\lim_{j,k\to\infty}V_\vep(F_j^\vep-F_k^\vep,T\cap(-\infty,t])=\psi(t)\quad
\mbox{for all}\quad t\in T\cap D.
  \end{equation}
We define the desired function $\vfi^\vep:T\to[0,C]$ by $\vfi^\vep(t)=\wt\vfi(t)$ if
$t\in T\setminus D$ and $\vfi^\vep(t)=\psi(t)$ if $t\in T\cap D$. Now, it follows from
\eq{e:tiw} and \eq{e:psizh} that equality \eq{e:fek} holds, where, in view of
Lemma~\ref{l:ele}(b), the function $\vfi^\vep$ is nondecreasing on~$T$.

This completes the proof of Lemma~\ref{l:RT}.
\end{proof}

\begin{remark} \label{r:Fj1} \rm
Here we present more details on the existence of the subsequence
$\{F_j^1\}_{j=1}^\infty$ of $\{F_j\}_{j=1}^\infty$ after the first application of
Ramsey's Theorem (cf.~p.~\pageref{p:Fj1}). By Theorem~\ref{t:Ramsey},
there is an infinite set $\Delta\subset\Gamma=\{F_j:j\in\Nb\}$ such that either
$\Delta[2]\subset G_1$ or $\Delta[2]\subset G_2$. We infer that
  \begin{equation} \label{e:Delt}
\mbox{$\Delta\!=\!\{F_{q(n)}\!:n\!\in\!\Nb\}$ for some strictly increasing
sequence $q:\Nb\!\to\!\Nb$,}
  \end{equation}
and, setting $F_j^1\!=\!F_{q(j)}$ for $j\!\in\!\Nb$, we have
$\Delta[2]\!=\!\bigl\{\{F_j^1,F_k^1\}:j,k\!\in\!\Nb,\,j\!\ne\! k\bigr\}$.

Since the set $\Nb$ of natural numbers is well-ordered (i.e., every nonempty subset
of $\Nb$ has the minimal element), the sequence $q:\Nb\to\Nb$ can be defined as follows:
$q(1)=\min\{j\in\Nb:F_j\in\Delta\}$, and, inductively, if $n\in\Nb$, $n\ge2$, and
natural numbers $q(1)\!<\!q(2)\!<\!\dots\!<\!q(n-1)$ are already defined, we~set
  \begin{equation} \label{e:qn}
q(n)=\min\bigl\{j\in\Nb\setminus\{q(1),q(2),\dots,q(n\!-\!1)\}:F_j\in\Delta\bigr\}.
  \end{equation}
The sequence $q$ is strictly increasing: if $n\in\Nb$ and
$j\in\Nb\setminus\{q(1),q(2),\dots,q(n)\}$
is such that $F_j\in\Delta$, then $j\ne q(n)$, and since
$j\in\Nb\setminus\{q(1),q(2),\dots,q(n\!-\!1)\}$,
we have, by \eq{e:qn}, $j\ge q(n)$, i.e., $j>q(n)$; by the arbitrariness of $j$ as
above and \eq{e:qn} (for $n+1$ in place of $n$), we get $q(n+1)>q(n)$.
Clearly, $q(n)\ge n$.

Let us verify the equality in \eq{e:Delt}. The inclusion ($\supset$) is clear from
\eq{e:qn}. To see that inclusion ($\subset$) holds, let $F\in\Delta$, so that
$\Delta\subset\Gamma$ implies $F=F_{j_0}$ for some $j_0\in\Nb$. We have
$q(j_0)\ge j_0$, and since $F_{j_0}\in\Delta$, $j_0\ge q(1)$. Hence
$q(1)\le j_0\le q(j_0)$. We claim that there is $1\le n_0\le j_0$ such that $q(n_0)=j_0$
(this implies $F=F_{j_0}=F_{q(n_0)}\in\{F_{q(n)}:n\in\Nb\}$ and establishes~($\subset$)).
By contradiction, if $q(n)\ne j_0$ for all $n=1,2,\dots,j_0$, then $j_0$ belongs to the set
$\{j\in\Nb\setminus\{q(1),q(2),\dots,q(j_0)\}:F_j\in\Delta\}$, and so, by \eq{e:qn},
$q(j_0+1)\le j_0$, which contradicts $q(j_0+1)>q(j_0)\ge j_0$.
\end{remark}
 
\begin{remark}
If $(M,\|\cdot\|)$ is a \emph{finite-dimensional\/} normed linear space, the condition of
relative compactness of sets $\{f_j(t):j\in\Nb\}$ at all points $t\in T$ in Theorem~%
\ref{t:SPir} can be lightened to the condition $\sup_{j\in\Nb}\|f_j(t_0)\|\equiv C_0%
<\infty$ for some $t_0\in T$. In fact, by Lemma~\ref{l:71}(b) and \eq{e:ejkC} with
fixed $\vep_0>0$ and $j_0\equiv j_0(\vep_0)$, we get
  \begin{equation*}
|(f_j-f_{j_0})(T)|\le V_{\vep_0}(f_j-f_{j_0},T)+2\vep_0\le C(\vep_0)+2\vep_0\,\,\,
\mbox{for all}\,\,\,j\ge j_0.
  \end{equation*}
Hence, given $t\in T$, we find
  \begin{align*}
\|f_j(t)\|&\le\|(f_j-f_{j_0})(t)-(f_j-f_{j_0})(t_0)\|+\|f_{j_0}(t)\|+\|f_j(t_0)\|
  +\|f_{j_0}(t_0)\|\\[3pt]
&\le(C(\vep_0)+2\vep_0)+\|f_{j_0}(t)\|+2C_0\quad\mbox{for all}\quad j\ge j_0,
  \end{align*}
and so, the set $\{f_j(t):j\in\Nb\}$ is \rc\ in $M$.
\end{remark}

\begin{remark}
Under the assumptions on $T$ and $M$ from Theorem~\ref{t:SPir}, if a sequence
$\{f_j\}\subset M^T$ converges \emph{uniformly\/} on $T$ to a function $f\in M^T$,
then
  \begin{equation} \label{e:fjik}
\lim_{j,k\to\infty}V_\vep(f_j-f_k,T)=0\quad\mbox{for all}\quad\vep>0,
  \end{equation}
i.e., condition \eq{e:spir} is \emph{necessary}. In fact, given $\vep>0$, there is
$j_0=j_0(\vep)\in\Nb$ such that $\|f_j-f_k\|_{\infty,T}\le\vep$ for all $j,k\ge j_0(\vep)$.
Since the zero function $0$ on $T$ is constant, we get $V_\vep(f_j-f_k,T)\le V(0,T)=0$
for all $j,k\ge j_0(\vep)$.
\end{remark}

\begin{remark}
In Example \ref{ex:ntns}, we show that condition \eq{e:spir} is \emph{not necessary\/}
for the pointwise convergence of $\{f_j\}$ to $f$. However, it is `almost necessary' in
the following sense (cf.\ Remark~\ref{r:neces}(b)). Let $T\subset\Rb$ be a measurable
set with \emph{finite\/} Lebesgue measure $\mathcal{L}(T)$ and $\{f_j\}\subset M^T$
be a sequence of measurable functions which converges \pw\ or almost everywhere
on $T$ to a function $f\in M^T$. Egorov's Theorem implies that for every $\eta>0$
there is a measurable set $T_\eta\subset T$ such that $\mathcal{L}(T\setminus T_\eta)%
\le\eta$ and $f_j\rra f$ on $T_\eta$. By \eq{e:fjik}, we get
  \begin{equation*}
\lim_{j,k\to\infty}V_\vep(f_j-f_k,T_\eta)=0\quad\mbox{for all}\quad\vep>0.
  \end{equation*}
\end{remark}

Applying Theorem~\ref{t:SPir} and the diagonal procedure we get the following

\begin{theorem} \label{t:SPirvar}
Under the assumptions of Theorem\/~{\rm\ref{t:SPir}}, if a sequence of functions
$\{f_j\}\subset M^T$ is such that, for all $\vep>0$,
  \begin{equation*}
\mbox{$\D\limsup_{j,k\to\infty}V_\vep(f_j-f_k,T\setminus E)<\infty$ for an at most
countable $E\subset T$}
  \end{equation*}
or
  \begin{equation*}
\mbox{$\D\limsup_{j,k\to\infty}V_\vep(f_j-f_k,T\cap[a,b])<\infty$ \,for all \,$a,b\in T$,
$a\le b$,}
  \end{equation*}
then $\{f_j\}$ contains a subsequence which converges pointwise on $T$.
\end{theorem}

\begin{example} \label{ex:ntns}
Condition \eq{e:spir} is \emph{not necessary\/} for the pointwise convergence even if
all functions in the sequence $\{f_j\}$ are regulated. To see this, let $\{f_j\}\subset M^T$
be the sequence from Example~\ref{ex:ucbw}, where $T=I=[0,1]$ and $(M,\|\cdot\|)$
is a normed linear space. First, note that
  \begin{equation} \label{e:xixx}
\limsup_{j,k\to\infty}V_\vep(f_j-f_k,T)\ge\limsup_{j\to\infty}V_\vep(f_j-f_{j+1},T)
\quad\mbox{for all}\quad\vep>0.
  \end{equation}
Let us fix $j\in\Nb$ and set $t_k=k/(j+1)!$, $k=0,1,\dots,(j+1)!$, so that
$f_{j+1}(t_k)=x$ for all such~$k$. We have $f_j(t_k)=x$ if and only if $j!t_k$ is an
integer, i.e., $k=n(j+1)$ with $n=0,1,\dots,j!$. It follows that $(f_j-f_{j+1})(t)=y-x$
if $t=t_k$ for those $k\in\{0,1,\dots,(j+1)!\}$, for which $k\ne n(j+1)$ for all
$n\in\{0,1,\dots,j!\}$ (and, in particular, for $k=1,2,\dots,j$); in the remaining cases
of $t\in T$ we have $(f_j-f_{j+1})(t)=0$. If $s_k=\frac12(t_{k-1}+t_k)=%
(k-\frac12)/(j+1)!$, $k=1,2,\dots,(j+1)!$, we get a partition of the interval
$T=[0,1]$ of the form
  \begin{equation*}
0=t_0<s_1<t_1<s_2<t_2<\dots<s_{(j+1)!}<t_{(j+1)!}=1,
  \end{equation*}
and $f_j(s_k)=f_{j+1}(s_k)=y$ for all $k=1,2,\dots,j$. Now, let
$0<\vep<\frac12\|x-y\|$, and a function $g\in M^T$ be arbitrary such that
$\|(f_j-f_{j+1})-g\|_{\infty,T}\le\vep$. By \eq{e:10}, we find
  \begin{align*}
V(g,T)&\ge\sum_{k=1}^{(j+1)!}\|g(t_k)\!-\!g(s_k)\|\ge\sum_{k=1}^j\bigl(
  \|(f_j\!-\!f_{j+1})(t_k)\!-\!(f_j\!-\!f_{j+1})(s_k)\|\!-\!2\vep\bigr)\\
&=(\|y-x\|-2\vep)j,
  \end{align*}
and so, by \eq{e:av}, $V_\vep(f_j-f_{j+1},T)\ge(\|y-x\|-2\vep)j$. Hence,
\eq{e:xixx} implies
  \begin{equation*}
\limsup_{j,k\to\infty}V_\vep(f_j-f_k,T)=\infty\quad\mbox{for \,all}\quad
0<\vep<\textstyle\frac12\|x-y\|.
  \end{equation*}
\end{example}

\begin{example} \label{ex:nrg}
Under the assumptions of Theorem~\ref{t:SPir} we cannot infer that the limit function
$f$ of an extracted subsequence of $\{f_j\}$ is a \emph{regulated\/} function
(this is the reason to term this theorem an \emph{irregular\/} selection principle).

Let $T=[a,b]$, $(M,\|\cdot\|)$ be a normed linear space, $x,y\in M$, $x\ne y$,~and
$\al_j=1+(1/j)$, $j\in\Nb$ (cf.\ Example~\ref{ex:rieq}). The sequence of Dirichlet
functions $f_j=\al_j\Dc_{x,y}=\Dc_{\al_jx,\al_jy}$, $j\in\Nb$, converges
\emph{uniformly\/} on $T$ to the Dirichlet function $f=\Dc_{x,y}$, which is
non-regulated. By virtue of \eq{e:fjik}, Theorem~\ref{t:SPir} can be applied to the
sequence $\{f_j\}$. On the other hand, Example~\ref{ex:rieq} shows that $\{f_j\}$ does
not satisfy condition \eq{e:sp}, and so, Theorem~\ref{t:SP} is inapplicable.

Sometimes it is more appropriate to apply Theorem~\ref{t:SPir} in the form of
Theorem~\ref{t:SPirvar}. Let $\{\beta_j\}_{j=1}^\infty\subset\Rb$ be a bounded
sequence (not necessarily convergent). Formally, Theorem~\ref{t:SPir} cannot be
applied to the sequence $f_j=\beta_j\Dc_{x,y}$, $j\in\Nb$, on $T=[a,b]$ (e.g., with
$\beta_j=(-1)^j$ or $\beta_j=(-1)^j+(1/j)$). However, note that for every $j\in\Nb$
the restriction of $f_j$ to the set $T\setminus\Qb$ is the constant function
$c(t)\equiv\beta_jy$ on $T\setminus\Qb$, whence $V_\vep(f_j-f_k,T\setminus\Qb)=0$
for all $\vep>0$. Hence Theorem~\ref{t:SPirvar} is applicable to $\{f_j\}$.\hfill$\square$
\end{example}

More examples, which can be adapted to the situation under consideration, can be found
in \cite[Section~4]{JMAA08}.

The  following theorem is a counterpart of Theorem~\ref{t:SPae}.

\begin{theorem}
Let $T\subset\Rb$, $(M,\|\cdot\|)$ be a normed linear space and $\{f_j\}\subset M^T$
be a \pw\ \rc\ (or a.e.\ \rc) on $T$ sequence of functions satisfying the condition\/{\rm:}
for every $p\in\Nb$ there is a measurable set $E_p\subset T$ with Lebesgue measure
$\mathcal{L}(E_p)\le1/p$ such that
  \begin{equation*}
\limsup_{j,k\to\infty}V_\vep(f_j-f_k,T\setminus E_p)<\infty\quad\mbox{for \,all}
\quad\vep>0.
  \end{equation*}
Then $\{f_j\}$ contains a subsequence which converges almost everywhere on~$T$.
\end{theorem}

Finally, we present an extension of Theorem~\ref{t:SPir} in the spirit of Theorems
\ref{t:SPweak} and \ref{t:SPw2}.

\begin{theorem}
Let $T\subset\Rb$ and $(M,\|\cdot\|)$ be a reflexive Banach space with separable dual
$(M^*,\|\cdot\|)$. Suppose the sequence of functions $\{f_j\}\subset M^T$ is such that
  \begin{itemize}
\renewcommand{\itemsep}{0.0pt plus 0.5pt minus 0.25pt}
\item[{\rm(i)}] $\sup_{j\in\Nb}\|f_j(t_0)\|\le C_0$ for some $t_0\in T$ and $C_0\ge0;$
\item[{\rm(ii)}] $\limsup_{j,k\to\infty}V_\vep(\lan(f_j-f_k)(\cdot),x^*\ran,T)<\infty$
  for all\/ $\vep>0$ and $x^*\in M^*$.
  \end{itemize}
Then, there is a subsequence of $\{f_j\}$, again denoted by $\{f_j\}$, and a function
$f\in M^T$ such that $f_j(t)\wto f(t)$ in $M$ as $j\to\infty$ for all $t\in T$.
\end{theorem}

\end{document}